\documentclass{article}

\usepackage{amsmath}
\usepackage{amsthm}
\usepackage{amssymb}
\usepackage{amscd}
\usepackage{xspace}
\usepackage{verbatim}
\newtheorem{theorem}{Theorem}[section]
\newtheorem{corollary}{Corollary}[section]
\newtheorem{lemma}{Lemma}[section]
\newtheorem{proposition}{Proposition}[section]
\theoremstyle{definition}
\newtheorem{definition}{Definition}[section]

\theoremstyle{remark}
\newtheorem{remark}{Remark}[section]
\numberwithin{equation}{section}

\newcommand{\e}{\varepsilon}

\renewcommand{\O}{\Omega}

\renewcommand{\liminf}{\varliminf}
\renewcommand{\limsup}{\varlimsup}

\newcommand{\field}[1]{\mathbb{#1}}

\newcommand{\R}{\field{R}}

\newcommand{\er}{\eqref}

\DeclareMathOperator{\Div}{div}

\DeclareMathOperator{\supp}{supp}

\renewcommand{\O}{\Omega}

\date{}

\begin{document}
\title{Variational resolution for some general classes of nonlinear evolutions. Part I}
\maketitle
\begin{center}
\textsc{Arkady Poliakovsky \footnote{E-mail:
poliakov@math.bgu.ac.il}
}\\[3mm]
Department of Mathematics, Ben Gurion University of the Negev,\\
P.O.B. 653, Be'er Sheva 84105, Israel
\\[2mm]
\end{center}
\begin{abstract}
We develop a variational technique for some wide classes of
nonlinear evolutions. The novelty here is that we derive the main
information directly from the corresponding Euler-Lagrange
equations. In particular, we  prove that not only the minimizer of
the appropriate energy functional but also any critical point must
be a solution of the corresponding evolutional system.
\end{abstract}
\section{Introduction}
Let $X$ be a reflexive Banach space. Consider the following
evolutional initial value problem:
\begin{equation}\label{abstpr}
\begin{cases}
\frac{d}{dt}\big\{I\cdot
u(t)\big\}+\Lambda_t\big(u(t)\big)=0\quad\quad\text{in}\;\;(0,T_0),
\\
I\cdot u(0)=v_0.
\end{cases}
\end{equation}
Here $I:X\to X^*$ ($X^*$ is the space dual to $X$) is a fixed
bounded linear inclusion operator, which we assume to be
self-adjoint and strictly positive, $u(t)\in L^q\big((0,T_0);X\big)$
is an unknown function, such that $I\cdot u(t)\in
W^{1,p}\big((0,T_0);X^*\big)$ (where $I\cdot h\in X^*$ is the value
of the operator $I$ at the point $h\in X$), $\Lambda_t(x):X\to X^*$
is a fixed nonlinear mapping, considered for every fixed
$t\in(0,T_0)$, and $v_0\in X^*$ is a fixed initial value. The most
trivial variational principle related to \er{abstpr} is the
following one. Consider some convex function
$\Gamma(y):X^*\to[0,+\infty)$, such that
$\Gamma(y)=0$ if and only if $y=0$. Next define the following energy
functional
\begin{multline}\label{abstprhfen}
E_0\big(u(\cdot)\big):=\int_0^{T_0}\Gamma\bigg(\frac{d}{dt}\big\{I\cdot
u(t)\big\}+\Lambda_t\big(u(t)\big)\bigg)dt\\ \forall\, u(t)\in
L^q\big((0,T_0);X\big)\;\;\text{s.t.}\;\;I\cdot u(t)\in
W^{1,p}\big((0,T_0);X^*\big)\;\;\text{and}\;\;I\cdot u(0)=v_0\,.
\end{multline}
Then it is obvious that $u(t)$ will be a solution to \er{abstpr} if
and only if $E_0\big( u(\cdot)\big)=0$. Moreover, the solution to
\er{abstpr} will exist if and only if there exists a minimizer
$u_0(t)$ of the energy $E_0(\cdot)$, which satisfies $E_0\big(
u_0(\cdot)\big)=0$.

 We have the following generalization of this variational principle.
Let $\Psi_t(x):X\to[0,+\infty)$ be some convex G\^{a}teaux
differentiable function, considered for every fixed $t\in(0,T_0)$
and such that $\Psi_t(0)=0$. Next define the Legendre transform of
$\Psi_t$ by
\begin{equation}\label{vjhgkjghkjghjk}
\Psi^*_t(y):=\sup\Big\{\big<z,y\big>_{X\times X^*}-\Psi_t(z):\;z\in
X\Big\}\quad\quad\forall y\in X^*\,.
\end{equation}
It is well known that $\Psi^*_t(y):X^*\to\R$ is a convex function
and
\begin{equation}\label{vhjfgvhjgjkgjkh}
\Psi_t(x)+\Psi^*_t(y)\;\geq\; \big<x,y\big>_{X\times
X^*}\quad\quad\forall\, x\in X,\,y\in X^*\,,
\end{equation}
with equality if and only if $y=D\Psi_t(x)$. Next for
$\lambda\in\{0,1\}$ define the energy
\begin{multline}\label{abstprhfenbjvfj}
E_\lambda\big(u\big):=\int\limits_0^{T_0}\Bigg\{\Psi_t\Big(\lambda
u(t)\Big)+\Psi^*_t\bigg(-\frac{d}{dt}\big\{I\cdot
u(t)\big\}-\Lambda_t\big(u(t)\big)\bigg)+\lambda\bigg<u(t),\frac{d}{dt}\big\{I\cdot
u(t)\big\}+\Lambda_t\big(u(t)\big)\bigg>_{X\times X^*}\Bigg\}dt\\
\forall\, u(t)\in L^q\big((0,T_0);X\big)\;\;\text{s.t.}\;\;I\cdot
u(t)\in W^{1,p}\big((0,T_0);X^*\big)\;\;\text{and}\;\;I\cdot
u(0)=v_0.
\end{multline}
Then, by \er{vhjfgvhjgjkgjkh} we have $E_\lambda\big(\cdot\big)\geq
0$ and moreover, $E_\lambda\big(u(\cdot)\big)=0$ if and only if
$u(t)$ is a solution to
\begin{equation}\label{abstprrrrcn}
\begin{cases}
\frac{d}{dt}\big\{I\cdot
u(t)\big\}+\Lambda_t\big(u(t)\big)+D\Psi_t\big(\lambda
u(t)\big)=0\quad\quad\text{in}\;\;(0,T_0),
\\
I\cdot u(0)=v_0
\end{cases}
\end{equation}
(note here that since $\Psi_t(0)=0$, in the case $\lambda=0$
\er{abstprrrrcn} coincides with \er{abstpr}. Moreover, if
$\lambda=0$ then the energy defined in \er{abstprhfen} is a
particular case of the energy in \er{abstprhfenbjvfj}, where we take
$\Gamma(x):=\Psi^*(-x)\,$). So, as before, a solution to
\er{abstprrrrcn} exists if and only if there exists a minimizer
$u_0(t)$ of the energy $E_\lambda(\cdot)$, which satisfies
$E_\lambda\big( u_0(\cdot)\big)=0$. Consequently, in order to
establish the existence of solution to \er{abstprrrrcn} we need to
answer the following questions:
\begin{itemize}
\item [{\bf (a)}] Does a minimizer to the energy
in \er{abstprhfenbjvfj} exist?
\item [{\bf (b)}] Does the minimizer $u_0(t)$ of the corresponding
energy $E_\lambda(\cdot)$ satisfies
$E_\lambda\big(u_0(\cdot)\big)=0$?
\end{itemize}

 To the best of our knowledge, the energy in \er{abstprhfenbjvfj} with
$\lambda=1$, related to \er{abstprrrrcn}, was first considered for
the heat equation and other types of evolutions by Brezis and
Ekeland in \cite{Brez}. In that work they also first asked question
{\bf (b)}: If we don't know a priori that a solution of the equation
\er{abstprrrrcn} exists, how to prove that the minimum of the
corresponding energy is zero. This question was asked even for very
simple PDE's like the heat equation. A detailed investigation of the
energy of type \er{abstprhfenbjvfj}, with $\lambda=1$, was done in a
series of works of N. Ghoussoub and his coauthors, see the book
\cite{NG}  and also \cite{Gho}, \cite{GosM}, \cite{GosM1},
\cite{GosTz}. In these works they considered a similar variational
principle, not only for evolutions but also for some other classes
of equations. They proved  some theoretical results about  general
self-dual variational principles, which in many cases, can provide
with the existence of a zero energy state (answering questions {\bf
(a)}+{\bf (b)} together) and, consequently, with the existence of
solution for the related equations (see \cite{NG} for details).

 In this work we provide  an alternative approach to the
questions {\bf (a)} and {\bf (b)}. We treat them separately and in
particular, for question {\bf (b)}, we derive the main information
by studying the Euler-Lagrange equations for the corresponding
energy. To our knowledge, such an approach was first considered in
\cite{P4} and provided there an alternative proof of existence of
solution for initial value problems for some parabolic systems.
Generalizing these results, we provide here the answer to  questions
{\bf (a)} and {\bf (b)} for some wide classes of evolutions. In
particular, regarding question {\bf (b)}, we are able to prove that
in some general cases not only the minimizer but also any critical
point $u_0(t)$ (i.e. any solution of corresponding Euler-Lagrange
equation) satisfies $E_\lambda\big(u_0(\cdot)\big)=0$, i.e. is a
solution to \er{abstprrrrcn}.

The approach of Ghoussoub in \cite{NG} is more general than ours as
he considered a more abstract setting. The main advantages of our
method
are:
\begin{itemize}
\item
We prove that under some growth and coercivity conditions
\underline{every critical point} of the energy \er{abstprhfenbjvfj}
is actually a minimizer and a solution of \er{abstprrrrcn}.

\item
Our result, giving the answer for question {\bf (b)}, doesn't
require any assumption of compactness or weak continuity of
$\Lambda_t$ (these assumptions are needed only for the proof of
existence of minimizer, i.e., in connection with question {\bf
(a)}).

\item
Our method for answering question {\bf (b)} uses only elementary
arguments.
\end{itemize}

We can rewrite the definition of $E_\lambda$ in \er{abstprhfenbjvfj}
as follows. Since $I$ is a self-adjoint and strictly positive
operator,
there exists a Hilbert space $H$ and an injective
bounded linear operator $T:X\to H$, whose image is
dense in $H$, such that if we consider the linear operator
$\widetilde{T}:H\to X^*$, defined by the formula
\begin{equation}\label{tildetjbghgjgklhgjkgkgkjjkjkl}
\big<x,\widetilde{T}\cdot y\big>_{X\times X^*}:=\big<T\cdot
x,y\big>_{H\times H}\quad\quad\text{for every}\; y\in
H\;\text{and}\;x\in X\,,
\end{equation}
then we will have
$\widetilde{T}\circ T\equiv I$, see Lemma \ref{hdfghdiogdiofg} for
details. We call $\{X,H,X^*\}$ an evolution triple with the
corresponding inclusion operator $T:X\to H$ and $\widetilde{T}:H\to
X^*$.
Thus, if $v_0=\widetilde{T}\cdot w_0$, for some $w_0\in H$ and
$p=q^*:=q/(q-1)$, where $q>1$, then we have
$$\int_0^{T_0}\bigg<u(t),\frac{d}{dt}\big\{I\cdot
u(t)\big\}\bigg>_{X\times X^*}dt=\frac{1}{2}\big\|T\cdot
u(T_0)\big\|^2_H-\frac{1}{2}\big\|w_0\big\|^2_H$$ (see Lemma
\ref{lem2} for details) and therefore,
\begin{multline}\label{abstprhfenbjvfjghgighjkjg}
E_\lambda\big(u\big)=J\big(u\big):=\\
\int\limits_0^{T_0}\Bigg\{\Psi_t\Big(\lambda
u(t)\Big)+\Psi^*_t\bigg(-\frac{d}{dt}\big\{I\cdot
u(t)\big\}-\Lambda_t\big(u(t)\big)\bigg)+\lambda\Big<u(t),\Lambda_t\big(u(t)\big)\Big>_{X\times
X^*}\Bigg\}dt+\frac{\lambda}{2}\big\|T\cdot
u(T_0)\big\|^2_H-\frac{\lambda}{2}\big\|w_0\big\|^2_H\\
\forall\, u(t)\in L^q\big((0,T_0);X\big)\;\;\text{s.t.}\;\;I\cdot
u(t)\in W^{1,q^*}\big((0,T_0);X^*\big)\;\;\text{and}\;\;I\cdot
u(0)=\widetilde{T}\cdot w_0
\end{multline}

Our first main result provides the answer for question {\bf
(b)}, under some coercivity and growth conditions on $\Psi_t$ and
$\Lambda_t$ (see an equivalent formulation in Theorem
\ref{EulerLagrange} and Proposition \ref{gftufuybhjkojutydrdkk}):
\begin{theorem}\label{EulerLagrangeInt}
Let $\{X,H,X^*\}$ be an evolution triple with the corresponding
inclusion linear operators $T:X\to H$, which we assume to be
bounded, injective and having dense image in $H$,
$\widetilde{T}:H\to X^*$ be defined by
\er{tildetjbghgjgklhgjkgkgkjjkjkl} and $I:=\widetilde{T}\circ T:X\to
X^*$. Next let $\lambda\in\{0,1\}$, $q\geq 2$, $p=q^*:=q/(q-1)$ and
$w_0\in H$. Furthermore, for every $t\in[0,T_0]$ let
$\Psi_t(x):X\to[0,+\infty)$ be a strictly convex function which is
G\^{a}teaux differentiable at every $x\in X$, satisfying
$\Psi_t(0)=0$ and the condition
\begin{equation}\label{roststrrr}
(1/C_0)\,\|x\|_X^q-C_0\leq \Psi_t(x)\leq
C_0\,\|x\|_X^q+C_0\quad\forall x\in X,\;\forall t\in[0,T_0]\,,
\end{equation}
for some $C_0>0$. We also assume that $\Psi_t(x)$ is a Borel
function of its variables $(x,t)$.
Next, for every $t\in[0,T_0]$ let $\Lambda_t(x):X\to X^*$ be a
function which is G\^{a}teaux differentiable at every $x\in X$, s.t.
$\Lambda_t(0)\in L^{q^*}\big((0,T_0);X^*\big)$ and the derivative of
$\Lambda_t$ satisfies the growth condition
\begin{equation}\label{roststlambdrrr}
\|D\Lambda_t(x)\|_{\mathcal{L}(X;X^*)}\leq g\big(\|T\cdot
x\|_H\big)\,\Big(\|x\|_X^{q-2}+\mu^{\frac{q-2}{q}}(t)\Big)\quad\forall
x\in X,\;\forall t\in[0,T_0]\,,
\end{equation}
for some non-decreasing function $g(s):[0+\infty)\to (0,+\infty)$
and some nonnegative function $\mu(t)\in L^1\big((0,T_0);\R\big)$.
We also assume that $\Lambda_t(x)$ is strongly Borel on the pair of
variables $(x,t)$ (see Definition
\ref{fdfjlkjjkkkkkllllkkkjjjhhhkkk}).
Assume also that $\Psi_t$ and $\Lambda_t$ satisfy the following
monotonicity condition
\begin{multline}\label{Monotonerrr}
\bigg<h,\lambda \Big\{D\Psi_t\big(\lambda x+h\big)-D\Psi_t(\lambda
x)\Big\}+D\Lambda_t(x)\cdot h\bigg>_{X\times X^*}\geq
-\hat g\big(\|T\cdot x\|_H\big)\Big(\|x\|_X^q+\hat
\mu(t)\Big)\,\|T\cdot h\|^{2}_H
\\ \forall x,h\in X,\;\forall t\in[0,T_0]\,,
\end{multline}
for some non-decreasing function $\hat g(s):[0+\infty)\to
(0,+\infty)$ and some nonnegative function $\hat\mu(t)\in
L^1\big((0,T_0);\R\big)$.
Consider the set
\begin{equation}\label{hgffckaaq1newrrrvhjhjhm}
\mathcal{R}_{q}:=\Big\{u(t)\in L^q\big((0,T_0);X\big):\;I\cdot
u(t)\in W^{1,q^*}\big((0,T_0);X^*\big)\Big\}\,,
\end{equation}
and the minimization problem
\begin{equation}\label{hgffckaaq1newrrr}
\inf\Big\{J(u):\,u(t)\in \mathcal{R}_{q}\;\;\text{s.t}\;\;I\cdot
u(0)=\widetilde{T}\cdot w_0\Big\}\,,
\end{equation}
where $J(u)$ is defined by \er{abstprhfenbjvfjghgighjkjg}. Then for
every $u\in\mathcal{R}_{q}$ such that $I\cdot
u(0)=\widetilde{T}\cdot w_0$ and for arbitrary function
$h(t)\in\mathcal{R}_{q}$, such that $I\cdot h(0)=0$, the finite
limit $\lim\limits_{s\to 0}\big(J(u+s h)-J(u)\big)/s$ exists.
Moreover, for every such $u$ the following four statements
are equivalent:
\begin{itemize}
\item[{\bf (1)}]
$u$ is a critical point of \er{hgffckaaq1newrrr}, i.e., for any
function $h(t)\in\mathcal{R}_{q}$, such that $I\cdot h(0)=0$ we have
\begin{equation}\label{nolkjrrr}
\lim\limits_{s\to 0}\frac{J(u+s h)-J(u)}{s}=0\,.
\end{equation}
\item[{\bf(2)}]
$u$ is a minimizer to \er{hgffckaaq1newrrr}.
\item[{\bf (3)}] $J(u)=0$.
\item[{\bf (4)}]
$u$ is a solution to
\begin{equation}\label{abstprrrrcnjkghjkh}
\begin{cases}
\frac{d}{dt}\big\{I\cdot
u(t)\big\}+\Lambda_t\big(u(t)\big)+D\Psi_t\big(\lambda
u(t)\big)=0\quad\quad\text{in}\;\;(0,T_0),
\\
I\cdot u(0)=\widetilde{T}\cdot w_0.
\end{cases}
\end{equation}
\end{itemize}
Finally, there exists at most one function $u\in\mathcal{R}_{q}$
which satisfies \er{abstprrrrcnjkghjkh}.
\end{theorem}
\begin{remark}\label{gdfgdghfjkllkhoiklj}
Assume that, instead of \er{Monotonerrr}, one requires that $\Psi_t$
and $\Lambda_t$ satisfy the following inequality
\begin{multline}\label{Monotone1111gkjglghjgj}
\bigg<h,\lambda \Big\{D\Psi_t\big(\lambda x+h\big)-D\Psi_t(\lambda
x)\Big\}+D\Lambda_t(x)\cdot h\bigg>_{X\times X^*}\geq
\\ \frac{\big|f(h,t)\big|^2}{\tilde g(\|T\cdot
x\|_H)}-\tilde g\big(\|T\cdot
x\|_H\big)\Big(\|x\|_X^q+\hat\mu(t)\Big)^{(2-r)/2}\big|f(h,t)\big|^r\,\|T\cdot
h\|^{(2-r)}_H\quad\forall x,h\in X,\;\forall t\in[0,T_0],
\end{multline}
for some non-decreasing function $\tilde g(s):[0+\infty)\to
(0,+\infty)$, some function $\hat\mu(t)\in L^1\big((0,T_0);\R\big)$,
some function $f(x,t):X\times[0,T_0]\to\R$ and some constant
$r\in(0,2)$.
Then \er{Monotonerrr}
follows by the trivial inequality $(r/2)\,a^2+\big((2-r)/2\big)\,
b^2\geq a^r \,b^{2-r}$.
\end{remark}

Our first result about the existence of minimizer for $J(u)$ is the
following Proposition (see Proposition \ref{premainnew} for an
equivalent formulation):
\begin{proposition}\label{premainnewEx}
Assume that $\{X,H,X^*\}$, $T,\widetilde{T},I$, $\lambda,q,p$,
$\Psi_t$ and $\Lambda_t$ satisfy all the conditions of Theorem
\ref{EulerLagrangeInt} together with the assumption $\lambda=1$.
Moreover, assume that
$\Psi_t$ and $\Lambda_t$ satisfy the following positivity condition
\begin{multline}\label{MonotonegnewhghghghghEx}
\Psi_t(x)+\Big<x,\Lambda_t(x)\Big>_{X\times X^*}\geq \frac{1}{\tilde
C}\,\|x\|^q_X -
\bar\mu(t)\Big(\|T\cdot x\|^{2}_H+1\Big)
\quad\forall x\in X,\;\forall
t\in[0,T_0],
\end{multline}
where
$\tilde C>0$ is some constant and $\bar\mu(t)\in
L^1\big((0,T_0);\R\big)$ is some nonnegative function.
Furthermore, assume that
\begin{equation}\label{jgkjgklhklhj}
\Lambda_t(x)=A_t\big(S\cdot x\big)+\Theta_t(x)\quad\quad\forall\,
x\in X,\;\forall\, t\in[0,T_0],
\end{equation}
where $Z$ is a Banach space, $S:X\to Z$ is a compact operator and
for every $t\in[0,T_0]$ $A_t(z):Z\to X^*$ is a function which is
strongly Borel on the pair of variables $(z,t)$ and G\^{a}teaux
differentiable at every $z\in Z$, $\Theta_t(x):X\to X^*$ is strongly
Borel on the pair of variables $(x,t)$ and G\^{a}teaux
differentiable at every $x\in X$, $\,\Theta_t(0),A_t(0)\in
L^{q^*}\big((0,T_0);X^*\big)$ and the derivatives of $A_t$ and
$\Theta_t$ satisfy the growth condition
\begin{equation}\label{roststlambdgnewjjjjjEx}
\|D\Theta_t(x)\|_{\mathcal{L}(X;X^*)}+\|D A_t(S\cdot
x)\|_{\mathcal{L}(Z;X^*)}\leq g\big(\|T\cdot
x\|\big)\,\Big(\|x\|_X^{q-2}+\mu^{\frac{q-2}{q}}(t)\Big)\quad\forall
x\in X,\;\forall t\in[0,T_0]
\end{equation}
for some nondecreasing function $g(s):[0,+\infty)\to (0+\infty)$ and
some nonnegative function $\mu(t)\in L^1\big((0,T_0);\R\big)$. Next
assume that for every sequence
$\big\{x_n(t)\big\}_{n=1}^{+\infty}\subset L^q\big((0,T_0);X\big)$
such that the sequence $\big\{I\cdot x_n(t)\big\}$ is bounded in
$W^{1,q^*}\big((0,T_0);X^*\big)$ and $x_n(t)\rightharpoonup x(t)$
weakly in $L^q\big((0,T_0);X\big)$ we have
\begin{itemize}
\item
$\Theta_t\big(x_n(t)\big)\rightharpoonup \Theta_t\big(x(t)\big)$
weakly in $L^{q^*}\big((0,T_0);X^*\big)$,
\item
$\liminf_{n\to+\infty}\int_{0}^{T_0}
\Big<x_n(t),\Theta_t\big(x_n(t)\big)\Big>_{X\times X^*}dt\geq
\int_{0}^{T_0}\Big<x(t),\Theta_t\big(x(t)\big)\Big>_{X\times
X^*}dt$.
\end{itemize}
Finally, let $w_0\in H$ be such that $w_0=T\cdot u_0$ for some
$u_0\in X$, or more generally, $w_0\in H$ be such that
$\mathcal{A}_{w_0}:=\big\{u\in\mathcal{R}_{q}:\,I\cdot
u(0)=\widetilde{T}\cdot w_0\big\}\neq\emptyset$. Then there exists a
minimizer to \er{hgffckaaq1newrrr}.
\end{proposition}
As a consequence of Theorem \ref{EulerLagrangeInt} and Proposition
\ref{premainnewEx} we have the following Corollary.
\begin{corollary}\label{CorCorCor}
Assume that we are in the settings of Proposition
\ref{premainnewEx}. Then there exists a unique solution
$u(t)\in\mathcal{R}_{q}$ to
\begin{equation}\label{abstprrrrcnjkghjkhggg}
\begin{cases}
\frac{d}{dt}\big\{I\cdot
u(t)\big\}+\Lambda_t\big(u(t)\big)+D\Psi_t\big(
u(t)\big)=0\quad\quad\text{in}\;\;(0,T_0),
\\
I\cdot u(0)=\widetilde{T}\cdot w_0.
\end{cases}
\end{equation}
\end{corollary}
As an important particular case of Corollary \ref{CorCorCor} we have
following statement:
\begin{theorem}\label{THSldInt}
Let $\{X,H,X^*\}$ be an evolution triple with the corresponding
inclusion linear operators $T:X\to H$, which we assume to be
bounded, injective and having dense image in $H$,
$\widetilde{T}:H\to X^*$ be defined by
\er{tildetjbghgjgklhgjkgkgkjjkjkl} and $I:=\widetilde{T}\circ T:X\to
X^*$. Next let $q\geq 2$. Furthermore, for every $t\in[0,T_0]$ let
$\Psi_t(x):X\to[0,+\infty)$ be a strictly convex function which is
G\^{a}teaux differentiable at every $x\in X$, satisfies
$\Psi_t(0)=0$ and satisfies the growth condition
\begin{equation}\label{roststSLDInt}
(1/C_0)\,\|x\|_X^q-C_0\leq \Psi_t(x)\leq
C_0\,\|x\|_X^q+C_0\quad\forall x\in X,\;\forall t\in[0,T_0]\,,
\end{equation}
and the following uniform convexity condition
\begin{equation}\label{roststghhh77889lkagagSLDInt}
\Big<h,D\Psi_t(x+h)-D\Psi_t(x)\Big>_{X\times X^*}\geq
\frac{1}{C_0}\Big(\big\|x\big\|^{q-2}_X+1
\Big)\cdot\|h\|_X^2\quad\forall
x,h\in X,\;\,\forall t\in[0,T_0],
\end{equation}
for some $C_0>0$.
We also assume that $\Psi_t(x)$ is Borel on the pair of variables
$(x,t)$ Next let $Z$ be a Banach space, $S:X\to Z$ be a compact
operator and for every $t\in[0,T_0]$ let $F_t(z):Z\to X^*$ be a
function, such that $F_t$ is strongly Borel on the pair of variables
$(z,t)$ and G\^{a}teaux differentiable at every $z\in Z$, $F_t(0)\in
L^{q^*}\big((0,T_0);X^*\big)$ and the derivatives of $F_t$ satisfies
the growth conditions
\begin{equation}\label{roststlambdgnewSLDInt}
\big\|D F_t(S\cdot x)\big\|_{\mathcal{L}(Z;X^*)}\leq g\big(\|T\cdot
x\|\big)\,\Big(\|x\|_X^{q-2}+1
\Big)\quad\forall
x\in X,\;\forall t\in[0,T_0]\,,
\end{equation}
for some non-decreasing function $g(s):[0+\infty)\to (0,+\infty)$.
Moreover, assume that
$\Psi_t$ and $F_t$ satisfy the following positivity condition:
\begin{multline}\label{MonotonegnewhghghghghSLDInt}
\Psi_t(x)+\Big<x,F_t(S\cdot x)\Big>_{X\times X^*}\geq \frac{1}{\bar
C}\,\|x\|^q_X -\bar C
\|S\cdot x\|^{2}_Z-
\bar\mu(t)
\Big(\|T\cdot x\|^{2}_H+1\Big)
\quad\forall x\in X,\;\forall t\in[0,T_0],
\end{multline}
where $\bar C>0$ is some constant and $\bar\mu(t)\in
L^1\big((0,T_0);\R\big)$ is a nonnegative function. Furthermore, let
$w_0\in H$ be such that $w_0=T\cdot u_0$ for some $u_0\in X$, or
more generally, $w_0\in H$ be such that
$\mathcal{A}_{w_0}:=\big\{u\in\mathcal{R}_{q}:\,I\cdot
u(0)=\widetilde{T}\cdot w_0\big\}\neq\emptyset$. Then there exists a
unique
solution $u(t)\in\mathcal{R}_{q}$ to the following equation
\begin{equation}\label{uravnllllgsnachnewSLDInt}
\begin{cases}\frac{d}{dt}\big\{I\cdot u(t)\big\}+F_t\big(S\cdot u(t)\big)+D\Psi_t\big(u(t)\big)=0\quad\text{for
a.e.}\; t\in(0,T_0)\,,\\ I\cdot u(0)=\widetilde T\cdot
w_0\,.\end{cases}
\end{equation}
\end{theorem}
As a consequence, we have the following easily formulated Corollary:
\begin{corollary}\label{THSldIntgkgiu}
Let $X$ be a reflexive Banach space and $X^*$ be the space dual to
$X$. Next let $I:X\to X^*$ be a self-adjoint and strictly positive
bounded linear operator. Furthermore, for every $t\in[0,T_0]$ let
$\Psi_t(x):X\to[0,+\infty)$ be a convex function which is
G\^{a}teaux differentiable at every $x\in X$, satisfies
$\Psi_t(0)=0$ and satisfies the growth condition
\begin{equation}\label{roststSLDIntcor}
(1/C_0)\,\|x\|_X^2-C_0\leq \Psi_t(x)\leq
C_0\,\|x\|_X^2+C_0\quad\forall x\in X,\;\forall t\in[0,T_0]\,,
\end{equation}
and the uniform convexity condition
\begin{equation}\label{roststghhh77889lkagagSLDIntcor}
\Big<h,D\Psi_t(x+h)-D\Psi_t(x)\Big>_{X\times X^*}\geq
\frac{1}{C_0}\,\|h\|_X^2\quad\forall x,h\in X,\;\,\forall
t\in[0,T_0],
\end{equation}
for some $C_0>0$.
We also assume that $\Psi_t(x)$ is Borel on the pair of variables
$(x,t)$ Next let $Z$ be a Banach space, $S:X\to Z$ be a compact
operator and for every $t\in[0,T_0]$ let $F_t(z):Z\to X^*$ be a
function, such that $F_t$ is strongly Borel on the pair of variables
$(z,t)$ and G\^{a}teaux differentiable at every $z\in Z$, $F_t(0)\in
L^{2}\big((0,T_0);X^*\big)$ and the derivative of $F_t$ satisfies
the Lipschitz's condition
\begin{equation}\label{roststlambdgnewSLDIntcor}
\big\|D F_t(z)\big\|_{\mathcal{L}(Z;X^*)}\leq C\quad\forall z\in
Z,\;\forall t\in[0,T_0]\,,
\end{equation}
for some constant $C>0$. Then for every $u_0\in X$ there exists a
unique $u(t)\in L^2\big((0,T_0);X\big)$ such that $I\cdot u(t)\in
W^{1,2}\big((0,T_0);X^*\big)$ and $u(t)$ is a
solution to
\begin{equation}\label{uravnllllgsnachnewSLDIntcor}
\begin{cases}\frac{d}{dt}\big\{I\cdot u(t)\big\}+F_t\big(S\cdot u(t)\big)+D\Psi_t\big(u(t)\big)=0\quad\text{for
a.e.}\; t\in(0,T_0)\,,\\ I\cdot u(0)=I\cdot u_0\,.\end{cases}
\end{equation}
\end{corollary}

In \cite{PII}, using Theorem \ref{THSldInt} as a basis, by the
appropriate approximation, we obtain further existence Theorems,
under much weaker assumption on coercivity and compactness.
Moreover, applying these general Theorems, we provide with the
existence results for various classes of time dependent partial
differential equations including parabolic, hyperbolic,
Shr\"{o}dinger and Navier-Stokes systems.

In order  to demonstrate the basic idea of the proof of the key
Theorem \ref{EulerLagrangeInt} consider the simple example of the
scalar parabolic equation of the following form:
\begin{equation}\label{cyfuggbjhghgjh}
\begin{cases}
\partial_t u+\Div_x F(u)-\Delta_x u=0\quad\quad\forall
x\in\O\subset\subset\R^N,\;\forall
t\in(0,T_0)\\
u(x,t)=0\quad\quad\forall x\in\partial\O,\;\quad\forall t\in(0,T_0)
\\
u(x,0)=v_0(x)\quad\quad\forall x\in\O,
\end{cases}
\end{equation}
where we assume $F:\R\to\R^N$ to be smooth and globally Lipschitz
function. In this case we take $X:=W^{1,2}_0(\O)$, $H:=L^2(\O)$ and
$T$ to be a trivial embedding. Then $X^*=W^{-1,2}(\O)$ and the
energy-functional takes the form
\begin{equation}\label{energht}
\bar E(u):=\frac{1}{2}\int_0^{T_0}\int_\O\Bigg(|\nabla_x
u|^2+\bigg|\nabla_x\Big\{\Delta_x^{-1}\big(\partial_t u+\Div_x F(u)
\big)\Big\}\bigg|^2\Bigg)\,dxdt+\frac{1}{2}\int_\O\Big(\big|u(x,T_0)\big|^2-\big|u(x,0)\big|^2\Big)dx,
\end{equation}
where $\Delta^{-1}f$ is the solution of
\begin{equation*}
\begin{cases}\Delta y=f\quad\quad x\in\O\,,\\ y=0\quad\quad\forall
x\in\partial\O\,.\end{cases}
\end{equation*}
Let us investigate the Euler-Lagrange equation for \er{energht}. If
$u$ satisfies $u(x,t)=0$ for every $(x,t)\in\partial\O\times(0,T_0)$
and $u(x,0)=v_0(x)$, then,
\begin{equation}\label{fhgrtfhifhj}
\bar E(u):=\frac{1}{2}\int_0^{T_0}\int_\O\Bigg(\bigg|\nabla_x\Big\{
u-\Delta_x^{-1}\big(\partial_t u+\Div_x
F(u)\big)\Big\}\bigg|^2\Bigg)\,dxdt.
\end{equation}
Set $W_u:=u-\Delta_x^{-1}\big(\partial_t u+\Div_x F(u)\big)$. Thus,
for every minimizer $u$ of the energy \er{fhgrtfhifhj} and for every
smooth test function $\delta(x,t)$, satisfying $\delta(x,t)=0$ for
every $(x,t)\in\partial\O\times(0,T_0)$ and $\delta(x,0)\equiv 0$,
we obtain
\begin{multline*}
0=\frac{d \bar E(u+s\delta)}{ds}\Big|_{(s=0)}=\lim\limits_{s\to
0}\frac{1}{2s}\int_0^{T_0}\int_\O\Big(|\nabla_x
W_{(u+s\delta)}|^2-|\nabla_x W_{u}|^2\Big)=\\
-\lim\limits_{s\to 0}\frac{1}{2s}\int_0^{T_0}\int_\O\big(\Delta_x
W_{(u+s\delta)}-\Delta_x W_{u}\big)\cdot\big(
W_{(u+s\delta)}+W_{u}\big)=\\
\lim\limits_{s\to 0}\frac{1}{2}\int_0^{T_0}\int_\O\Big(-\Delta_x
\delta+\partial_t\delta+\Div_x\big(F(u+s\delta)-F(u)\big)/s\Big)\cdot\big(
W_{(u+s\delta)}+W_{u}\big)=\\ \int_0^{T_0}\int_\O\Big(\nabla
W_{u}\cdot\nabla_x
\delta+W_{u}\cdot\partial_t\delta-\big(F'(u)\cdot\nabla_x
W_{u}\big)\delta\Big)\,.
\end{multline*}
Since $\delta$ was arbitrary (in particular $\delta(x,T_0)$ is free)
we deduce that $\Delta_x W_u+\partial_t W_u+F'(u)\cdot\nabla_x
W_u=0$, $W_u(x,T_0)=0$ and $W_u=0$ if $x\in\partial\O$. Changing
variables $\tau:=T_0-t$ gives that $W_u$ is a solution of the
following linear parabolic equation with the trivial initial and
boundary conditions:
\begin{equation*}
\begin{cases}
\partial_\tau W_u-F'(u)\cdot\nabla_x
W_u=\Delta_x W_u\quad\quad\forall (x,\tau)\in\O\times(0,T_0)\,,\\
W_u(x,0)=0\quad\quad\quad\forall x\in\O\,,\\
W_u(x,\tau)=0\quad\quad\quad\forall
(x,\tau)\in\partial\O\times(0,T_0)\,.
\end{cases}
\end{equation*}
Therefore $W_u=0$ and then $\Delta_x u=\partial_t u+\Div_x F(u)$,
i.e., $u$ is the solution of \er{cyfuggbjhghgjh}.

 In section \ref{dkgfkghfhkljl} we give one more example, applying
our result to a general parabolic system.

\section{Notations and preliminaries}
Throughout the paper by linear space we mean a real linear space.
\begin{itemize}
\item For given normed space $X$ we denote by $X^*$ the dual space (the space of continuous (bounded) linear functionals from $X$ to $\R$).
\item For given $h\in X$ and $x^*\in X^*$ we denote by $\big<h,x^*\big>_{X\times X^*}$ the value in $\R$ of the functional $x^*$ on the vector $h$.
\item For given two normed linear spaces $X$ and $Y$ we denote by $\mathcal{L}(X;Y)$ the linear space of continuous (bounded) linear operators from $X$ to $Y$.
\item For given $A\in\mathcal{L}(X;Y)$  and $h\in X$ we denote by $A\cdot h\in Y$ the value of the operator $A$ at the point $h$.
\item We set
$\|A\|_{\mathcal{L}(X;Y)}=\sup\{\|A\cdot h\|_Y:\;h\in
X,\;\|h\|_X\leq 1\}$. Then it is well known that $\mathcal{L}(X;Y)$
will be a normed linear space. Moreover $\mathcal{L}(X;Y)$ will be a
Banach space if $Y$ is a Banach space.
\end{itemize}
\begin{definition}\label{2bdf}
Let $X$ and $Y$ be two normed linear spaces. We say that a function
$F:X\to Y$ is G\^{a}teaux differentiable at the point $x\in X$ if
there exists $A\in\mathcal{L}(X;Y)$ such that the following limit
exists in $Y$ and satisfy,
$$\lim\limits_{s\to 0}\frac{1}{s}\Big(F(x+sh)-F(x)\Big)=A\cdot h\quad\forall h\in X\,.$$
In this case we denote the operator $A$ by $DF(x)$ and the value $A\cdot h$ by $DF(x)\cdot h$.
\end{definition}

\begin{definition}\label{fdfjlkjjkkkkkllllkkkjjjhhhkkk}
Let $X$ and $Y$ be two normed linear spaces and $U\subset X$ be a
Borel subset. We say that the mapping $F(x):U\to Y$ is strongly
Borel if the following two conditions are satisfied.
\begin{itemize}
\item
$F$ is a Borel mapping i.e. for every Borel set $W\subset Y$, the
set $\{x\in U:\,F(x)\in W\}$ is also Borel.
\item For every separable subspace $X'\subset X$, the set $\{y\in Y:\,y=F(x),\; x\in U\cap
X'\}$ is also contained in some separable subspace of $Y$.
\end{itemize}
\end{definition}

\begin{definition}\label{3bdf}
For a given Banach space $X$ with the associated norm $\|\cdot\|_X$
and a real interval $(a,b)$ we denote by $L^q(a,b;X)$ the linear
space of (equivalence classes of) strongly measurable (i.e
equivalent to some strongly Borel mapping)
functions $f:(a,b)\to
X$ such that the functional
\begin{equation*}
\|f\|_{L^q(a,b;X)}:=
\begin{cases}
\Big(\int_a^b\|f(t)\|^q_X dt\Big)^{1/q}\quad\text{if }\;1\leq
q<\infty\\
{\text{es$\,$sup}}_{t\in (a,b)}\|f(t)\|_X\quad\text{if }\; q=\infty
\end{cases}
\end{equation*}
is finite. It is known that this functional defines a norm with
respect to which $L^q(a,b;X)$ becomes a Banach space. Moreover, if
$X$ is reflexive and $1<q<\infty$ then $L^q(a,b;X)$ will be a reflexive
space with the corresponding dual space $L^{q^*}(a,b;X^*)$,
where $q^*=q/(q-1)$.
It is also well known that the subspace of continuous functions $C^0([a,b];X)\subset L^q(a,b;X)$ is dense i.e. for every $f(t)\in L^q(a,b;X)$ there exists a sequence
$\{f_n(t)\}\subset C^0([a,b];X)$ such that $f_n(t)\to f(t)$ in the strong topology of $L^q(a,b;X)$.
\end{definition}
We will need the following simple
well known fact:
\begin{lemma}\label{lebesguepoints}
Let $X$ be a Banach space, $(a,b)$ be a bounded real interval and $f(t)\in L^q(a,b;X)$ for some $1\leq q<+\infty$. Then
if we denote
\begin{equation}\label{fbarddd77788}
\bar f(t):=\begin{cases}
f(t)\quad\text{if}\;\;t\in(a,b)\,,\\
0\quad\text{if}\;\;t\notin(a,b)\,,
\end{cases}
\end{equation}
then
\begin{equation}\label{fbardddjjjkk77788}
\lim\limits_{h\to 0}\int_\R\big\|\bar f(t+h)-\bar f(t)\big\|_X^q dt=0\,,
\end{equation}
and
\begin{equation}\label{fbardddjjjkkfg777888sddf88888hgthtyj9999}
\lim\limits_{h\to 0}\int_\R\bigg(\frac{1}{h}\int_{-h}^{h}\big\|\bar f(t+\tau)-\bar f(t)\big\|_X^q d\tau\bigg) dt
=0\,.
\end{equation}
Moreover, for every sequence $\e_n\to 0^+$ as $n\to +\infty$, up to a subsequence, still denoted by $\e_n$ we have
\begin{equation}\label{fbardddjjjkkfg777888}
\lim\limits_{n\to +\infty}\frac{1}{\e_n}\int_{-\e_n}^{\e_n}\big\|\bar f(t+\tau)-\bar f(t)\big\|_X^q d\tau=
\lim\limits_{n\to +\infty}\int_{-1}^{1}\big\|\bar f(t+\e_n s)-\bar f(t)\big\|_X^q ds=0\quad\text{for a.e.}\;\;t\in\R\,.
\end{equation}
\end{lemma}

\begin{definition}\label{4bdf}
Let $X$ be a reflexive Banach space and let $(a,b)$ be a finite real
interval. We say that $v(t)\in L^q(a,b;X)$ belongs to
$W^{1,q}(a,b;X)$ if there exists $f(t)\in L^q(a,b;X)$ such that for
every $\delta(t)\in C^1\big((a,b);X^*\big)$ satisfying $\supp
\delta\subset\subset (a,b)$ we have
$$\int\limits_a^b\big<f(t),\delta(t)\big>_{X\times X^*}dt=-\int\limits_a^b\Big<v(t),\frac{d\delta}{dt}(t)\Big>_{X\times X^*}dt\,.$$
In this case we denote $f(t)$ by $v'(t)$ or by $\frac{d v}{dt}(t)$.
It is well known that if $v(t)\in W^{1,1}(a,b;X)$ then $v(t)$ is a
bounded and continuous function on $[a,b]$ (up to a redefining of
$v(t)$ on a subset of $[a,b]$ of Lebesgue measure zero), i.e.
$v(t)\in C^0\big([a,b];X\big)$ and for every $\delta(t)\in
C^1\big([a,b];X^*\big)$ and every subinterval
$[\alpha,\beta]\subset[a,b]$ we have
\begin{equation}\label{sobltr}
\int\limits_\alpha^\beta\bigg\{\Big<\frac{dv}{dt}(t),\delta(t)\Big>_{X\times
X^*}+\Big<v(t),\frac{d\delta}{dt}(t)\Big>_{X\times X^*}\bigg\}dt=
\big<v(\beta),\delta(\beta)\big>_{X\times
X^*}-\big<v(\alpha),\delta(\alpha)\big>_{X\times X^*}\,.
\end{equation}
\end{definition}
We have the following obvious result:
\begin{lemma}\label{vlozhenie}
Let $X$ and $Y$ be two reflexive Banach spaces,
$S\in\mathcal{L}(X,Y)$ be an injective operator (i.e. it satisfies
$\ker S=0$) and $(a,b)$ be a finite real interval. Then if $u(t)\in
L^q(a,b;X)$ is such that $v(t):=S\cdot u(t)\in W^{1,q}(a,b;Y)$ and
there exists $f(t)\in L^q(a,b;X)$ such that $\frac{d
v}{dt}(t)=S\cdot f(t)$ then $u(t)\in W^{1,q}(a,b;X)$ and $\frac{d
u}{dt}(t)=f(t)$.
\end{lemma}
\begin{definition}\label{5bdf}
Let $X$ be a Banach space. We say that a function $\Psi(x):X\to\R$ is convex (strictly convex) if for every $\lambda\in(0,1)$ and for every $x,y\in X$ s.t. $x\neq y$ we have
\begin{equation*}
\Psi\Big(\lambda x+(1-\lambda)y\Big)\;\;\leq\;\big(\,<\,\big)\;\;\;\;\lambda\Psi(x)+(1-\lambda)\Psi(y)\,.
\end{equation*}
It is well known that if $\Psi(x):X\to\R$ is a convex (strictly
convex) function which is G\^{a}teaux differentiable at every $x\in
X$ then for every $x,y\in X$ s.t. $x\neq y$ we have
\begin{equation}\label{conprmin}
\Psi(y)\;\;\geq\;\big(\,>\,\big)\;\;\;\;\Psi(x)+\Big<y-x,D\Psi(x)\Big>_{X\times X^*}\,,
\end{equation}
and
\begin{equation}\label{conprmonsv}
\Big<y-x,D\Psi(y)-D\Psi(x)\Big>_{X\times X^*}\;\;\geq\;\big(\,>\,\big)\;\;\;\;0\,,
\end{equation}
(remember that $D\Psi(x)\in X^*$). Furthermore, $\Psi$ is weakly
lower semicontinuous on $X$. Moreover, if some function
$\Psi(x):X\to\R$ is G\^{a}teaux differentiable at every $x\in X$ and
satisfy either \er{conprmin} or \er{conprmonsv} for every $x,y\in X$
s.t. $x\neq y$, then $\Psi(y)$ is convex (strictly convex).
\end{definition}
\begin{definition}\label{6bdf}
Let $X$ be a reflexive Banach space and let $\Psi(x):X\to\R$ be a
convex function. For every $y\in X^*$ set the Legendre transform of
$\Psi$ by
$$\Psi^*(y):=\sup\Big\{\big<z,y\big>_{X\times X^*}-\Psi(z):\;z\in X\Big\}\,.$$
\end{definition}
\begin{lemma}\label{Legendre}
Let $X$ be a reflexive Banach space and let
$\Psi(x):X\to[0,+\infty)$ be a strictly convex function which is
G\^{a}teaux differentiable at every $x\in X$ and satisfies
$\Psi(0)=0$ and
\begin{equation}\label{rostlll}
\lim_{\|x\|_X\to+\infty}\frac{1}{\|x\|_X}\Psi(x)=+\infty\,.
\end{equation}
Then $\Psi^*(y)$ is a strictly convex function from $X^*$ to
$[0,+\infty)$ and satisfies $\Psi^*(0)=0$.
Furthermore, $\Psi^*(y)$ is G\^{a}teaux differentiable at every
$y\in X^*$. Moreover $x\in X$ satisfies $x=D\Psi^*(y)$ (remember
that $D\Psi^*(y)\in X^{**}=X$) if and only if $y\in X^*$ satisfies
$y=D\Psi(x)$ (remember that $D\Psi(x)\in X^*$). Finally, if in
addition $\Psi$ satisfies
\begin{equation}\label{rost}
(1/C_0)\,\|x\|_X^q-C_0\leq \Psi(x)\leq
C_0\,\|x\|_X^q+C_0\quad\forall x\in X\,,
\end{equation}
for some $q>1$ and $C_0>0$, then
\begin{equation}\label{rost*}
(1/C)\,\|y\|_{X^*}^{q^*}-C\leq \Psi^*(y)\leq
C\,\|y\|_{X^*}^{q^*}+C\quad\forall y\in X^*\,,
\end{equation}for some $C>0$ depending only on $C_0$ and $q$, where
$q^*:=q/(q-1)$. Moreover, for some $\bar C_0,\bar C>0$, that depend
only on $C$ and $q$ from \er{rost}, we have
\begin{equation}\label{rostgrad}
\|D\Psi(x)\|_{X^*}\leq \bar C_0\|x\|_X^{q-1}+\bar C_0\quad\forall
x\in X\,,
\end{equation}
and
\begin{equation}\label{rostgrad*}
\|D\Psi^*(y)\|_{X}\leq \bar C\|y\|_{X^*}^{q^*-1}+\bar C\quad\forall
y\in X^*\,.
\end{equation}
\end{lemma}
\begin{proof}
First since $\Psi(0)=0$ it is clear
that  for
every $y\in X^*$ we have $\Psi^*(y)\geq 0$. Next since for every
$x\in X$ we have $\Psi(x)\geq 0$ then $\Psi^*(0)\leq 0$ and so
$\Psi^*(0)=0$. Next by the growth condition \er{rostlll} we deduce
that $\Psi^*(y)<+\infty$ for every $y\in X^*$. So $\Psi^*(y):X^*\to
[0,+\infty)$. Moreover, it easy follows from the definition of
Legendre transform, that $\Psi^*(y)$ is a convex function on $X^*$.

 Next since $\Psi$ is weakly lower semicontinuous on $X$ and satisfies growth condition \er{rostlll} then for every $y\in X^*$ there exists $z_y\in X$ such that
\begin{equation}\label{minsv}
\Psi(z_y)-\big<z_y,y\big>_{X\times X^*}=\inf\big\{\Psi(z)-<z,y>_{X\times X^*}:\;z\in X\big\}\,,
\end{equation}
i.e.
\begin{equation}\label{Lcny}
\Psi^*(y):=\big<z_y,y\big>_{X\times X^*}-\Psi(z_y)\,.
\end{equation}
Moreover we have
\begin{equation}\label{grdyl8}
D\Psi(z_y)=y\,,
\end{equation}
However, since $\Psi$ is a strictly convex function, by
\er{conprmonsv} for every $z\in X$ s.t. $z\neq z_y$ we must have
\begin{equation}\label{grdylrrr}
\big<z-z_{y},D\Psi(z)-D\Psi(z_y)\big>_{X\times X^*}>0\,.
\end{equation}
Therefore, in particular for every $y\in X^*$ $z=z_y$ is a unique solution of the equation $D\Psi(z)=y$ and
\begin{equation}\label{grdyl}
\big<z_{y_1}-z_{y_2},y_1-y_2\big>_{X\times X^*}>0\,,
\end{equation}
for every $y_1,y_2\in X^*$ s.t. $y_1\neq y_2$.
Next let $y_0,h\in X^*$. Then by the definition of $\Psi^*$ for every $s\in\R$ we have
\begin{equation}\label{grdprivch}
s\,\big<z_{y_0},h\big>_{X\times X^*}\leq\Psi^*(y_0+sh)-\Psi^*(y_0)\leq s\,\big<z_{(y_0+sh)},h\big>_{X\times X^*}\,.
\end{equation}
On the other hand by \er{minsv} we have
$\Psi(z_{(y_0+sh)})\leq\big<z_{y_0+sh},y_0+sh\big>_{X\times X^*}$.
Therefore, using growth condition \er{rostlll} we deduce that there
exists $\bar C>0$ such that $\|z_{y_0+sh}\|_X\leq\bar C$ for every
$s\in(-1,1)$. Thus using the fact that $X$ is reflexive we deduce
that for any sequence $\{s_n\}_{n=1}^{+\infty}\subset(-1,1)$ such
that $\lim_{n\to+\infty}s_n=0$, up to a subsequence, we must have
$z_{y_0+s_nh}\rightharpoonup \tilde z$ weakly in $X$. However, by
\er{minsv} we have
\begin{equation}\label{minsvsh}
\Psi\big(z_{(y_0+s_nh)}\big)-\big<z_{(y_0+s_nh)},y_0+s_n h\big>_{X\times X^*}\leq\Psi(z)-<z,y_0+s_nh>_{X\times X^*}\quad\forall \,z\in X\,.
\end{equation}
Then tending $n\to+\infty$ in \er{minsvsh}, using the fact that, up to a subsequence,
$z_{y_0+s_nh}\rightharpoonup \tilde z$ and that $\Psi$ is weakly lower semicontinuous function we deduce that
\begin{equation}\label{minsvshlim}
\Psi(\tilde z)-<\tilde z,y_0>_{X\times X^*}\leq\Psi(z)-<z,y_0>_{X\times X^*}\quad\forall \,z\in X\,.
\end{equation}
So $\tilde z$ is a minimizer to \er{minsv} with $y=y_0$ and
therefore, $D\Psi(\tilde z)=y_0$. On the other hand $z=z_{y_0}$ is a
unique solution of the equation $D\Psi(z)=y_0$. Therefore, $\tilde
z=z_{y_0}$. So by \er{grdprivch}, up to a subsequence, we have
\begin{equation}\label{grdprivchlim}
\frac{1}{s_n}\Big(\Psi^*(y_0+s_n h)-\Psi^*(y_0)\Big)\to\big<z_{y_0},h\big>_{X\times X^*}\,.
\end{equation}
Since the sequence $s_n$ was chosen arbitrary we deduce that
\begin{equation}\label{grdprivchlimpoln}
\lim\limits_{s\to 0}\frac{1}{s}\Big(\Psi^*(y_0+s h)-\Psi^*(y_0)\Big)\to\big<z_{y_0},h\big>_{X\times X^*}\,.
\end{equation}
Finally, $y_0,h\in X$ also were chosen arbitrary and therefore we
deduce that $\Psi^*(y)$ is G\^{a}teaux differentiable at every $y\in
X^*$ and $D\Psi^*(y)=z_y$. Thus since $z=z_{y}$ is a unique solution
of the equation $D\Psi(z)=y$ we deduce that $D\Psi^*(y)=z$ if and
only if $D\Psi(z)=y$. Moreover by \er{grdyl} we deduce that
\begin{equation}\label{grdylmon}
\big<D\Psi^*(y_1)-D\Psi^*(y_2),y_1-y_2\big>_{X\times X^*}>0\,,
\end{equation}
for every $y_1,y_2\in X^*$ such that $y_1\neq y_2$. So $\Psi^*$ is a
strictly convex on $X^*$ function.

Next if we consider function $\zeta(y):X^*\to\R$ defined by
$$\zeta(y):=\sup\big\{<z,y>_{X\times X^*}-k\|z\|_X^q:\;z\in X\big\}\,,$$
for some $k>0$, then
\begin{multline*}
\zeta(y)=\sup\big\{t<z,y>_{X\times X^*}-k|t|^q:\;t\in\R,\;z\in X,\,\|z\|_X=1\big\}=\\
\sup\Big\{K|<z,y>_{X\times X^*}|^{q^*}:\;z\in
X,\,\|z\|_X=1\Big\}=K\|y\|^{q^*}_{X^*}\,,
\end{multline*}
for some $K>0$ depending only on $k$ and $q$. Thus using growth
condition \er{rost} and the definition of $\Psi^*$ we easily deduce
growth condition \er{rost*}. So it remains to prove that growth
condition \er{rostgrad} follows from growth condition \er{rost}  and
\er{rostgrad*} follows from  \er{rost*}.
 Indeed since $\Psi$ is convex, from \er{conprmin}, for every $x,h\in X$ we have
\begin{equation}\label{conprminsled}
\big<h,D\Psi(x)\big>_{X\times X^*}\leq \Psi(x+h)-\Psi(x)\,.
\end{equation}
Therefore, for every $x,h\in X$ such that $\|h\|_X\leq 1$ and $\|x\|_X\geq 1$ we have
\begin{equation}\label{conprminsledggg}
\big<h,D\Psi(x)\big>_{X\times X^*}\leq \frac{1}{\|x\|_X}\Big(\Psi\big(x+\|x\|_Xh\big)-\Psi(x)\Big)\,.
\end{equation}
Thus using  growth condition \er{rost} we deduce that for every $x,h\in X$ such that $\|h\|_X\leq 1$ and $\|x\|_X\geq 1$ we have
\begin{equation}\label{conprminsledgggddd8}
\big<h,D\Psi(x)\big>_{X\times X^*}\leq \tilde C\|x\|^{q-1}_X\,,
\end{equation}
where $\tilde C>0$ is a constant, depending on $C_0$ and $q$ only,
and so
\begin{equation}\label{conprminsledgggddd}
\|D\Psi(x)\|_{X^*}\leq \tilde C\|x\|^{q-1}_X\,,
\end{equation}
for every $x$ which satisfy $\|x\|_X\geq 1$. However, by
\er{conprminsled} and \er{rost} we have
\begin{equation}\label{conprminsleddddfergrh}
\big<h,D\Psi(x)\big>_{X\times X^*}\leq \hat C\,,
\end{equation}
for every $x,h\in X$ such that $\|x\|_X\leq 1$ and $\|h\|_X\leq 1$,
where $\hat C>0$ is a constant depending on $C_0$ and $q$ only. So
$\|D\Psi(x)\|_{X^*}\leq \hat C$ for every $x$ which satisfy
$\|x\|_X\leq 1$. This together with \er{conprminsledgggddd} gives
the desired result \er{rostgrad}. Finally, $\Psi^*$ is a convex on
$X^*$ and satisfy \er{rost*}. Therefore, \er{rostgrad*} follows
exactly by the same way.
\end{proof}

\begin{definition}\label{hfguigiugyuyhkjjhlkklkk}
Let $Z$ be a
Banach space and $Z^*$ be a corresponding dual space. We say that
the mapping $\Lambda(z):Z\to Z^*$ is monotone (strictly monotone) if
we have
\begin{equation}
\label{ftguhhhhihggjgjkjggkjgj}
\Big<y-z,\Lambda(y)-\Lambda(z)\Big>_{Z\times Z^*}\,\geq\,(>)\,
0\quad\forall\, y\neq z\in Z\,.
\end{equation}
\end{definition}
%
%
%
%
%
%
%
\begin{definition}\label{hfguigiugyuyhkjjh}
Let $Z$ be a
Banach space and $Z^*$ be a corresponding dual space. We say that
the mapping $\Lambda(z):Z\to Z^*$ is pseudo-monotone if for every
sequence $\{z_n\}_{n=1}^{+\infty}\subset Z$, satisfying
\begin{equation}
\label{ftguhhhhikk} z_n\rightharpoonup z\;\;\text{weakly
in}\;\;Z\quad\quad\text{and}\quad\quad\limsup_{n\to+\infty}\Big<z_n-z,\Lambda(z_n)\Big>_{Z\times
Z^*}\leq 0
\end{equation}
we have
\begin{equation}
\label{ftguhhhhihggjgjk}
\liminf_{n\to+\infty}\Big<z_n-y,\Lambda(z_n)\Big>_{Z\times
Z^*}\geq\Big<z-y,\Lambda(z)\Big>_{Z\times Z^*}\quad\forall y\in Z\,.
\end{equation}
\end{definition}
\begin{lemma}\label{hhhhhhhhhhhhhhhhhhiogfydtdtyd}
Let $Z$ be a
Banach space and $Z^*$ be a corresponding dual space. Then the
mapping $\Lambda(z):Z\to Z^*$ is pseudo-monotone if and only if it
satisfies the following conditions:
\begin{itemize}
\item[{\bf(i)}] For every sequence $\{z_n\}_{n=1}^{+\infty}\subset Z$, such
that $z_n\rightharpoonup z$ weakly in $Z$ we have
\begin{equation}
\label{ftguhhhhikkjhjhjkjkkkkkkk}
\liminf_{n\to+\infty}\Big<z_n-z,\Lambda(z_n)\Big>_{Z\times Z^*}\geq
0\,.
\end{equation}
\item[{\bf(ii)}] If for some
sequence $\{z_n\}_{n=1}^{+\infty}\subset Z$, such that
$z_n\rightharpoonup z$ weakly in $Z$ we have
\begin{equation}
\label{ftguhhhhikkjhjhjkjkkkkkkkjjjjhgghhh}
\lim_{n\to+\infty}\Big<z_n-z,\Lambda(z_n)\Big>_{Z\times Z^*}= 0\,,
\end{equation}
then $\Lambda(z_n)\rightharpoonup \Lambda(z)$ weakly$^*$ in $Z^*$.
\end{itemize}
\end{lemma}
\begin{proof}
Assume that the mapping $\Lambda(z):Z\to Z^*$ is pseudo-monotone.
Choose arbitrary sequence $z_n\rightharpoonup z$ weakly in $Z$ and
denote
$$K=\liminf_{n\to+\infty}\Big<z_n-z,\Lambda(z_n)\Big>_{Z\times Z^*}\,.$$
Then, up to a subsequence, still denoted by $z_n$ we have
$$K=\lim_{n\to+\infty}\Big<z_n-z,\Lambda(z_n)\Big>_{Z\times Z^*}\,.$$
Thus if we assume that $K\leq 0$, by \er{ftguhhhhihggjgjk} with
$y=z$ we deduce $K\geq 0$. Therefore, since the sequence
$z_n\rightharpoonup z$ was chosen arbitrary we deduce that for every
sequence $\{z_n\}_{n=1}^{+\infty}\subset Z$, such that
$z_n\rightharpoonup z$ weakly in $Z$ we have
\er{ftguhhhhikkjhjhjkjkkkkkkk}.
Next assume that for some sequence $\{z_n\}_{n=1}^{+\infty}\subset
Z$, such that $z_n\rightharpoonup z$ weakly in $Z$ we have
\er{ftguhhhhikkjhjhjkjkkkkkkkjjjjhgghhh}. Then, by
\er{ftguhhhhihggjgjk} for this sequence we must have
\begin{equation}
\label{ftguhhhhihggjgjkbghghfg}
\lim_{n\to+\infty}\Big<z_n-z,\Lambda(z_n)\Big>_{Z\times Z^*}=
0\quad\text{and}\quad\liminf_{n\to+\infty}\Big<z_n-y,\Lambda(z_n)\Big>_{X\times
X^*}\geq\Big<z-y,\Lambda(z)\Big>_{Z\times Z^*}\;\;\forall y\in Z\,.
\end{equation}
Therefore, plugging the first equality in
\er{ftguhhhhihggjgjkbghghfg} into the second inequality we obtain
\begin{equation}
\label{ftguhhhhihggjgjkbghghfggggfhffff}
\liminf_{n\to+\infty}\Big<z-y,\Lambda(z_n)\Big>_{Z\times
Z^*}\geq\Big<z-y,\Lambda(z)\Big>_{Z\times Z^*}\quad\forall y\in Z\,.
\end{equation}
We can rewrite \er{ftguhhhhihggjgjkbghghfggggfhffff} as
\begin{equation}
\label{ftguhhhhihggjgjkbghghfggggfhffffghjjhfjfg}
\liminf_{n\to+\infty}\big<h,\Lambda(z_n)\big>_{Z\times
Z^*}\geq\big<h,\Lambda(z)\big>_{Z\times Z^*}\quad\forall h\in Z\,.
\end{equation}
Thus, interchanging between $h$ and $-h$ in
\er{ftguhhhhihggjgjkbghghfggggfhffffghjjhfjfg} we obtain
\begin{equation}
\label{ftguhhhhihggjgjkbghghfggggfhffffghjjhfjfgghddd}
\limsup_{n\to+\infty}\big<h,\Lambda(z_n)\big>_{Z\times
Z^*}\leq\big<h,\Lambda(z)\big>_{Z\times Z^*}\quad\forall h\in Z\,.
\end{equation}
So, by plugging \er{ftguhhhhihggjgjkbghghfggggfhffffghjjhfjfg} and
\er{ftguhhhhihggjgjkbghghfggggfhffffghjjhfjfgghddd} we finally
deduce
\begin{equation}
\label{ftguhhhhihggjgjkbghghfggggfhffffghjjhfjfgghdddgjkgghfjjklhklhhjh}
\lim_{n\to+\infty}\big<h,\Lambda(z_n)\big>_{Z\times
Z^*}=\big<h,\Lambda(z)\big>_{Z\times Z^*}\quad\forall h\in Z\,.
\end{equation}
I.e. $\Lambda(z_n)\rightharpoonup \Lambda(z)$ weakly$^*$ in $Z^*$.

 Next assume that the mapping $\Lambda(z):Z\to Z^*$ satisfies the
conditions {\bf(i)} and {\bf(ii)}. Consider the sequence
$\{z_n\}_{n=1}^{+\infty}\subset Z$, satisfying
\begin{equation*}
z_n\rightharpoonup z\;\;\text{weakly
in}\;\;Z\quad\quad\text{and}\quad\quad\limsup_{n\to+\infty}\Big<z_n-z,\Lambda(z_n)\Big>_{Z\times
Z^*}\leq 0
\end{equation*}
Then by condition {\bf(i)} we must have
\begin{equation}
\label{ftguhhhhikkjhjhjkjkkkkkkkjjbgjkgjkgh}
\lim_{n\to+\infty}\Big<z_n-z,\Lambda(z_n)\Big>_{Z\times Z^*}= 0\,.
\end{equation}
Thus by condition {\bf(ii)} we must have
\begin{equation}
\label{ftguhhhhihggjgjkbghghfggggfhffffhhjhgigyguygy}
\lim_{n\to+\infty}\Big<z-y,\Lambda(z_n)\Big>_{Z\times
Z^*}=\Big<z-y,\Lambda(z)\Big>_{Z\times Z^*}\quad\forall y\in Z\,.
\end{equation}
Thus by \er{ftguhhhhikkjhjhjkjkkkkkkkjjbgjkgjkgh} and
\er{ftguhhhhihggjgjkbghghfggggfhffffhhjhgigyguygy} we finally deduce
\begin{multline}
\label{ftguhhhhihggjgjkvcfcfgcfxd}
\lim_{n\to+\infty}\Big<z_n-y,\Lambda(z_n)\Big>_{Z\times
Z^*}=\lim_{n\to+\infty}\Big<z_n-z,\Lambda(z_n)\Big>_{Z\times
Z^*}+\lim_{n\to+\infty}\Big<z-y,\Lambda(z_n)\Big>_{Z\times Z^*}
\\=0+\Big<z-y,\Lambda(z)\Big>_{Z\times Z^*}\quad\quad\quad\quad\quad\quad\quad\quad\forall y\in Z\,.
\end{multline}
Thus, the mapping $\Lambda(z):Z\to Z^*$ is pseudo-monotone.
\end{proof}
\begin{lemma}\label{hhhhhhhhhhhhhhhhhhiogfydtdtydjkgkgk}
Let $Z$ be a
Banach space and $Z^*$ be a corresponding dual space. Assume that
the mapping $\Lambda(z):Z\to Z^*$ is monotone. Moreover assume that
$\Lambda(z):Z\to Z^*$ is continuous for every $z\in Z$ or more
generally the function $\zeta_{z,h}(t):\R\to\R$, defined by
\begin{equation}\label{fguyfuyfugyuguhkg}
\zeta_{z,h}(t):=\Big<h,\Lambda\big(z-t h\big)\Big>_{Z\times
Z^*}\quad\forall z,h\in Z\,,\quad\forall t\in\R\,,
\end{equation}
is continuous on $t$ for every $z,h\in Z$. Then the mapping
$\Lambda(z)$ is pseudo-monotone.
\end{lemma}
\begin{proof}
Assume that the mapping $\Lambda(z):Z\to Z^*$ is monotone. I.e.
\begin{equation}
\label{ftguhhhhihggjgjkjggkjgjfjfgdgjhghgh}
\Big<y-z,\Lambda(y)-\Lambda(z)\Big>_{Z\times Z^*}\geq
0\quad\forall\, y,z\in Z\,.
\end{equation}
Then in particular for every sequence
$\{z_n\}_{n=1}^{+\infty}\subset Z$, such that $z_n\rightharpoonup z$
weakly in $Z$, we obtain
\begin{equation}\label{ftguhhhhikkjhjhjkjkkkkkkkjjbgjkgjkghhghff}
\liminf_{n\to+\infty}\Big<z_n-z,\Lambda(z_n)\Big>_{Z\times
Z^*}\geq\liminf_{n\to+\infty}\Big<z_n-z,\Lambda(z)\Big>_{Z\times
Z^*}=0\,.
\end{equation}
So the condition {\bf(i)} of Lemma
\ref{hhhhhhhhhhhhhhhhhhiogfydtdtyd} is satisfied. Next assume that
the sequence $\{z_n\}_{n=1}^{+\infty}\subset Z$ satisfies
\begin{equation}
\label{ftguhhhhikkjhjhjkjkkkkkkkjjjjhgghhhghghhghgfjfg}
z_n\rightharpoonup z\;\;\text{weakly
in}\;\;Z\quad\text{and}\quad\lim_{n\to+\infty}\Big<z_n-z,\Lambda(z_n)\Big>_{Z\times
Z^*}= 0\,.
\end{equation}
We will prove now that we must have $\Lambda(z_n)\rightharpoonup
\Lambda(z)$ weakly$^*$ in $Z^*$. Indeed, by
\er{ftguhhhhihggjgjkjggkjgjfjfgdgjhghgh} we obtain
\begin{equation}
\label{ftguhhhhihggjgjkjggkjgjfjfgdgjhghghgfffff}
\liminf_{n\to+\infty}\Big<z_n-y,\Lambda(z_n)-\Lambda(y)\Big>_{Z\times
Z^*}\geq 0\quad\forall\,y\in Z\,.
\end{equation}
Thus plugging \er{ftguhhhhikkjhjhjkjkkkkkkkjjjjhgghhhghghhghgfjfg}
into \er{ftguhhhhihggjgjkjggkjgjfjfgdgjhghghgfffff} we deduce
\begin{multline}
\label{ftguhhhhihggjgjkjggkjgjfjfgdgjhghghgfffffgfjfgfgfjfffdgh}
\liminf_{n\to+\infty}\Big<z-y,\Lambda(z_n)\Big>_{Z\times
Z^*}=\liminf_{n\to+\infty}\Big<z_n-y,\Lambda(z_n)\Big>_{Z\times
Z^*}\geq\\ \liminf_{n\to+\infty}\Big<z_n-y,\Lambda(y)\Big>_{Z\times
Z^*}=\Big<z-y,\Lambda(y)\Big>_{Z\times Z^*}\quad\forall\,y\in Z\,.
\end{multline}
Then choosing $y:=z-th$ in
\er{ftguhhhhihggjgjkjggkjgjfjfgdgjhghghgfffffgfjfgfgfjfffdgh} for
arbitrary $h\in Z$ and $t>0$ we obtain
\begin{equation}
\label{ftguhhhhihggjgjkjggkjgjfjfgdgjhghghgfffffgfjfgfgfjfffdghccgccgh}
\liminf_{n\to+\infty}\Big<h,\Lambda(z_n)\Big>_{Z\times Z^*}\geq
\Big<h,\Lambda(z-th)\Big>_{Z\times Z^*}\quad\forall\,h\in
Z\,,\;\forall t>0\,.
\end{equation}
Therefore, tending $t\to 0^+$ in
\er{ftguhhhhihggjgjkjggkjgjfjfgdgjhghghgfffffgfjfgfgfjfffdghccgccgh}
and using the continuity of the function in the r.h.s. of
\er{ftguhhhhihggjgjkjggkjgjfjfgdgjhghghgfffffgfjfgfgfjfffdghccgccgh}
we infer
\begin{equation}
\label{ftguhhhhihggjgjkjggkjgjfjfgdgjhghghgfffffgfjfgfgfjfffdghccgccghdgghcgvn}
\liminf_{n\to+\infty}\big<h,\Lambda(z_n)\big>_{Z\times Z^*}\geq
\big<h,\Lambda(z)\big>_{Z\times Z^*}\quad\forall\,h\in Z\,.
\end{equation}
Thus, as before, interchanging between $h$ and $-h$ in
\er{ftguhhhhihggjgjkjggkjgjfjfgdgjhghghgfffffgfjfgfgfjfffdghccgccghdgghcgvn}
we obtain
\begin{equation}
\label{ftguhhhhihggjgjkbghghfggggfhffffghjjhfjfgghdddkkkhh}
\limsup_{n\to+\infty}\big<h,\Lambda(z_n)\big>_{Z\times
Z^*}\leq\big<h,\Lambda(z)\big>_{Z\times Z^*}\quad\forall h\in Z\,.
\end{equation}
So, by plugging
\er{ftguhhhhihggjgjkjggkjgjfjfgdgjhghghgfffffgfjfgfgfjfffdghccgccghdgghcgvn}
and \er{ftguhhhhihggjgjkbghghfggggfhffffghjjhfjfgghdddkkkhh} we
finally deduce
\begin{equation}
\label{ftguhhhhihggjgjkbghghfggggfhffffghjjhfjfgghdddgjkgghfjjklhklhhjhbvvhhjjg}
\lim_{n\to+\infty}\big<h,\Lambda(z_n)\big>_{Z\times
Z^*}=\big<h,\Lambda(z)\big>_{Z\times Z^*}\quad\forall h\in Z\,.
\end{equation}
I.e. $\Lambda(z_n)\rightharpoonup \Lambda(z)$ weakly$^*$ in $Z^*$.
So the condition {\bf(ii)} of Lemma
\ref{hhhhhhhhhhhhhhhhhhiogfydtdtyd} is satisfied. Therefore, by this
Lemma the mapping $\Lambda(z)$ is pseudo-monotone.
\end{proof}

Next we have the following well known Lemma:
\begin{lemma}\label{vbnhjjm}
Let $Y$ and $Z$ be two reflexive Banach spaces. Furthermore, let
$S\in \mathcal{L}(Y;Z)$  be an injective operator (i.e. it satisfies
$\ker S=\{0\}$) and  let $S^*\in \mathcal{L}(Z^*;Y^*)$ be the
corresponding adjoint operator, which satisfies
\begin{equation}\label{tildetdall}
\big<y,S^*\cdot z^*\big>_{Y\times Y^*}:=\big<S\cdot
y,z^*\big>_{Z\times Z^*}\quad\quad\text{for every}\; z^*\in
Z^*\;\text{and}\;y\in Y\,.
\end{equation}
Next assume that $a,b\in\R$ s.t. $a<b$. Let $w(t)\in
L^\infty(a,b;Y)$ be such that the function $v:[a,b]\to Z$ defined by
$v(t):=S\cdot \big(w(t)\big)$ belongs to $W^{1,q}(a,b;Z)$ for some
$q\geq 1$. Then we can redefine $w$ on a subset of $[a,b]$ of
Lebesgue measure zero, so that $w(t)$ will be $Y$-weakly
continuous in $t$ on $[a,b]$ ( i.e. $w\in C_w^0(a,b;Y)$ ). Moreover,
for every $a\leq \alpha<\beta\leq b$ and for every $\delta(t)\in
C^1\big([a,b];Z^*\big)$ we will have
\begin{equation}\label{eqmult}
\int\limits_\alpha^\beta\bigg\{\Big<\frac{dv}{dt}(t),\delta(t)\Big>_{Z\times
Z^*}+\Big<v(t),\frac{d\delta}{dt}(t)\Big>_{Z\times Z^*}\bigg\}dt=
\big<w(\beta),S^*\cdot\delta(\beta)\big>_{Y\times
Y^*}-\big<w(\alpha),S^*\cdot\delta(\alpha)\big>_{Y\times Y^*}\,.
\end{equation}
\end{lemma}

\begin{definition}\label{7bdf}
Let $X$ be a reflexive Banach space and $X^*$ the corresponding dual
space. Furthermore let $H$ be a Hilbert space and $T\in
\mathcal{L}(X,H)$ be an injective (i.e. it satisfies $\ker T=\{0\}$)
inclusion operator such that its image is dense on $H$. Then we call
the triple $\{X,H,X^*\}$ an evolution triple with the corresponding
inclusion operator $T$. Throughout this paper we assume the space
$H^*$ be equal to $H$ (remember that $H$ is a Hilbert space) but in
general we don't associate $X^*$ with $X$ even in the case where $X$
is a Hilbert space (and thus $X^*$ will be isomorphic to $X$).
Furthermore, we define the bounded linear operator $\widetilde{T}\in
\mathcal{L}(H;X^*)$ by the formula
\begin{equation}\label{tildet}
\big<x,\widetilde{T}\cdot y\big>_{X\times X^*}:=\big<T\cdot x,y\big>_{H\times H}\quad\quad\text{for every}\; y\in H\;\text{and}\;x\in X\,.
\end{equation}
In particular
$\|\widetilde{T}\|_{\mathcal{L}(H;X^*)}=\|T\|_{\mathcal{L}(X;H)}$
and since we assumed that the image of $T$ is dense in $H$ we deduce
that $\ker \widetilde{T}=\{0\}$ and so $\widetilde{T}$ is an
injective operator. So $\widetilde{T}$ is an inclusion of $H$ to
$X^*$ and the operator $I:=\widetilde{T}\circ T$ is an injective
inclusion of $X$ to $X^*$. Furthermore, clearly
\begin{equation}\label{tildethlhjhghjf}
\big<x,I\cdot z\big>_{X\times X^*}=\big<T\cdot x,T\cdot
z\big>_{H\times H}=\big<z,I\cdot x\big>_{X\times
X^*}\quad\quad\text{for every}\; x,z\in X\,.
\end{equation}
So $I\in \mathcal{L}(X,X^*)$ is self-adjoint operator. Moreover, $I$
is strictly positive, since
\begin{equation}\label{tildethlhjhghjffgfhfh}
\big<x,I\cdot x\big>_{X\times X^*}=\|T\cdot
x\|^2_H>0\quad\quad\forall x\neq 0\in X\,.
\end{equation}
\end{definition}
\begin{lemma}\label{hdfghdiogdiofg}
Let $X$ be a reflexive Banach space and $X^*$ the corresponding dual
space. Furthermore, let $I\in \mathcal{L}(X,X^*)$ be a self-adjoint
and strictly positive operator. i.e.
\begin{equation}\label{tildethlhjhghjfvvjhjhj}
\big<x,I\cdot z\big>_{X\times X^*}=\big<z,I\cdot x\big>_{X\times
X^*}\quad\quad\text{for every}\; x,z\in X\,,
\end{equation}
and
\begin{equation}\label{tildethlhjhghjffgfhfhhffkgh}
\big<x,I\cdot x\big>_{X\times X^*}>0\quad\quad\forall x\neq 0\in
X\,.
\end{equation}
Then there exists a Hilbert space $H$ and an injective operator
$T\in \mathcal{L}(X,H)$ (i.e. $\ker T=\{0\}$), whose image is dense
in $H$, and such that if we consider the operator $\widetilde{T}\in
\mathcal{L}(H;X^*)$, defined by the formula \er{tildet},
then we will have
\begin{equation}\label{tildethlhjhghjffgfhfhhffkghbjhjkhjjk}
(\widetilde{T}\circ T)\cdot x=I\cdot x\quad\quad\forall x\in X\,.
\end{equation}
I.e. $\{X,H,X^*\}$ is an evolution triple with the corresponding
inclusion operator $T\in \mathcal{L}(X;H)$, as it was defined in
Definition \ref{7bdf}, together with the corresponding operator
$\widetilde{T}\in \mathcal{L}(H;X^*)$, defined as in \er{tildet},
and $I\equiv \widetilde{T}\circ T$.
\end{lemma}
\begin{proof}
Since $I$ is a self-adjoint and strictly positive operator, the
identity
\begin{equation}\label{tildethlhjhghjfvvjhjhjkggkhgfjfg}
\big<\big<x,z\big>\big>:=\big<x,I\cdot z\big>_{X\times
X^*}=\big<z,I\cdot x\big>_{X\times X^*}\quad\quad\text{for every}\;
x,z\in X\,,
\end{equation}
defines a scalar product in $X$ and the corresponding Euclidian norm
$\big|\big|x\big|\big|:=\sqrt{<<x,x>>}$. Denote the closure of the
space $X$ with respect to this Euclidian norm by $H$ and the trivial
embedding of $X$ into $H$ by $T$. Thus $H$ will be a Hilbert space
and $T\in\mathcal{L}(X,H)$ will be an injective bounded linear
operator whose image is dense in $H$. Moreover,
\begin{equation}\label{tildethlhjhghjfggyhfyufyuygh}
\big<T\cdot x,T\cdot z\big>_{H\times
H}=\big<\big<x,z\big>\big>=\big<x,I\cdot z\big>_{X\times
X^*}\quad\quad\text{for every}\; x,z\in X\,.
\end{equation}
Thus if we consider the operator $\widetilde{T}\in
\mathcal{L}(H;X^*)$, defined as in \er{tildet}, by the formula
\begin{equation}\label{tildetjbghgjgklhkhj}
\big<x,\widetilde{T}\cdot y\big>_{X\times X^*}:=\big<T\cdot
x,y\big>_{H\times H}\quad\quad\text{for every}\; y\in
H\;\text{and}\;x\in X\,,
\end{equation}
then plugging \er{tildetjbghgjgklhkhj} into
\er{tildethlhjhghjfggyhfyufyuygh} we deduce
\begin{equation}\label{tildethlhjhghjfggyhfyufyuyghghfhgchgdgjg}
\big<x,(\widetilde{T}\circ T)\cdot z\big>_{X\times X^*}=\big<T\cdot
x,T\cdot z\big>_{H\times H}=\big<x,I\cdot z\big>_{X\times
X^*}\quad\quad\text{for every}\; x,z\in X\,.
\end{equation}
I.e. $\widetilde{T}\circ T\equiv I$.
\end{proof}
Next as a particular case of Lemma \ref{vbnhjjm} we have the
following Corollary.
\begin{corollary}\label{vbnhjjmcor}
Let $\{X,H,X^*\}$ be an evolution triple with the corresponding
inclusion operator $T\in \mathcal{L}(X;H)$, as it was defined in
Definition \ref{7bdf}, together with the corresponding operator
$\widetilde{T}\in \mathcal{L}(H;X^*)$, defined as in \er{tildet},
and let $a,b\in\R$ be s.t. $a<b$. Let $w(t)\in L^\infty(a,b;H)$ be
such that the function $v:[a,b]\to X^*$ defined by $v(t):=\widetilde
T\cdot \big(w(t)\big)$ belongs to $W^{1,q}(a,b;X^*)$ for some $q\geq
1$. Then we can redefine $w$ on a subset of $[a,b]$ of Lebesgue
measure zero, so that $w(t)$ will be $H$-weakly continuous in $t$ on
$[a,b]$ ( i.e. $w\in C_w^0(a,b;H)$ ). Moreover, for every $a\leq
\alpha<\beta\leq b$ and for every $\delta(t)\in
C^1\big([a,b];X\big)$ we will have
\begin{equation}\label{eqmultcor}
\int\limits_\alpha^\beta\bigg\{\Big<\delta(t), \frac{dv}{dt}(t)\Big>_{X\times X^*}+\Big<\frac{d\delta}{dt}(t), v(t)\Big>_{X\times X^*}\bigg\}dt=
\big<T\cdot\delta(\beta),w(\beta)\big>_{H\times H}-\big<T\cdot\delta(\alpha),w(\alpha)\big>_{H\times H}\,.
\end{equation}
\end{corollary}

The following result is well known in the study of evolutional
equations:
\begin{lemma}\label{lem2}
Let $\{X,H,X^*\}$ be an evolution triple with the corresponding
inclusion operator $T\in \mathcal{L}(X;H)$, as it was defined in
Definition \ref{7bdf}, together with the corresponding operator
$\widetilde{T}\in \mathcal{L}(H;X^*)$, defined as in \er{tildet},
and let $a,b\in\R$ be s.t. $a<b$. Let $u(t)\in L^q(a,b;X)$ for some
$q>1$ such that the function $v(t):[a,b]\to X^*$ defined by
$v(t):=I\cdot \big(u(t)\big)$ belongs to $W^{1,q^*}(a,b;X^*)$ for
$q^*:=q/(q-1)$, where we denote $I:=\widetilde T\circ T:\,X\to X^*$.
Then the function $w(t):[a,b]\to H$ defined by $w(t):=T\cdot
\big(u(t)\big)$ belongs to $L^\infty(a,b;H)$ and for every
subinterval $[\alpha,\beta]\subset[a,b]$ we have
\begin{equation}\label{energyravenstvo}
\int_\alpha^\beta\Big<u(t),\frac{dv}{dt}(t)\Big>_{X\times X^*}dt=\frac{1}{2}\Big(\|w(\beta)\|_H^2
-\|w(\alpha)\|_H^2\Big)\,,
\end{equation}
up to a redefinition of $w(t)$ on a subset of $[a,b]$ of Lebesgue measure zero, such that $w$ is $H$-weakly continuous, as it was stated in
Corollary \ref{vbnhjjmcor}.
\end{lemma}

We will need in the sequel the following compactness results.
%
%
%
%
%
%
%
%
%
%
\begin{lemma}\label{ComTem1PP}
Let $X$, $Y$ $Z$
be three Banach spaces, such that $X$
is a reflexive space. Furthermore, let $T\in \mathcal{L}(X;Y)$ and
$S\in \mathcal{L}(X;Z)$
be bounded
linear operators. Moreover assume that $S$
is an injective inclusion (i.e. it satisfies $\ker S=\{0\}$)
and $T$ is a
compact operator. Assume that $a,b\in\R$ such that $a<b$, $1\leq
q<+\infty$ and $\{u_n(t)\}\subset L^q(a,b;X)$ is a bounded in
$L^q(a,b;X)$ sequence of functions, such that the functions
$v_n(t):(a,b)\to Z$, defined by $v_n(t):=S\cdot\big(u_n(t)\big)$,
belongs to $L^\infty(a,b;Z)$, the sequence $\{v_n(t)\}$ is bounded
in $L^\infty(a,b;Z)$ and for a.e. $t\in(a,b)$ we have
\begin{equation}\label{fghffhdpppplllkkkll}
v_n(t)\rightharpoonup v(t)\quad\text{weakly in}\;
Z\;\text{as}\;n\to+\infty\,.
\end{equation}
Then,
\begin{equation}\label{staeae1glllukjkjkojkl}
\big\{T\cdot\big(u_n(t)\big)\big\}\quad\text{converges strongly in }
L^q(a,b;Y)\,.
\end{equation}
\end{lemma}
\begin{proof}
First of all we would like to observe that without loss of
generality we may assume that the spaces $X$, $Y$ and $Z$ are
separable. Indeed, in the general case since $u_n(t)\in L^q(a,b;X)$
then $u_n$ is strictly measurable. Thus in particular $u_n$ is
separately valued, i.e. for every $n$ there exists a separable
subspace $X_n\subset X$ such that $u_n(t)\in X_n$ for a.e.
$t\in(a,b)$. Define $\bar X$ be the closure in $X$ of the linear
span of $\bigcup_{n=1}^{+\infty} X_n$. Then $\bar X$ is a separable
subspace of $X$ by itself and by the construction for a.e.
$t\in(a,b)$ for every $n$ we have $u_n(t)\in \bar X$. Then we also
can define $\bar Y$ and $\bar Z$ as the closures of the images of
the subspace $\bar X$ under the transformations $T$ and $S$
respectively. So $\bar Y$ and $\bar Z$ are separable subspaces of
$Y$ and $Z$ respectively. The new spaces $\bar X$, $\bar Y$ and
$\bar Z$ together with the operators $T\llcorner\bar X$ and
$S\llcorner\bar X$ and the functions $u_n(t)\in L^q(a,b;\bar X)$
satisfy all the conditions of the present lemma.

 Therefore, from now we assume the spaces $X$, $Y$ and $Z$ to be
separable. Thus by Lemma \ref{hilbcombanback} from the Appendix
there exists a separable Hilbert space $U$ and an operator $L\in
\mathcal{L}(Z;U)$ such that $L$ is injective (i.e. $\ker L=\{0\}$)
and compact.
Thus since $L$ is a compact operator, by \er{fghffhdpppplllkkkll},
for a.e. $t\in(a,b)$ we have
\begin{equation}\label{fghffhdpppplllkkkllkkklll}
L\cdot v_n(t)\to L\cdot v(t)\quad\text{strongly in}\;
U\;\text{as}\;n\to+\infty\,.
\end{equation}
Moreover, the sequence $\{L\cdot v_n(t)\}$ is bounded in
$L^\infty(a,b;U)$. Therefore, by the Dominated Convergence Theorem
we deduce that
\begin{equation}\label{fghffhdpppplllkkkllkkklllkkklmmm}
\big(L\circ S\big)\cdot \big(u_n(t)\big)=L\cdot \big(v_n(t)\big)\to
L\cdot \big(v(t)\big)\quad\text{strongly in}\;
L^q(a,b;U)\;\text{as}\;n\to+\infty\,.
\end{equation}
Next since $S$ and $L$ are injective inclusions, we deduce that
$L\circ S\in \mathcal{L}(X,U)$ is an injective inclusion. Therefore,
using Lemma \ref{Aplem1} from the Appendix we deduce that for every
$\e>0$ there exists $c_\e>0$ such that for all $n,m\in\mathbb{N}$ we
must have
\begin{multline}\label{Eqetaenerwithtgjg}
\Big\|T\cdot u_{n+m}(t)- T\cdot
u_n(t)\Big\|_Y\leq\e\Big\|u_{n+m}(t)-
u_n(t)\Big\|_X+c_\e\Big\|(L\circ S)\cdot u_{n+m}(t)-(L\circ S)\cdot
u_n(t)\Big\|_U\quad\forall t\in(a,b)\,.
\end{multline}
Therefore, for every $\e>0$, for all $n,m\in\mathbb{N}$ we obtain
\begin{multline}\label{Eqetaenerwithtglobggjkl}
\Big\|T\cdot \big(u_{n+m}(t)\big)-T\cdot\big(
u_n(t)\big)\Big\|_{L^q(a,b;Y)}\\ \leq\e\Big\|u_{n+m}(t)-
u_n(t)\Big\|_{L^q(a,b;X)}+c_\e\Big\|L\cdot
\big(v_{n+m}(t)\big)-L\cdot\big(v_n(t)\big)\Big\|_{L^q(a,b;U)}\,.
\end{multline}
However, since the sequence $\{u_n\}$ is bounded in $L^q(a,b;X)$,
using \er{Eqetaenerwithtglobggjkl} we deduce that there exists a
constant $C_0>0$ independent on $\e$ such that
\begin{multline}\label{Eqetaenerwithtglobggjklkklklkl}
\Big\|T\cdot \big(u_{n+m}(t)\big)-T\cdot\big(
u_n(t)\big)\Big\|_{L^q(a,b;Y)}\leq C_0\e+c_\e\Big\|L\cdot
\big(v_{n+m}(t)\big)-L\cdot\big(v_n(t)\big)\Big\|_{L^q(a,b;U)}\,.
\end{multline}
On the other hand, by \er{fghffhdpppplllkkkllkkklllkkklmmm}, there
exists $n_0=n_0(\e)\in\mathbb{N}$ such that for every $n>n_0$ and
every $m\in\mathbb{N}$ we have
\begin{equation}\label{Eqetaenerwithtglobggjklkklklklyky}
\Big\|L\cdot
\big(v_{n+m}(t)\big)-L\cdot\big(v_n(t)\big)\Big\|_{L^q(a,b;U)}\leq\frac{\e}{c_\e}\,.
\end{equation}
Thus, plugging \er{Eqetaenerwithtglobggjklkklklklyky} into
\er{Eqetaenerwithtglobggjklkklklkl}, we obtain that for every
$n>n_0$ and every $m\in\mathbb{N}$ we must have
\begin{equation}\label{Eqetaenerwithtglobggjklkklklklfbfkk}
\Big\|T\cdot \big(u_{n+m}(t)\big)-T\cdot\big(
u_n(t)\big)\Big\|_{L^q(a,b;Y)}\leq (C_0+1)\e\,.
\end{equation}
Therefore, since $\e>0$ in \er{Eqetaenerwithtglobggjklkklklklfbfkk}
is arbitrary and since $L^q(a,b;Y)$ is a Banach space we finally
deduce \er{staeae1glllukjkjkojkl}.
\end{proof}
%
%
%
%
%
%
%
%
%
%

\begin{lemma}\label{helplemlll}
Let $Z$ be a reflexive Banach space and let
$\big\{v_n(t)\big\}_{n=1}^{+\infty}\subset W^{1,1}(a,b;Z)$ be a
sequence of functions, bounded in $W^{1,1}(a,b;Z)$.
Then, $\big\{v_n(t)\big\}_{n=1}^{+\infty}$ is bounded in
$L^\infty(a,b;Z)$ and, up to a subsequence, we have
\begin{equation}\label{fghffhdpppplllkkkllkklkl}
v_n(t)\rightharpoonup v(t)\quad\text{weakly in}\;
Z\;\text{as}\;\;n\to+\infty\,,\quad\text{for a.e}\;\,t\in(a,b)\,.
\end{equation}
\end{lemma}
\begin{proof}
As before without loss of generality we may assume that the space
$Z$ is separable. Thus since $Z$ is a reflexive we deduce that the
space $Z^*$ is also separable. Next since $W^{1,1}(a,b;Z)$ is
continuously embedded in $L^\infty(a,b;Z)$, there exists a constant
$C>0$ such that
\begin{equation}\label{fhgfhjtjuhyt}
\big\|v_n(t)\big\|_Z\leq C\quad\forall
t\in\mathfrak{R}_0\subset(a,b)\,,
\end{equation}
where $\mathcal{L}^1\big((a,b)\setminus\mathfrak{R}_0\big)=0$. On
the other hand, since $Z^*$ is separable, there exists a sequence
$\{\sigma_j\}_{j=1}^{+\infty}\subset Z^*$ which is dense in $Z^*$.
For every $j,n\in\mathbb{N}$ and every $t\in(a,b)$ set
\begin{equation}\label{ugufuhrlll}
h_{n}^{(j)}(t):=\big<v_n(t),\sigma_j\big>_{Z\times Z^*}\in
W^{1,1}(a,b;\R)\,.
\end{equation}
Then for every fixed $j$ the sequence
$\{h_{n}^{(j)}(t)\}_{n=1}^{+\infty}$ is bounded in
$W^{1,1}(a,b;\R)$. Thus since the embedding of $W^{1,1}(a,b;\R)$
into $L^1(a,b;\R)$ is compact and since every $L^1$-convergent
sequence has a subsequence which converges almost everywhere, there
exists an increasing sequence
$\{n^{(1)}_k\}_{k=1}^{+\infty}\subset\mathbb{N}$ and a set
$\mathfrak{R}_1\subset \mathfrak{R}_0$ such that
$\mathcal{L}^1\big((a,b)\setminus\mathfrak{R}_1\big)=0$ and
$\lim_{k\to +\infty}h_{n^{(1)}_k}^{(1)}(t)=\bar h^{(1)}(t)$ for
every $t\in\mathfrak{R}_1$. In the same way there exists a further
subsequence
$\{n^{(2)}_k\}_{k=1}^{+\infty}\subset\{n^{(1)}_k\}_{k=1}^{+\infty}$
and a set $\mathfrak{R}_2\subset \mathfrak{R}_1$ such that
$\mathcal{L}^1\big((a,b)\setminus\mathfrak{R}_2\big)=0$ and
$\lim_{k\to +\infty}h_{n^{(2)}_k}^{(2)}(t)=\bar h^{(2)}(t)$ for
every $t\in\mathfrak{R}_2$. Continuing this process we obtain that
for every $m\in\mathbb{N}$ there exists a subsequence
$\{n^{(m+1)}_k\}_{k=1}^{+\infty}\subset\{n^{(m)}_k\}_{k=1}^{+\infty}$
and a set $\mathfrak{R}_{m+1}\subset \mathfrak{R}_m$ such that
$\mathcal{L}^1\big((a,b)\setminus\mathfrak{R}_{m+1}\big)=0$ and
$\lim_{k\to +\infty}h_{n^{(m+1)}_k}^{(m+1)}(t)=\bar h^{(m+1)}(t)$
for every $t\in\mathfrak{R}_{m+1}$. Thus taking the diagonal
subsequence we obtain that, up to a subsequence, still denoted by
$v_n(t)$ we have
$$\lim_{n\to +\infty}\big<v_n(t),\sigma_j\big>_{Z\times Z^*}=
\bar h^{(j)}(t)\quad\forall t\in \mathfrak{\bar R},\;\;\forall
j\in\mathbb{N}\,,$$ where $\mathfrak{\bar
R}:=\cap_{j=1}^{+\infty}\mathfrak{R}_j$. Moreover, clearly we have
$\mathcal{L}^1\big((a,b)\setminus\mathfrak{\bar R}\big)=0$.
Therefore by the fact that $\{\sigma_j\}_{j=1}^{+\infty}\subset Z^*$
is dense in $Z^*$ and by \er{fhgfhjtjuhyt} we obtain that
$$v_n(t)\rightharpoonup v(t)\quad\text{weakly in}\;
Z\;\text{as}\;n\to+\infty\,,\quad\forall t\in\mathfrak{\bar R}\,,$$
i.e. almost everywhere in $(a,b)$.
\end{proof}
As a direct consequence of Lemma \ref{ComTem1PP} and Lemma
\ref{helplemlll} we have the following Lemma.
\begin{lemma}\label{ComTem1}
Let $X$, $Y$ and $Z$ be three Banach spaces, such that $X$ and $Z$
are reflexive.
Furthermore, let $T\in
\mathcal{L}(X;Y)$ and $S\in \mathcal{L}(X;Z)$ be bounded linear
operators. Moreover assume that $S$ is an injective inclusion (i.e.
it satisfies $\ker S=\{0\}$) and $T$
is a compact operator. Assume that $a,b\in\R$ such that $a<b$,
$1\leq q<+\infty$ and $\{u_n(t)\}\subset L^q(a,b;X)$ is a bounded in
$L^q(a,b;X)$ sequence of functions, such that the functions
$v_n(t):(a,b)\to Z$, defined by $v_n(t):=S\cdot\big(u_n(t)\big)$,
belongs to $W^{1,1}(a,b;Z)$ and the sequence
$\big\{\frac{dv_n}{dt}(t)\big\}$ is bounded in $L^1(a,b;Z)$. Then,
up to a subsequence,
\begin{equation}\label{staeae1gllluk}
\big\{T\cdot\big(u_n(t)\big)\big\}\quad\text{converges strongly in }
L^q(a,b;Y)\,.
\end{equation}
\end{lemma}

\section{The properties of the Euler-Lagrange equations}
\begin{definition}\label{defH}
Let $\{X,H,X^*\}$ be an evolution triple with the corresponding
inclusion operator $T\in \mathcal{L}(X;H)$, as it was defined in
Definition \ref{7bdf}, together with the corresponding operator
$\widetilde{T}\in \mathcal{L}(H;X^*)$, defined as in \er{tildet},
and let $a,b\in\R$ be s.t. $a<b$. Let $u(t)\in L^q(a,b;X)$ for some
$q>1$ be such, that the function $v(t):[a,b]\to X^*$ defined by
$v(t):=I\cdot \big(u(t)\big)$ belongs to $W^{1,q^*}(a,b;X^*)$ for
$q^*:=q/(q-1)$, where $I:=\widetilde T\circ T:\,X\to X^*$. Denote
the set of all such functions $u$ by $\mathcal{R}_{q}(a,b)$. Note
that by Lemma \ref{lem2}, for every $u(t)\in\mathcal{R}_{q}(a,b)$
the function $w(t):[a,b]\to H$ defined by $w(t):=T\cdot
\big(u(t)\big)$ belongs to $L^\infty(a,b;H)$ and, up to a
redefinition of $w(t)$ on a subset of $[a,b]$ of Lebesgue measure
zero,  $w$ is $H$-weakly continuous, as it was stated in Corollary
\ref{vbnhjjmcor}.
\end{definition}
\begin{definition}\label{defHkkkk}
Let $\{X,H,X^*\}$ be an evolution triple with the corresponding
inclusion operator $T\in \mathcal{L}(X;H)$, as it was defined in
Definition \ref{7bdf}, together with the corresponding operator
$\widetilde{T}\in \mathcal{L}(H;X^*)$, defined as in \er{tildet},
and let $\lambda\in\{0,1\}$ and $a,b,q\in\R$ be s.t. $a<b$ and
$q\geq 2$. Furthermore, for every $t\in[a,b]$ let
$\Psi_t(x):X\to[0,+\infty)$ be a strictly convex function which is
G\^{a}teaux differentiable at every $x\in X$, satisfies
$\Psi_t(0)=0$ and satisfies the growth condition
\begin{equation}\label{rostst}
(1/C_0)\,\|x\|_X^q-C_0\leq \Psi_t(x)\leq C_0\,\|x\|_X^q+C_0\quad\forall x\in X,\;\forall t\in[a,b]\,,
\end{equation}
for some $C_0>0$. We also assume that $\Psi_t(x)$ is Borel on the
pair of variables $(x,t)$ (see Definition
\ref{fdfjlkjjkkkkkllllkkkjjjhhhkkk}).
For every $t\in[a,b]$ denote by $\Psi_t^*$ the Legendre transform of
$\Psi_t$, defined by
$$\Psi_t^*(y):=\sup\big\{<z,y>_{X\times X^*}-\Psi_t(z):\;z\in X\big\}\quad\quad\forall y\in X^*\,.$$
Next for every $t\in[a,b]$ let $\Lambda_t(x):X\to X^*$ be a function
which is G\^{a}teaux differentiable at every $x\in X$,
$\Lambda_t(0)\in L^{q^*}(a,b;X^*)$ and the derivative of $\Lambda_t$
satisfies the growth condition
\begin{equation}\label{roststlambd}
\|D\Lambda_t(x)\|_{\mathcal{L}(X;X^*)}\leq g\big(\|T\cdot
x\|_H\big)\,\Big(\|x\|_X^{q-2}+\mu^{\frac{q-2}{q}}(t)\Big)\quad\forall
x\in X,\;\forall t\in[a,b]\,,
\end{equation}
for some non-decreasing function $g(s):[0+\infty)\to (0,+\infty)$
and some nonnegative function $\mu(t)\in L^1(a,b;\R)$.
We also assume that $\Lambda_t(x)$ is strongly
Borel on the pair of variables $(x,t)$ (see Definition
\ref{fdfjlkjjkkkkkllllkkkjjjhhhkkk}).
Assume also that $\Psi_t$ and
$\Lambda_t$ satisfy the following monotonicity condition
\begin{multline}\label{Monotone}
\bigg<h,\lambda \Big\{D\Psi_t\big(\lambda x+h\big)-D\Psi_t(\lambda
x)\Big\}+D\Lambda_t(x)\cdot h\bigg>_{X\times X^*}\geq
-\hat g\big(\|T\cdot
x\|_H\big)\Big(\|x\|^q_X+\hat\mu(t)\Big)\,\|T\cdot h\|^{2}_H
\\ \forall x,h\in X,\;\forall t\in[a,b]\,,
\end{multline}
for some non-decreasing function $\hat g(s):[0+\infty)\to
(0,+\infty)$ and some nonnegative function $\hat\mu(t)\in
L^1(a,b;\R)$.
Define the
functional $J(u):\mathcal{R}_{q}(a,b)\to\R$ (where
$\mathcal{R}_{q}(a,b)$ was defined in Definition \ref{defH}) by
\begin{multline}\label{hgffck}
J(u):=\\ \int\limits_a^b\Bigg\{\Psi_t\big(\lambda u(t)\big)+\Psi_t^*\bigg(-\frac{dv}{dt}(t)-\Lambda_t\big(u(t)\big)\bigg)+
\lambda\Big<u(t),\Lambda_t\big(u(t)\big)\Big>_{X\times X^*}\Bigg\}dt+\frac{\lambda}{2}\Big(\|w(b)\|_H^2-\|w(a)\|_H^2\Big)\,,
\end{multline}
where $w(t):=T\cdot\big(u(t)\big)$, $v(t):=I\cdot \big(u(t)\big)=\widetilde T\cdot \big(w(t)\big)$ with
$I:=\widetilde T\circ T:\,X\to X^*$ and we assume that $w(t)$ is $H$-weakly continuous on $[a,b]$, as it was stated in
Corollary \ref{vbnhjjmcor}.
Finally, for every $w_0\in H$ consider the minimization
problem
\begin{equation}\label{hgffckaaq1}
\inf\Big\{J(u):\,u\in\mathcal{R}_{q}(a,b),\,w(a)=w_0\Big\}\,.
\end{equation}
\end{definition}
\begin{remark}\label{rmmmdd}
Note that by Lemma \ref{Legendre}, for every $t\in[a,b]$
$\Psi_t^*(y)$ is a strictly convex function from $X^*$ to
$[0,+\infty)$, satisfies $\Psi_t^*(0)=0$ and
\begin{equation}\label{rostst*}
(1/C)\,\|y\|_{X^*}^{q^*}-C\leq \Psi_t^*(y)\leq
C\,\|y\|_{X^*}^{q^*}+C\quad\forall y\in X^*\;\forall t\in[a,b]\,,
\end{equation}
for some $C>0$ where $q^*:=q/(q-1)$. Moreover, $\Psi_t^*(y)$ is
G\^{a}teaux differentiable at every $y\in X^*$, and $x\in X$
satisfies $x=D\Psi_t^*(y)$ if and only if $y\in X^*$ satisfies
$y=D\Psi_t(x)$. Note also that $\Psi^*_t(y)$ is a Borel mapping
on the pair of variables $(y,t)$. Finally, note that since by Lemma
\ref{lem2} we have
\begin{equation*}
\int_a^b\Big<u(t),\frac{dv}{dt}(t)\Big>_{X\times X^*}dt=\frac{1}{2}\Big(\|w(b)\|_H^2-\|w(a)\|_H^2\Big)\,,
\end{equation*}
we can rewrite the definition of $J$ in \er{hgffck} by
\begin{multline}\label{fyhfypppp}
J(u):= \int\limits_a^b\Bigg\{\Psi_t\big(\lambda u(t)\big)+\Psi_t^*\bigg(-\frac{dv}{dt}(t)-\Lambda_t\big(u(t)\big)\bigg)+
\lambda\Big<u(t),\frac{dv}{dt}(t)+\Lambda_t\big(u(t)\big)\Big>_{X\times X^*}\Bigg\}dt\,.
\end{multline}
Then by the definition of the Legendre transform we deduce that
$J(u)\geq 0$ for every $u\in\mathcal{R}_{q}(a,b)$ and $J(u)=0$ if
and only if $u(t)$ is a solution of
\begin{equation}\label{uravnllll}
\frac{dv}{dt}(t)+\Lambda_t\big(u(t)\big)+D\Psi_t\big(\lambda u(t)\big)=0\quad\text{for a.e.}\; t\in(a,b)\,.
\end{equation}
\end{remark}
\begin{remark}\label{gdfgdghf}
Assume that, instead of \er{Monotone}, one requires that $\Psi_t$
and $\Lambda_t$ satisfy the following inequality
\begin{multline}\label{Monotone1111}
\bigg<h,\lambda \Big\{D\Psi_t\big(\lambda x+h\big)-D\Psi_t(\lambda
x)\Big\}+D\Lambda_t(x)\cdot h\bigg>_{X\times X^*}\geq
\\ \frac{\big|f(h,t)\big|^2}{\tilde g(\|T\cdot
x\|_H)}-\tilde g\big(\|T\cdot
x\|_H\big)\Big(\|x\|_X^q+\hat\mu(t)\Big)^{(2-r)/2}\big|f(h,t)\big|^r\,\|T\cdot
h\|^{(2-r)}_H\quad\forall x,h\in X,\;\forall t\in[a,b],
\end{multline}
for some non-decreasing function $\tilde g(s):[0+\infty)\to
(0,+\infty)$, some nonnegative function $\hat\mu(t)\in L^1(a,b;\R)$,
some function $f(x,t):X\times[a,b]\to\R$ and some constant
$r\in(0,2)$.
Then, \er{Monotone}
follows by the trivial inequality $(r/2)\,a^2+\big((2-r)/2\big)\,
b^2\geq a^r \,b^{2-r}$, valid for every two nonnegative real numbers
$a$ and $b$.
\end{remark}
\begin{lemma}\label{EulerLagrange1}
Let $\{X,H,X^*\}$ be an evolution triple with the corresponding
inclusion operator $T\in \mathcal{L}(X;H)$, as it was defined in
Definition \ref{7bdf}, together with the corresponding operator
$\widetilde{T}\in \mathcal{L}(H;X^*)$, defined as in \er{tildet}.
Furthermore, Let $a,b,q ,\lambda$
be such that $a<b$, $q\geq 2$ and $\lambda\in\{0,1\}$.
Assume that $\Psi_t$ and $\Lambda_t$ satisfy all the conditions of
Definition \ref{defHkkkk}. Furthermore, let $\mathcal{R}_{q}(a,b)$
and $J$ be as in Definitions \ref{defH} and \ref{defHkkkk}
respectively. Then for every $u\in\mathcal{R}_{q}(a,b)$ we have
$J(u)<+\infty$. Moreover, for every $u,h\in\mathcal{R}_{q}(a,b)$ and
every $s\in\R$ we have $(u+sh)\in\mathcal{R}_{q}(a,b)$ and
\begin{multline}\label{nolkj91}
\lim\limits_{s\to 0}\frac{1}{s}\Big(J(u+s h)-J(u)\Big)=\int\limits_a^b\Bigg\{\bigg<h(t),\lambda \Big\{D\Psi_t\big(\lambda u(t)\big)-H_u(t)\Big\}\bigg>_{X\times X^*}\\+
\bigg<\Big\{\lambda u(t)-D\Psi^*_t\big(H_u(t)\big)\Big\},\Big\{\frac{d\sigma}{dt}(t)+D\Lambda_t\big(u(t)\big)\cdot \big(h(t)\big)\Big\}\bigg>_{X\times X^*}\Bigg\}\,dt\,,
\end{multline}
where $\sigma(t):=I\cdot(h(t))$ with $I:=\widetilde T\circ T:\,X\to X^*$ and we denote
\begin{equation}\label{nolkj91defHHH}
H_u(t):=-\frac{dv}{dt}(t)-\Lambda_t\big(u(t)\big)\in X^*\quad\forall t\in(a,b)\,.
\end{equation}
\end{lemma}
\begin{proof}
First of all by \er{roststlambd} it is easy to deduce that
\begin{equation}\label{roststlambdosn}
\big\|\Lambda_t(x)\big\|_{X^*}\leq 2g\big(\|T\cdot
x\|_H\big)\,\Big(\|x\|_X^{q-1}+\mu^{\frac{q-1}{q}}(t)\Big)+\big\|\Lambda_t(0)\big\|_{X^*}\quad\forall
x\in X,\;\forall t\in[a,b]\,,
\end{equation}
for some nondecreasing function $g(s):[0,+\infty)\to(0+\infty)$ and
some nonnegative function $\mu(t)\in L^1(a,b;\R)$. Therefore, using
\er{rostst*}, for every $u\in\mathcal{R}_{q}(a,b)$ we obtain
\begin{multline}\label{roststlambdosnkjgggggggggggggg}
\Psi_t^*\bigg(-\frac{dv}{dt}(t)-\Lambda_t\big(u(t)\big)\bigg)\leq\\
L\bigg\{\Big\|\frac{dv}{dt}(t)\Big\|_{X^*}^{q^*}+
\Big(g\big(\|w(t)\|_H\big)\Big)^{q^*}\Big(\|u(t)\|^q_X+\mu(t)+\|\Lambda_t(0)\|_{X^*}^{q^*}\Big)+1\bigg\}\;\;\forall
t\in[a,b],
\end{multline}
for some constant $L>0$.
Since $w(t)\in L^\infty(a,b;H)$,
using above inequality and \er{rostst} gives that $J(u)<+\infty$.

 Next clearly for every $u,h\in\mathcal{R}_{q}(a,b)$ and every $s\in\R$ we have $(u+sh)\in\mathcal{R}_{q}(a,b)$ i.e. $\mathcal{R}_{q}(a,b)$ is a
linear space. Furthermore, observe that
\begin{equation}\label{vjhfhjfjhgvhhhjh}
\big\|T\cdot u(t)+s \,T\cdot h(t)\big\|_H\leq M\quad\quad\forall
s\in[-1,1],\;\forall t\in(a,b)\,,
\end{equation}
with $M>0$ independent on $t$ and $s$. We claim that for a.e.
$t\in[a,b]$,
\begin{equation}\label{weakcngrd}
\begin{split}
D\Psi^*_t\big(H_{(u+sh)}(t)\big)\rightharpoonup D\Psi^*_t\big(H_{u}(t)\big)\quad\text{weakly in}\;\;X\quad\text{as}\;\;s\to 0\,.
\end{split}
\end{equation}
Indeed by
\er{rostgrad} and
\er{rostgrad*} in Lemma \ref{Legendre}, for some
$\bar C_0,\bar C>0$
we have
\begin{equation}\label{roststt}
\|D\Psi_t(x)\|_{X^*}\leq \bar C_0\|x\|^{q-1}+\bar C_0\quad\forall x\in X,\;\forall t\in[a,b]\,,
\end{equation}
and
\begin{equation}\label{roststt*}
\|D\Psi_t^*(y)\|_{X}\leq \bar C\|y\|^{q^*-1}+\bar C\quad\forall y\in X^*,\;\forall t\in[a,b]\,.
\end{equation}
However, since every bounded sequence of elements of a reflexive
Banach space has a subsequence which converges weakly,  by
\er{roststlambdosnkjgggggggggggggg}, \er{rostst*} and \er{roststt*},
for a.e. fixed $t$ and for every sequence of real numbers $s_n\to
0$, up to a subsequence, we have
\begin{equation}\label{weakcngrdlll}
\begin{split}
D\Psi^*_t\big(H_{(u+s_nh)}(t)\big)\rightharpoonup x_0\quad\text{weakly in}\;\;X\quad\text{as}\;\;s_n\to 0\,.
\end{split}
\end{equation}
On the other hand, since
$\Psi^*_t$ is a convex function, by \er{conprmin}
we have
\begin{equation}\label{conprminergjfjhfhijjhbghgj}
\begin{split}
\Psi^*_t(y)\geq\Psi^*_t\big(H_{(u+s_n h)}(t)\big)+\bigg<D\Psi^*_t\big(H_{(u+s_n h)}(t)\big),y-H_{(u+s_nh)}(t)\bigg>_{X\times X^*}\quad\forall y\in X^*\,.
\end{split}
\end{equation}
Thus letting $s_n\to 0$ in \er{conprminergjfjhfhijjhbghgj} and using \er{weakcngrdlll} we deduce,
\begin{equation}\label{conprminergjfjhfhijjhbghgjpredlm}
\begin{split}
\Psi^*_t(y)\geq\Psi^*_t\big(H_{u}(t)\big)+\big<x_0,y-H_{u}(t)\big>_{X\times X^*}\quad\forall y\in X^*\,.
\end{split}
\end{equation}
Therefore,
\begin{equation}\label{conprminergjfjhfhijjhbghgjpredlmhfh}
\begin{split}
\Psi^*_t\big(H_{u}(t)\big)-\big<x_0,H_{u}(t)\big>_{X\times X^*}=\inf\Big\{\Psi^*_t(y)-\big<x_0,y\big>_{X\times X^*}:\;y\in X^*\Big\}\,.
\end{split}
\end{equation}
So
$\Psi^*_t\big(H_{u}(t)\big)$ is a minimizer to the problem in the r.h.s. of \er{conprminergjfjhfhijjhbghgjpredlmhfh} and thus satisfies
the corresponding Euler-Lagrange equation
$D\Psi^*_t\big(H_{u}(t)\big)=x_0$. Therefore, using \er{weakcngrdlll}, since $s_n\to 0$
was arbitrary sequence we deduce \er{weakcngrd}.

 Next clearly for a.e. fixed $t\in[a,b]$ we have
\begin{equation}\label{nolkj91hhhkkkk}
\lim\limits_{s\to 0}\frac{1}{s}\bigg\{\Psi_t\Big(\lambda\big(u(t)+sh(t)\big)\Big)-\Psi_t\big(\lambda u(t)\big)\bigg\}=\Big<h(t),\lambda D\Psi_t\big(\lambda u(t)\big)\Big>_{X\times X^*}\,,
\end{equation}
and
\begin{equation}\label{nolkj91hhhkkkklll***}
\lim\limits_{s\to 0}\frac{1}{s}\Big(H_{(u+s h)}(t)-H_u(t)\Big)=-\Big\{\frac{d\sigma}{dt}(t)+D\Lambda_t\big(u(t)\big)\cdot \big(h(t)\big)\Big\}\,,
\end{equation}
where the last limit is taken in the $X^*$-strong topology. Then in particular
\begin{multline}\label{nolkj91hhhkkkklllnnnnkkkhhkhkllppp***}
\lim\limits_{s\to 0}\frac{1}{s}\bigg\{\Big<u(t)+sh(t), H_{(u+s h)}(t)\Big>_{X\times X^*}-\Big<u(t),H_u(t)\Big>_{X\times X^*}\bigg\}\\=\Big<h(t), H_{u}(t)\Big>_{X\times X^*}
-\bigg<u(t),\Big\{\frac{d\sigma}{dt}(t)+D\Lambda_t\big(u(t)\big)\cdot \big(h(t)\big)\Big\}\bigg>_{X\times X^*}\,.
\end{multline}
Next
since
$\Psi^*_t$ is a convex functions, as before,
we have
\begin{equation}\label{conprminergjfjhfhijjhbghgjasskkk}
\begin{split}
\Psi^*_t\big(H_{(u+s h)}(t)\big)-\Psi^*_t\big(H_u(t)\big)\leq\bigg<D\Psi^*_t\big(H_{(u+sh)}(t)\big),H_{(u+sh)}(t)-H_u(t)\bigg>_{X\times X^*}\quad\forall y\in X^*\,,\\
\Psi^*_t\big(H_{(u+s h)}(t)\big)-\Psi^*_t\big(H_u(t)\big)\geq\bigg<D\Psi^*_t\big(H_{u}(t)\big),H_{(u+sh)}(t)-H_u(t)\bigg>_{X\times X^*}\quad\forall y\in X^*\,.
\end{split}
\end{equation}
Therefore, by \er{conprminergjfjhfhijjhbghgjasskkk}, \er{nolkj91hhhkkkklll***} and \er{weakcngrd} we deduce
\begin{multline}\label{nolkj91hhhkkkk***}
\lim\limits_{s\to 0}\frac{1}{s}\bigg\{\Psi^*_t\big(H_{(u+s h)}(t)\big)-\Psi^*_t\big(H_u(t)\big)\bigg\}=
-\bigg<D\Psi^*_t\big(H_u(t)\big),\Big\{\frac{d\sigma}{dt}(t)+D\Lambda_t\big(u(t)\big)\cdot \big(h(t)\big)\Big\}\bigg>_{X\times X^*}\,.
\end{multline}
On the other hand, using \er{roststt}, by \er{nolkj91hhhkkkk} and
the Lagrange Theorem from the elementary Calculus we deduce that for
every $t\in[a,b]$ and every $s\in[-1,1]$, there exists $\tau$ in the
interval with the endpoints $0$ and $s$, such that
\begin{multline}\label{nolkj91hhhkkkksled}
\bigg|\frac{1}{s}
\bigg\{\Psi_t\Big(\lambda\big(u(t)+sh(t)\big)\Big)-\Psi_t\big(\lambda
u(t)\big)\bigg\}\bigg|\leq\\ \|h(t)\|_X\cdot\bigg\|\lambda
D\Psi_t\Big(\lambda\big(u(t)+\tau h(t)\big)\Big)\bigg\|_{X^*} \leq
C\Big(\|h(t)\|_X^q+\|u(t)\|_X^q+1\Big)\,,
\end{multline}
for some constant $C>0$ independent on $t$ and $s$. Similarly, using
\er{roststlambd}, \er{roststlambdosn} and \er{vjhfhjfjhgvhhhjh},
from \er{nolkj91hhhkkkklllnnnnkkkhhkhkllppp***} we deduce
\begin{multline}\label{nolkj91hhhkkkklllnnnnkkkhhkhkllpppsled***}
\bigg|\frac{1}{s}\bigg\{\Big<u(t)+sh(t), H_{(u+s h)}(t)\Big>_{X\times X^*}-\Big<u(t),H_u(t)\Big>_{X\times X^*}\bigg\}\bigg|\leq\\
\|h(t)\|_X\cdot\big\|H_{u+\tau h}(t)\big\|_{X^*}+ \big\|u(t)+\tau
h(t)\big\|_X\cdot\bigg\|\frac{d\sigma}{dt}(t)+D\Lambda_t\Big(u(t)+\tau
h(t)\Big)\cdot \big(h(t)\big)\bigg\|_{X^*}\leq\\  \bar
C\Big(\|u(t)\|_X+\|h(t)\|_X\Big)\cdot\bigg(\big\|u(t)\big\|^{q-1}_{X}+\big\|h(t)\big\|^{q-1}_{X}+\Big\|\frac{dv}{dt}(t)\Big\|_{X^*}+
\Big\|\frac{d\sigma}{dt}(t)\Big\|_{X^*}+\big\|\Lambda_t(0)\big\|_{X^*}+\mu^{\frac{q-1}{q}}(t)\bigg)\\
\leq C_0
\bigg(\big\|u(t)\big\|^{q}_{X}+\big\|h(t)\big\|^{q}_{X}+\Big\|\frac{dv}{dt}(t)\Big\|^{q^*}_{X^*}+
\Big\|\frac{d\sigma}{dt}(t)\Big\|^{q^*}_{X^*}+\big\|\Lambda_t(0)\big\|_{X^*}^{q^*}+\mu(t)\bigg)\,,
\end{multline}
where the constants $\bar C,C_0>0$ are independent on $t$ and $s$.
Finally, in the same way, using \er{roststt*}, \er{roststlambd},
\er{roststlambdosnkjgggggggggggggg} and the fact that
$\Lambda_t(0)\in L^{q^*}(a,b;X^*)$, from \er{nolkj91hhhkkkk***} we
infer
\begin{multline}\label{nolkj91hhhkkkksled***}
\bigg|\frac{1}{s}\bigg\{\Psi^*_t\big(H_{(u+s h)}(t)\big)-\Psi^*_t\big(H_u(t)\big)\bigg\}\bigg|\leq
\big\|D\Psi^*_t\big(H_{u+\tau h}(t)\big)\big\|_X\cdot\Big\|\frac{d\sigma}{dt}(t)+D\Lambda_t\Big(u(t)+\tau h(t)\Big)\cdot \big(h(t)\big)\Big\|_{X^*}\leq\\
C\bigg(\big\|u(t)\big\|^{q-1}_{X}+\big\|h(t)\big\|^{q-1}_{X}+\Big\|\frac{dv}{dt}(t)\Big\|_{X^*}+
\Big\|\frac{d\sigma}{dt}(t)\Big\|_{X^*}+\big\|\Lambda_t(0)\big\|_X+\mu^{\frac{q-1}{q}}(t)+1\bigg)^{q^*-1}\times\\
\bigg(\big\|u(t)\big\|^{q-1}_{X}+\big\|h(t)\big\|^{q-1}_{X}+
\Big\|\frac{d\sigma}{dt}(t)\Big\|_{X^*}+\mu^{\frac{q-1}{q}}(t)\bigg)
\\ \leq C_1
\bigg(\big\|u(t)\big\|^{q}_{X}+\big\|h(t)\big\|^{q}_{X}+\Big\|\frac{dv}{dt}(t)\Big\|^{q^*}_{X^*}+
\Big\|\frac{d\sigma}{dt}(t)\Big\|^{q^*}_{X^*}+\tilde\mu(t)\bigg),
\end{multline}
where again the constant $C_1>0$ is independent on $t$ and $s$ and
$\tilde\mu(t)\in L^1(a,b;\R)$. However, by the definition of
$\mathcal{R}_{q}(a,b)$ we have $u,h\in L^q(a,b;X)$ and
$\frac{dv}{dt},\frac{d\sigma}{dt}\in L^{q^*}(a,b;X^*)$. Therefore,
using the Dominated Convergence Theorem, by \er{fyhfypppp},
\er{nolkj91hhhkkkk} and \er{nolkj91hhhkkkksled},
\er{nolkj91hhhkkkklllnnnnkkkhhkhkllppp***} and
\er{nolkj91hhhkkkklllnnnnkkkhhkhkllpppsled***},
\er{nolkj91hhhkkkk***} and \er{nolkj91hhhkkkksled***}, we deduce
\er{nolkj91}.
\end{proof}

\begin{theorem}\label{EulerLagrange}
Let $\{X,H,X^*\}$ be an evolution triple with the corresponding
inclusion operator $T\in \mathcal{L}(X;H)$, as it was defined in
Definition \ref{7bdf}, together with the corresponding operator
$\widetilde{T}\in \mathcal{L}(H;X^*)$, defined as in \er{tildet}.
Furthermore, let $a,b,q,\lambda\in\R$ be such that $a<b$, $q\geq 2$
and $\lambda\in\{0,1\}$.
Assume that $\Psi_t$ and $\Lambda_t$ satisfy all the conditions of
Definition \ref{defHkkkk}.
Furthermore let $w_0\in H$ and $\mathcal{R}_{q}(a,b)$ and $J$ be as in Definitions \ref{defH} and \ref{defHkkkk} respectively.
Consider the minimization
problem
\begin{equation}\label{hgffckaaq1new}
\inf\{J(u):\,u\in\mathcal{R}_{q}(a,b),\,w(a)=w_0\}\,,
\end{equation}
where $w(t):=T\cdot\big(u(t)\big)$. Then for
$u\in\mathcal{R}_{q}(a,b)$ such that $w(a)=w_0$ the following
 four statements
 are equivalent:
\begin{itemize}
\item[{\bf(a)}]
$u$ is a minimizer to \er{hgffckaaq1new}.
\item[{\bf (b)}]
$u$ is a critical point of \er{hgffckaaq1new} i.e. for an arbitrary
function $h(t)\in\mathcal{R}_{q}(a,b)$, such that $T\cdot h(a)=0$ we
have
\begin{equation}\label{nolkj}
\lim\limits_{s\to 0}\frac{J(u+s h)-J(u)}{s}=0\,.
\end{equation}
\item[{\bf (c)}]
$u$ is a solution to
\begin{equation}\label{uravnllllnew}
\frac{dv}{dt}(t)+\Lambda_t\big(u(t)\big)+D\Psi_t\big(\lambda
u(t)\big)=0\quad\text{for a.e.}\; t\in(a,b)\,,
\end{equation}
where
$v(t):=I\cdot \big(u(t)\big)=\widetilde T\cdot \big(w(t)\big)$ with
$I:=\widetilde T\circ T:\,X\to X^*$.
\item[{\bf (d)}] $J(u)=0$.
\end{itemize}
\end{theorem}
\begin{proof}
Assume that $u$ is a minimizer to \er{hgffckaaq1new}. Let
$h(t)\in\mathcal{R}_{q}(a,b)$ be an arbitrary function, such that
$T\cdot h(a)=0$. Here we assume that $T\cdot (h(t))$ is $H$-weakly
continuous on $[a,b]$. Then for every $s\in\R$
$\{u(t)+sh(t)\}\in\mathcal{R}_{q}(a,b)$ and
$T\cdot\big(u(a)+sh(a)\big)=w_0$. Therefore, since $u$ is a
minimizer we deduce that $f_h(s):=J(u+sh)\geq J(u)=f_h(0)$ for every
$s\in\R$. Thus by the elementary Calculus we must have $f'_h(0)=0$
(remember that, by Lemma \ref{EulerLagrange1}, $f_h$ is a
differentiable function). So we obtain \er{nolkj}.

Next assume that for some $u\in\mathcal{R}_{q}(a,b)$ such that
$w(a)=w_0$ we have \er{nolkj} i.e.
\begin{equation}\label{nolkjthjyt}
\lim\limits_{s\to 0}\frac{J(u+s h)-J(u)}{s}=0\,.
\end{equation}
for every $h(t)\in\mathcal{R}_{q}(a,b)$, such that $T\cdot h(a)=0$.
Then, by \er{nolkj91} in Lemma \ref{EulerLagrange1}, for every
$h(t)\in\mathcal{R}_{q}(a,b)$, such that $T\cdot h(a)=0$ we must
have
\begin{multline}\label{nolkj989hh}
\int\limits_a^b\Bigg\{\bigg<h(t),\lambda \Big\{D\Psi_t\big(\lambda u(t)\big)-H_u(t)\Big\}\bigg>_{X\times X^*}\\+
\bigg<\Big\{\lambda u(t)-D\Psi^*_t\big(H_u(t)\big)\Big\},\Big\{\frac{d\sigma}{dt}(t)+D\Lambda_t\big(u(t)\big)\cdot \big(h(t)\big)\Big\}\bigg>_{X\times X^*}\Bigg\}\,dt=0\,,
\end{multline}
where $\sigma(t):=I\cdot(h(t))$ with $I:=\widetilde T\circ T:\,X\to X^*$ and we, as before, denote
\begin{equation}\label{nolkj91defHHHsleddd}
H_u(t):=-\frac{dv}{dt}(t)-\Lambda_t\big(u(t)\big)\in X^*\quad\forall t\in(a,b)\,.
\end{equation}
Next as in \er{roststt} we have
\begin{equation}\label{roststtstttt}
\|D\Psi_t(x)\|_{X^*}\leq \bar C_0\|x\|^{q-1}+\bar C_0\quad\forall x\in X,\;\forall t\in[a,b]\,,
\end{equation}
and as in \er{roststt*} we have
\begin{equation}\label{roststtstttt*}
\|D\Psi_t^*(y)\|_{X}\leq \bar C\|y\|^{q^*-1}+\bar C\quad\forall y\in X^*,\;\forall t\in[a,b]\,,
\end{equation}
and since we have $\int_a^b\Psi_t^*\big(H_u(t)\big)dt<+\infty$, by \er{rostst*} we obtain
$H_u(t)\in L^{q^*}(a,b;X^*)$.
Therefore, if we set
\begin{equation}\label{Wuopredd778899jjk}
W_u(t):=\lambda u(t)-D\Psi^*_t\big(H_u(t)\big)\quad\quad\forall t\in[a,b]\,,
\end{equation}
then, since $u(t)\in L^q(a,b;X)$, we deduce
\begin{equation}\label{prinadlezhnosthhh}
W_u(t)\in L^q(a,b;X)\,.
\end{equation}
On the other hand, by Remark \ref{rmmmdd}, $x\in X$ satisfies
$x=D\Psi_t^*(y)$ if and only if $y\in X^*$ satisfies $y=D\Psi_t(x)$.
Therefore, by \er{nolkj989hh} we have
\begin{multline}\label{nolkj989}
\int\limits_a^b\Bigg(\bigg<\lambda h(t),\bigg\{D\Psi_t\Big(D\Psi^*_t\big(H_u(t)\big)+W_u(t)\Big)-D\Psi_t\Big(D\Psi^*_t\big(H_u(t)\big)\Big)\bigg\}\bigg>_{X\times X^*}+\\
\bigg<W_u(t),D\Lambda_t\big(u(t)\big)\cdot \big(h(t)\big)
\bigg>_{X\times X^*}+\bigg<W_u(t),\frac{d\sigma}{dt}(t)\bigg>_{X\times X^*}\Bigg)\,dt=0\,.
\end{multline}
Next for every $t\in[a,b]$
let $\big(D\Lambda_t\big(u(t)\big)\big)^*\in\mathcal{L}(X,X^*)$ be the adjoint to $D\Lambda_t\big(u(t)\big)\in\mathcal{L}(X,X^*)$ operator defined by
\begin{equation}\label{conjDlam}
\bigg<x_2,\big(D\Lambda_t\big(u(t)\big)\big)^*\cdot x_1
\bigg>_{X\times X^*}:=\bigg<x_1,D\Lambda_t\big(u(t)\big)\cdot x_2
\bigg>_{X\times X^*}\quad\forall\, x_1,x_2\in X\,.
\end{equation}
Define the strictly measurable function $P_u(t):(a,b)\to X^*$ by
\begin{equation}\label{sferytuj898}
P_u(t):=\lambda\bigg\{D\Psi_t\Big(D\Psi^*_t\big(H_u(t)\big)+W_u(t)\Big)-D\Psi_t\Big(D\Psi^*_t\big(H_u(t)\big)\Big)\bigg\}+\Big(D\Lambda_t\big(u(t)\big)\Big)^*\cdot \big(W_u(t)\big)\,.
\end{equation}
Then, since by \er{conjDlam} we have
\begin{multline}\label{lfnguggu9999}
\Big<x,P_u(t)\Big>_{X\times X^*}= \bigg<\lambda
x,\bigg\{D\Psi_t\Big(D\Psi^*_t\big(H_u(t)\big)+W_u(t)\Big)-D\Psi_t\Big(D\Psi^*_t\big(H_u(t)\big)\Big)\bigg\}\bigg>_{X\times
X^*}\\+ \bigg<W_u(t),D\Lambda_t\big(u(t)\big)\cdot x \bigg>_{X\times
X^*}\quad\quad\forall x\in X\,,
\end{multline}
 we deduce from \er{nolkj989},
\begin{equation}\label{nolkj989kdfsil}
\int\limits_a^b\Big<h(t),P_u(t)\Big>_{X\times X^*}dt+\int\limits_a^b\bigg<W_u(t),\frac{d\sigma}{dt}(t)\bigg>_{X\times X^*}dt=0\,.
\end{equation}
On the other hand, using the fact $T\cdot u(t)\in L^\infty(a,b;H)$,
by \er{lfnguggu9999}, \er{roststtstttt}, \er{roststtstttt*} and
\er{roststlambd} we deduce
$$\big\|P_u(t)\big\|_{X^*}\leq C\Big\{\big\|H_u(t)\big\|_{X^*}+\big\|W_u(t)\big\|^{q-1}_X+\big\|u(t)\big\|^{q-1}_X+\mu^{\frac{q-1}{q}}(t)\Big\}\,,$$
where the constant $C>0$ doesn't depend on $t$ and $\mu(t)\in
L^1(a,b;\R)$. Thus since $H_u(t)\in L^{q^*}(a,b;X^*)$, $u(t)\in
L^q(a,b;X)$ and $W_u(t)\in L^q(a,b;X)$ we infer
\begin{equation}\label{nolkj989kdfsilkukukuku}
P_u(t)\in L^{q^*}(a,b;X^*)\,.
\end{equation}
Next remember that we established \er{nolkj989kdfsil} for every
$h(t)\in\mathcal{R}_{q}(a,b)$, such that
$T\cdot h(a)=0$. In particular \er{nolkj989kdfsil} holds for $h(t)=\delta(t)$ where $\delta(t)\in C^1\big((a,b);X\big)$ satisfying $\supp \delta\subset\subset (a,b)$.
For such $\delta(t)$ we have
\begin{equation}\label{nolkj989kdfsilfgfh}
\int\limits_a^b\Big<\delta(t),P_u(t)\Big>_{X\times X^*}dt=-\int\limits_a^b\bigg<W_u(t),I\cdot\Big(\frac{d\delta}{dt}(t)\Big)\bigg>_{X\times X^*}dt
=-\int\limits_a^b\bigg<\frac{d\delta}{dt}(t), I\cdot\big(W_u(t)\big)\bigg>_{X\times X^*}dt\,.
\end{equation}
Thus, by \er{nolkj989kdfsilfgfh}, \er{nolkj989kdfsilkukukuku} and by Definition \ref{4bdf}, $I\cdot\big(W_u(t)\big)$ belongs to $W^{1,q^*}(a,b;X^*)$ and
\begin{equation}\label{proizvuggu}
\frac{d (I\cdot W_u)}{dt}(t)=P_u(t)\,.
\end{equation}
Thus since $W_u(t)\in L^q(a,b;X)$ and $P_u(t)\in L^{q^*}(a,b;X^*)$ we have
$W_u\in\mathcal{R}_{q}(a,b)$.
Then by Lemma \ref{lem2}, the function $L_u(t):[a,b]\to H$ defined by $L_u(t):=T\cdot \big(W_u(t)\big)$ belongs to $L^\infty(a,b;H)$ and,
up to a redefinition of $L_u(t)$ on a subset of $[a,b]$ of Lebesgue measure zero,  $L_u$ is $H$-weakly continuous, as it was stated in
Corollary \ref{vbnhjjmcor}.
Moreover, by the same Corollary for every $a\leq \alpha<\beta\leq b$ and for every
$\delta(t)\in C^1\big([a,b];X\big)$ we will have
\begin{equation}\label{eqmultcorcor}
\int\limits_\alpha^\beta\bigg\{\Big<\delta(t), P_u(t)\Big>_{X\times X^*}+\Big<\frac{d\delta}{dt}(t), I\cdot\big(W_u(t)\big)\Big>_{X\times X^*}\bigg\}dt=
\big<T\cdot\delta(\beta),L_u(\beta)\big>_{H\times H}-\big<T\cdot\delta(\alpha),L_u(\alpha)\big>_{H\times H}\,.
 \end{equation}
Thus in particular for every $h_0(t)\in C^1\big([a,b];X\big)$ such that $h_0(a)=0$ we have
\begin{equation}\label{eqmultcorcorcor}
\int\limits_a^b\Big<h_0(t), P_u(t)\Big>_{X\times X^*}dt=-\int\limits_a^b\Big<\frac{dh_0}{dt}(t), I\cdot\big(W_u(t)\big)\Big>_{X\times X^*}dt+
\big<T\cdot h_0(b),L_u(b)\big>_{H\times H}\,.
\end{equation}
However, inserting $h:=h_0$ into \er{nolkj989kdfsil} gives
\begin{equation}\label{nolkj989kdfsiln}
\int\limits_a^b\Big<h_0(t),P_u(t)\Big>_{X\times X^*}dt=-\int\limits_a^b\bigg<W_u(t),\frac{d(I\cdot h_0)}{dt}(t)\bigg>_{X\times X^*}dt=
-\int\limits_a^b\bigg<\frac{dh_0}{dt}(t),I\cdot \big(W_u(t)\big)\bigg>_{X\times X^*}dt\,.
\end{equation}
Comparing \er{eqmultcorcorcor} with \er{nolkj989kdfsiln} we obtain
\begin{equation}\label{eqnmmxx}
\big<T\cdot h_0(b),L_u(b)\big>_{H\times H}=0\,.
\end{equation}
In particular we can test \er{eqnmmxx} with $h_0(t):=(t-a)x$ for arbitrary $x\in X$. Then we obtain
\begin{equation}\label{eqnmmxx999}
\big<Tx,L_u(b)\big>_{H\times H}=0\quad\forall x\in X\,,
\end{equation}
and since $T$ has dense image in $H$ we finally get
\begin{equation}\label{eqnmmxx9998899}
L_u(b)=0\,.
\end{equation}
Thus, since $L_u(t):=T\cdot \big(W_u(t)\big)$, using \er{proizvuggu}, \er{eqnmmxx9998899} and
the fact that $W_u\in\mathcal{R}_{q}(a,b)$, by Lemma \ref{lem2} we deduce
\begin{equation}\label{eqbbbkl}
\int_\alpha^b\Big<W_u(t),P_u(t)\Big>_{X\times X^*}dt=
-\frac{1}{2}\|L_u(\alpha)\|_H^2\quad\forall \alpha\in[a,b]\,.
\end{equation}
Then plugging \er{lfnguggu9999} into \er{eqbbbkl}, by \er{Wuopredd778899jjk} we infer
\begin{multline}\label{eqbbbklhtgjy}
-\frac{1}{2}\Big\|T\cdot \big(W_u(\alpha)\big)\Big\|_H^2=\\ \int\limits_\alpha^b\Bigg(\bigg<W_u(t),\lambda \bigg\{D\Psi_t
\Big(D\Psi^*_t\big(H_u(t)\big)+W_u(t)\Big)-D\Psi_t\Big(D\Psi^*_t\big(H_u(t)\big)\Big)\bigg\}+
D\Lambda_t\big(u(t)\big)\cdot \big(W_u(t)\big)\bigg>_{X\times X^*}\Bigg)dt
\\=\int\limits_\alpha^b\Bigg(\bigg<W_u(t),\lambda \bigg\{D\Psi_t
\big(\lambda u(t)\big)-D\Psi_t\Big(\lambda u(t)-W_u(t)\Big)\bigg\}+
D\Lambda_t\big(u(t)\big)\cdot \big(W_u(t)\big)\bigg>_{X\times X^*}\Bigg)dt\quad\forall \alpha\in[a,b]\,.
\end{multline}
Next, by the condition \er{Monotone} in the Definition \ref{defHkkkk} we have
\begin{multline}\label{Monotone1}
\bigg<h,\lambda \Big\{D\Psi_t\big(\lambda x\big)-D\Psi_t(\lambda
x-h)\Big\}+D\Lambda_t(x)\cdot h\bigg>_{X\times X^*}\geq -\hat
g\big(\|T\cdot x\|_H\big)\Big(\|x\|^q_X+\hat\mu(t)\Big)\,\|T\cdot h\|^{2}_H\\
\forall x,h\in X,\;\forall t\in[a,b]\,,
\end{multline}
for some non-decreasing function $\hat g(s):[0+\infty)\to
(0,+\infty)$ and some nonnegative function $\hat\mu(t)\in
L^1(a,b;\R)$.
%
%
%
%
%
%
%
%
%
%
Therefore, using \er{eqbbbklhtgjy}, we deduce
\begin{equation}\label{Monotone1288988999ll88899}
\frac{1}{2}\Big\|T\cdot \big(W_u(\alpha)\big)\Big\|_H^2
\leq \int_\alpha^b \hat
g\Big(\big\|T\cdot\big(u(t)\big)\big\|_H\Big)\Big(\big\|u(t)\big\|_X^q+\hat\mu(t)\Big)\Big\|T\cdot
\big(W_u(t)\big)\Big\|^{2}_Hdt\quad\forall \alpha\in[a,b]\,.
\end{equation}
In particular, since $T\cdot u\in L^\infty(a,b;H)$ we clearly obtain
\begin{equation}\label{Monotone1288988999ll88899uuu89}
\frac{1}{2}\Big\|T\cdot \big(W_u(t)\big)\Big\|_H^2 \leq
K\int_t^b\Big(\big\|u(s)\big\|_X^q+\hat\mu(s)\Big) \Big\|T\cdot
\big(W_u(s)\big)\Big\|^{2}_Hds\quad\forall t\in[a,b]\,,
\end{equation}
for some $K>0$ independent on $t$.
%
%
%
%
Next define the functions
$g(t):=\Big(\big\|u(t)\big\|_X^q+\hat\mu(t)\Big)$ and
$h(t):=g(t)\Big\|T\cdot \big(W_u(t)\big)\Big\|_H^2$. Then since
$u(t)\in L^q(a,b;X)$ and $T\cdot \big(W_u(t)\big)\in
L^\infty(a,b;H)$ we clearly have $g(t),h(t)\in L^1(a,b;\R)$.
Moreover, by \er{Monotone1288988999ll88899uuu89}, we infer
\begin{equation}\label{hjkydddy99}
h(t)\leq 2Kg(t)\int_t^b h(s)\,ds\quad\quad\text{for
a.e.}\;\,t\in[a,b]\,.
\end{equation}
On the other hand, the function
$\gamma(t):=e^{-2K\int_t^bg(s)ds}\cdot\int_t^b h(s)\,ds\in
W^{1,1}(a,b;\R)$ and by \er{hjkydddy99}, we obtain
$$\frac{d}{dt}\bigg(e^{-2K\int_t^bg(s)ds}\cdot\int_t^b h(s)\,ds\bigg)\geq 0\quad\quad\text{for
a.e.}\;\,t\in[a,b]\,.$$ Therefore,
$$e^{-2K\int_a^bg(s)ds}\cdot\int_a^b h(s)\,ds\leq e^{-2K\int_b^bg(s)ds}\cdot\int_b^b h(s)\,ds=0\,.$$
Therefore, since $h(t)\geq 0$ we obtain $h(t)=0$ for a.e.
$t\in[a,b]$.
%
%
%
%
Thus $\big\|T\cdot \big(W_u(t)\big)\big\|_H^2=0$ for a.e.
$t\in[a,b]$. Therefore, since $T$ is injective, by the definition of
$W_u$ in \er{Wuopredd778899jjk} we have $\lambda
u(t)=D\Psi^*_t\big(H_u(t)\big)$ for a.e. $t\in[a,b]$. I.e.
\begin{equation}\label{Wuopredd778899jjkkfhf}
D\Psi_t\big(\lambda
u(t)\big)=H_u(t)=-\frac{dv}{dt}(t)-\Lambda_t\big(u(t)\big)\quad\quad\text{for
a.e.}\;\,t\in[a,b]\,,
\end{equation}
So $u$ is a solution to \er{uravnllllnew}.

 Finally, if $u$ is a solution to \er{uravnllllnew} then
by Remark \ref{rmmmdd} we have $J(u)=0$. Moreover, by Remark
\ref{rmmmdd} we always have $J(\cdot)\geq 0$. Thus if we have
$J(u)=0$ then trivially $u$ is a minimizer to \er{hgffckaaq1new}.
\end{proof}

The following proposition provides uniqueness of the solution.
\begin{proposition}\label{gftufuybhjkojutydrdkk}
Let $\{X,H,X^*\}$ be an evolution triple with the corresponding
inclusion operator $T\in \mathcal{L}(X;H)$, as it was defined in
Definition \ref{7bdf}, together with the corresponding operator
$\widetilde{T}\in \mathcal{L}(H;X^*)$, defined as in \er{tildet}.
Furthermore, let $a,b,q\in\R$ be such that $a<b$, $q\geq 2$ and
$\lambda\in\{0,1\}$.
Assume that $\Psi_t$ and $\Lambda_t$ satisfy all the conditions of
Definition \ref{defHkkkk}.
Then for every $w_0\in H$ there exists at most one function $u(t)\in
L^q(a,b;X)$, such that $w(t):=T\cdot\big(u(t)\big)\in
L^\infty(a,b;H)$, $v(t):=\widetilde T\cdot\big(w(t)\big)=\widetilde
T\circ T\big(u(t)\big)\in W^{1,q^*}(a,b; X^*)$ ,where
$q^*:=q/(q-1)$, and $u(t)$ is a solution to
\begin{equation}\label{uravnllllgsnachlklimjjjfffjjfff}
\begin{cases}\frac{d v}{dt}(t)+\Lambda_t\big(u(t)\big)+D\Psi_t\big(\lambda u(t)\big)=0\quad\text{for a.e.}\; t\in(a,b)\,,\\
w(a)=w_0\,.
\end{cases}
\end{equation}
\end{proposition}
\begin{proof}
Let $w_0\in H$ and let $\hat u(t),\tilde u(t)\in L^q(a,b;X)$ be such
that $\hat w(t):=T\cdot\big(\hat u(t)\big)\in L^\infty(a,b;H)$,
$\tilde w(t):=T\cdot\big(\tilde u(t)\big)\in L^\infty(a,b;H)$, $\hat
v(t):=\widetilde T\cdot\big(\hat w(t)\big)=\widetilde T\circ
T\big(\hat u(t)\big)\in W^{1,q^*}(a,b; X^*)$, $\tilde
v(t):=\widetilde T\cdot\big(\tilde w(t)\big)=\widetilde T\circ
T\big(\tilde u(t)\big)\in W^{1,q^*}(a,b; X^*)$, where
$q^*:=q/(q-1)$, and $\hat u(t), \tilde u(t)$ are both solutions to
\er{uravnllllgsnachlklimjjjfffjjfff}. Then the function
$\underline{u}(t):=\big(\hat u(t)-\tilde u(t)\big)\in L^q(a,b;X)$ is
such that $\underline{w}(t):=T\cdot\big(\underline{u}(t)\big)\in
L^\infty(a,b;H)$, $\underline{v}(t):=\widetilde
T\cdot\big(\underline{w}(t)\big)=\widetilde T\circ
T\big(\underline{u}(t)\big)\in W^{1,q^*}(a,b; X^*)$,
$\underline{w}(a)=0$ and by \er{uravnllllgsnachlklimjjjfffjjfff},
and the facts that $\lambda\in\{0,1\}$ and $\Psi_t(0)=0$, we obtain
\begin{multline}\label{uravnllllgsnachlklimjjjghtguffssff}
\frac{d\underline{v}}{dt}(t)+ \Big\{\Lambda_t\big(\hat
u(t)\big)-\Lambda_t\big(\tilde
u(t)\big)\Big\}+
\Big\{D\Psi_t\big(\lambda\hat u(t)\big)-D\Psi_t\big(\lambda\tilde
u(t)\big)\Big\}=0\quad\quad \text{for a.e.}\; t\in(a,b)\,.
\end{multline}
Thus plugging \er{Monotone}
into
\er{uravnllllgsnachlklimjjjghtguffssff} we deduce
\begin{equation}\label{uravnllllgsnachlklimjjjghtguyjukffssff}
\int\limits_a^\tau\Bigg\{\bigg<\underline{u}(t),\frac{d\underline{v}}{dt}(t)\bigg>_{X\times
X^*}-
\bar\gamma(t)\|\underline{w}(t)\|^{2}_H
\Bigg\}dt\leq 0\quad \text{for every}\; \tau\in[a,b]\,,
\end{equation}
where $\bar\gamma(t)\in L^1\big(a,b;[0,+\infty)\big)$. On the other
hand, since $\underline{w}(a)=0$, by Lemma \ref{lem2} we have
\begin{equation*}
\int_a^\tau\Big<\underline{u}(t),\frac{d
\underline{v}}{dt}(t)\Big>_{X\times
X^*}dt=\frac{1}{2}\big\|\underline{w}(\tau)\big\|_H^2\quad \text{for
every}\; \tau\in[a,b]\,.
\end{equation*}
Thus, inserting it into \er{uravnllllgsnachlklimjjjghtguyjukffssff},
we deduce
%
%
%
%
%
%
%
%
\begin{equation}\label{uravnllllgsnachlklimjjjghtguyjukfhhhhjkghkfjfjfjkffssff}
\big\|\underline{w}(\tau)\big\|_H^2\leq
C\int_a^\tau\bar\gamma(t)\big\|\underline{w}(t)\big\|^2_H dt\quad
\text{for every}\; \tau\in[a,b]\,.
\end{equation}
Therefore, exactly as before in the end of the proof of Theorem
\ref{EulerLagrange}, by
\er{uravnllllgsnachlklimjjjghtguyjukfhhhhjkghkfjfjfjkffssff} we
deduce $\underline{w}(t)=0$ for a.e. $t\in[a,b]$ and since $T$ is an
injective operator we have $\hat u(t)=\tilde u(t)$ for a.e.
$t\in[a,b]$. This completes the proof.
\end{proof}

\begin{definition}\label{defHkkkkgnew}
Let $\{X,H,X^*\}$ be an evolution triple with the corresponding
inclusion operator $T\in \mathcal{L}(X;H)$, as it was defined in
Definition \ref{7bdf}, together with the corresponding operator
$\widetilde{T}\in \mathcal{L}(H;X^*)$, defined as in \er{tildet}.
Furthermore, let $a,b,q\in\R$ be s.t. $a<b$ and $q\geq 2$. Next
assume that $\Psi_t$ and $\Lambda_t$ satisfy all the conditions of
Definition \ref{defHkkkk} together with the assumption $\lambda=1$.
Moreover, assume that
$\Psi_t$ and $\Lambda_t$ satisfy the following positivity condition:
%
%
%
%
\begin{multline}\label{Monotonegnewhghghghgh}
\Psi_t(x)+\Big<x,\Lambda_t(x)\Big>_{X\times X^*}\geq \frac{1}{\bar
C}\,\|x\|^q_X -\bar\mu(t)
\Big(\|T\cdot x\|^{2}_H+1\Big)
\quad\quad\forall x\in X,\;\forall t\in[a,b],
\end{multline}
where $\bar C>0$ is some constant and $\bar\mu(t)\in L^1(a,b;\R)$ is
a fixed nonnegative function.
Furthermore, assume that
\begin{equation}\label{jgkjgklhklhjhiuhyh}
\Lambda_t(x)=A_t\big(S\cdot x\big)+\Theta_t(x)\quad\quad\forall\,
x\in X,\;\forall\, t\in[a,b],
\end{equation}
where $Z$ is a Banach space, $S:X\to Z$ is a compact operator, for
every $t\in[a,b]$ $A_t(z):Z\to X^*$ and $\Theta_t(x):X\to X^*$ are
functions, such that $A_t$ is strongly Borel on the pair of
variables $(z,t)$ and G\^{a}teaux differentiable at every $z\in Z$,
$\Theta_t$ is strongly Borel on the pair of variables $(x,t)$ and
G\^{a}teaux differentiable at every $x\in X$,
$\,\Theta_t(0),A_t(0)\in L^{q^*}\big(a,b;X^*\big)$ and the
derivatives of $A_t$ and $\Theta_t$ satisfy the growth conditions
\begin{equation}\label{roststlambdgnew}
\|D A_t(S\cdot x)\|_{\mathcal{L}(Z;X^*)}\leq g\big(\|T\cdot
x\|\big)\,\Big(\|x\|_X^{q-2}+\mu^{\frac{q-2}{q}}(t)\Big)\quad\forall
x\in X,\;\forall t\in[a,b]\,,
\end{equation}
\begin{equation}\label{roststlambdgnewjjjjj}
\|D\Theta_t(x)\|_{\mathcal{L}(X;X^*)}\leq g\big(\|T\cdot
x\|\big)\,\Big(\|x\|_X^{q-2}+\mu^{\frac{q-2}{q}}(t)\Big)\quad\forall
x\in X,\;\forall t\in[a,b]\,,
\end{equation}
for some nondecreasing function $g(s):[0,+\infty)\to (0+\infty)$ and
some nonnegative function $\mu(t)\in L^1(a,b;\R)$. Finally, assume
that for every sequence $\big\{x_n(t)\big\}_{n=1}^{+\infty}\subset
L^q(a,b;X)$ such that the sequence $\big\{(\widetilde{T}\circ
T)\cdot x_n(t)\big\}$ is bounded in $W^{1,q^*}(a,b;X^*)$ and
$x_n(t)\rightharpoonup x(t)$ weakly in $L^q(a,b;X)$ we have
\begin{itemize}
\item
$\Theta_t\big(x_n(t)\big)\rightharpoonup \Theta_t\big(x(t)\big)$
weakly in $L^{q^*}(a,b;X^*)$,
\item
$\liminf_{n\to+\infty}\int_a^b
\Big<x_n(t),\Theta_t\big(x_n(t)\big)\Big>_{X\times X^*}dt\geq
\int_a^b \Big<x(t),\Theta_t\big(x(t)\big)\Big>_{X\times X^*}dt$.
\end{itemize}
%
%
%
%
%
%
Next, as in \er{hgffck} with $\lambda=1$, let
$J_0(u):\mathcal{R}_{q}(a,b)\to\R$ (where $\mathcal{R}_{q}(a,b)$ was
defined in Definition \ref{defH}) be defined by
\begin{multline}\label{hgffckgnew}
J_0(u):=\frac{1}{2}\Big(\|w(b)\|_H^2-\|w(a)\|_H^2\Big)+\\
\int\limits_a^b\Bigg\{\Psi_t\big(u(t)\big)+\Psi_t^*\bigg(-\frac{d
v}{dt}(t)-\Lambda_t\big(u(t)\big)\bigg)+
\Big<u(t),\Lambda_t\big(u(t)\big)\Big>_{X\times X^*}\Bigg\}dt\,,
\end{multline}
where $w(t):=T\cdot\big(u(t)\big)$, $v(t):=I\cdot
\big(u(t)\big)=\widetilde T\cdot \big(w(t)\big)$ with $I:=\widetilde
T\circ T:\,X\to X^*$ and we assume that $w(t)$ is $H$-weakly
continuous on $[a,b]$, as it was stated in Corollary
\ref{vbnhjjmcor}. Moreover, for every $w_0\in H$ consider the
minimization problem
\begin{equation}\label{hgffckaaq1gnew}
\inf\Big\{J_0(u):\, u\in\mathcal{R}_{q}(a,b),\,w(a)=w_0\Big\}\,.
\end{equation}
\end{definition}
\begin{remark}\label{biughiuhiyhiuh}
As before, we can rewrite the definition of $J_0$ in \er{hgffckgnew}
by
\begin{multline}\label{fyhfyppppgnew}
J_0(u):=
\int\limits_a^b\Bigg\{\Psi_t\big(u(t)\big)+\Psi_t^*\bigg(-\frac{d
v}{dt}(t)-\Lambda_t\big(u(t)\big)\bigg)+ \bigg<u(t),\frac{d
v}{dt}(t)+\Lambda_t\big(u(t)\big)\bigg>_{X\times X^*}\Bigg\}dt\,.
\end{multline}
\end{remark}
\begin{proposition}\label{premainnew}
Let $J_0(u)$ be as in Definition \ref{defHkkkkgnew} and all the
conditions of Definitions \ref{defHkkkkgnew} are satisfied. Moreover
let $w_0\in H$ be such that $w_0=T\cdot u_0$ for some $u_0\in X$, or
more generally, $w_0\in H$ be such that
$\mathcal{A}_{w_0}:=\big\{u\in\mathcal{R}_{q}(a,b):\,w(a)=w_0\big\}\neq\emptyset$.
Then there exists a minimizer to \er{hgffckaaq1gnew}. In particular
there exists a unique solution to
\begin{equation}\label{uravnllllgsnachnew}
\begin{cases}\frac{d v}{dt}(t)+\Lambda_t\big(u(t)\big)+D\Psi_t\big(u(t)\big)=0\quad\text{for
a.e.}\; t\in(a,b)\,,\\ w(a)=w_0\,,\end{cases}
\end{equation}
where $w(t):=T\cdot\big(u(t)\big)$, $v(t):=I\cdot
\big(u(t)\big)=\widetilde T\cdot \big(w(t)\big)$ with $I:=\widetilde
T\circ T:\,X\to X^*$ and we assume that $w(t)$ is $H$-weakly
continuous on $[a,b]$, as it was stated in Corollary
\ref{vbnhjjmcor}.
\end{proposition}
\begin{proof}
First of all we would like to note that in the case, $w_0=T\cdot
u_0$ for some $u_0\in X$, the set
$$\mathcal{A}_{w_0}:=\big\{u\in\mathcal{R}_{q}(a,b):\,w(a)=w_0\big\}=\big\{u\in\mathcal{R}_{q}(a,b):\,T\cdot u(a)=T\cdot u_0\big\}$$ is not empty.
In particular the function $u_0(t)\equiv u_0$ belongs to
$\mathcal{A}_{w_0}$. Thus, in any case
$\mathcal{A}_{w_0}
\neq\emptyset$.
Next let
\begin{equation}\label{naimensheenew}
K:=\inf\limits_{\theta\in \mathcal{A}_{w_0}}J_{0}(\theta)\,.
\end{equation}
Then $K\geq 0$. Consider a minimizing sequence $\{u_n(t)\}\subset
\mathcal{A}_{w_0}$, i.e. a sequence such that
\begin{equation}\label{naimensheeminsenew}
\lim_{n\to\infty}J_{0}(u_n)=K\,.
\end{equation}
Set $\Upsilon_n(t):(a,b)\to\R$ by
\begin{multline}\label{novfnkiiionew}
\Upsilon_n(t):= \Psi_t\big(u_n(t)\big)+\Psi_t^*\bigg(-\frac{d
v_n}{dt}(t)-\Lambda_t\big(u_n(t)\big)\bigg)+ \Big<u_n(t),\frac{d
v_n}{dt}(t)+\Lambda_t\big(u_n(t)\big)\Big>_{X\times X^*}\,,
\end{multline}
where $w_n(t):=T\cdot\big(u_n(t)\big)$ and $v_n(t):=\widetilde
T\cdot \big(w_n(t)\big)$. Then by the definition of Legendre
transform we deduce that $\Upsilon_n(t)\geq 0$ for a.e. $t\in(a,b)$.
On the other hand, by \er{fyhfyppppgnew} we obtain
\begin{equation}\label{hgffckggj999999new}
\int_a^b\Upsilon_n(t)dt=J_{0}(u_n)\to K\quad\quad\text{as}\;\;n\to
+\infty\,.
\end{equation}
Therefore, by \er{hgffckggj999999new}
and the fact that
$\Upsilon_n(t)\geq 0$ we deduce that there exists a constant $C_0>0$
such that for every $n\in\mathbb{N}$ and $t\in[a,b]$ we have
\begin{equation}\label{hgffckggj999999guy77788888999999new}
\int_a^t\Upsilon_n(s)ds\leq C_0\,.
\end{equation}
However, since by Lemma \ref{lem2} we have
\begin{equation*}
\int_a^t\Big<u_n(s),\frac{d v_n}{dt}(s)\Big>_{X\times
X^*}ds=\frac{1}{2}\Big(\|w_n(t)\|_H^2-\|w_0\|_H^2\Big)\,,
\end{equation*}
plugging \er{novfnkiiionew} into
\er{hgffckggj999999guy77788888999999new} and using
\er{Monotonegnewhghghghgh}
gives for every $n\in\mathbb{N}$,
\begin{multline}\label{Monotonegghfhtgunew}
\int\limits_a^t\Bigg\{\frac{1}{\hat
C}\,\big\|u_n(s)\big\|^q_X+\Psi_s^*\bigg(-\frac{d
v_n}{dt}(s)-\Lambda_s\big(u_n(s)\big)\bigg)\Bigg\}ds
+\frac{1}{2}\|w_n(t)\|_H^2\\ \leq
\frac{1}{2}\|w_0\|_H^2
+\int_a^t\bar\mu(s)\Big(\|w_n(s)\|^{2}_H+1\Big)
ds
\quad\quad\forall t\in[a,b]\,,
\end{multline}
where $\hat C$ is some positive constant. Therefore, in particular,
for every $n$ we
have
\begin{equation}
\label{Monotonegghfhtgufbhjhkjnew} \frac{1}{\hat
C}\,\int_a^t\big\|u_n(s)\big\|^q_X ds +\frac{1}{2}\|w_n(t)\|_H^2\leq
\int_a^t\bar\mu(s)\big\|w_n(s)\big\|^{2}_H ds
+\tilde C\quad\quad\forall t\in[a,b]\,.
\end{equation}
where $\tilde C$ is a positive constant.
In particular we deduce
\begin{equation}\label{bjgjggg88888888889999new}
\|w_n(t)\|_H^2\leq C_2\int_a^t\bar\mu(s)\|w_n(s)\|^{2}_H
ds+C_2\quad\quad\forall t\in[a,b]\quad\forall n\in\mathbb{N}\,,
\end{equation}
where $C_2>0$ doesn't depend on $n$ and $t$. Then
\begin{multline}\label{bjgjggg88888888889999dgvg99new}
\frac{d}{dt}\Bigg\{\exp{\bigg(-\int_a^t
C_2\bar\mu(s)ds\bigg)}\cdot\int_a^t\bar\mu(s)\|w_n(s)\|^{2}_H
ds\Bigg\}\leq C_2\bar\mu(t)\exp{\bigg(-\int_a^t
C_2\bar\mu(s)ds\bigg)}\\ \text{for a.e}\;\,t\in[a,b]\quad\forall
n\in\mathbb{N}\,,
\end{multline}
and thus
\begin{multline}\label{bjgjggg88888888889999dgvg99mcv9999new}
\int_a^t\bar\mu(s)\|w_n(s)\|^{2}_H ds\leq C_2\exp{\bigg(\int_a^t
C_2\bar\mu(s)ds\bigg)}\cdot
\int_a^t\bar\mu(\tau)\exp{\bigg(-\int_a^\tau
C_2\bar\mu(s)ds\bigg)}d\tau\\
\leq C_2\exp{\bigg(C_2\int_a^b\bar\mu(s)ds\bigg)}\cdot
\int_a^b\bar\mu(\tau)d\tau
\quad\quad\forall
t\in[a,b]\quad\forall n\in\mathbb{N}\,.
\end{multline}
Therefore, by \er{bjgjggg88888888889999new} the sequence
$\{w_n(t)\}$ is bounded in $L^\infty(a,b;H)$. Moreover, returning to
\er{Monotonegghfhtgufbhjhkjnew}
we deduce that the sequence $\{u_n\}$ is bounded in $L^q(a,b;X)$. On
the other hand since, by \er{roststlambdgnew} and the fact that
$\{w_n(t)\}$ is bounded in $L^\infty(a,b;H)$ we have
$$\big\|A_t(S\cdot u_n(t))\big\|_{X^*}\leq
C\Big(\|u_n(t)\|_{X}^{q-1}+\mu^{\frac{q-1}{q}}(t)\Big)+\big\|A_t(0)\big\|_{X^*},$$
we deduce that $\big\{A_t(S\cdot u_n(t))\big\}$ is bounded in
$L^{q^*}(a,b;X^*)$. Moreover, by \er{roststlambdgnewjjjjj},
$\big\{\Theta_t\big(u(t)\big)\big\}$ is bounded in
$L^{q^*}(a,b;X^*)$. Therefore, by \er{Monotonegghfhtgunew}, using
the growth conditions in \er{rostst*} we infer that the sequence
$\big\{\frac{d v_n}{dt}(t)\big\}$ is bounded in $L^{q^*}(a,b;X^*)$.
So
\begin{equation}\label{boundnes77889999new}
\begin{cases}
\big\{u_n(t)\big\}\quad\text{is bounded in}\;\;L^q(a,b;X)\,,\\
\big\{\frac{d v_n}{dt}(t)\big\}\quad\text{is bounded in}\;\;L^{q^*}(a,b;X^*)\,,\\
\{w_n(t)\}\quad\text{is bounded in}\;\;L^\infty(a,b;H)\,.
\end{cases}
\end{equation}
On the other hand by Corollary \ref{vbnhjjmcor} as in \er{eqmultcor}
for every $a\leq \alpha<\beta\leq b$ and for every $\delta(t)\in
C^1\big([a,b];X\big)$ we have
\begin{equation}\label{eqmultcorcorhjjhjhonew}
\int\limits_\alpha^\beta\bigg\{\Big<\delta(t), \frac{d
v_n}{dt}(t)\Big>_{X\times X^*}+\Big<\frac{d\delta}{dt}(t),
v_n(t)\Big>_{X\times X^*}\bigg\}dt=
\big<T\cdot\delta(\beta),w_n(\beta)\big>_{H\times
H}-\big<T\cdot\delta(\alpha),w_n(\alpha)\big>_{H\times H}\,.
\end{equation}
However, since $S$ is a compact operator, by
\er{boundnes77889999new} and Lemma \ref{ComTem1} we obtain that, up
to a subsequence,
\begin{equation}\label{boundnes77889999shodimostnew}
\begin{cases}
u_n(t)\rightharpoonup u(t)\quad\text{weakly in}\;\;L^q(a,b;X)\,,\\
\frac{d v_n}{dt}(t)\rightharpoonup\zeta(t)\quad\text{weakly in}\;\;L^{q^*}(a,b;X^*)\,,\\
w_n(t)\rightharpoonup w(t)\quad\text{weakly in}\;\;L^q(a,b;H)
\,,\\
v_n(t)\rightharpoonup v(t)\quad\text{weakly in}\;\;L^q(a,b;X^*)\,,
\\
S\cdot u_n(t)\to S\cdot u(t)\quad\text{strongly in}\;\;L^q(a,b;Z)\,,
\end{cases}
\end{equation}
where $w(t):=T\cdot\big(u(t)\big)$ and $v(t):=\widetilde T\cdot
\big(w(t)\big)$. In particular, by \er{boundnes77889999shodimostnew}
and \er{eqmultcorcorhjjhjhonew} for every $a\leq \alpha<\beta\leq b$
and for every $\delta(t)\in C^1\big([a,b];X\big)$ we have
\begin{equation}\label{eqmultcorcorhjjhjholimnew}
\int\limits_\alpha^\beta\bigg\{\big<\delta(t),
\zeta(t)\big>_{X\times X^*}+\Big<\frac{d\delta}{dt}(t),
v(t)\Big>_{X\times X^*}\bigg\}dt=\lim\limits_{n\to +\infty}\bigg\{
\big<T\cdot\delta(\beta),w_n(\beta)\big>_{H\times
H}-\big<T\cdot\delta(\alpha),w_n(\alpha)\big>_{H\times H}\bigg\}\,.
\end{equation}
Thus in particular for every $\delta(t)\in C^1_c\big((a,b);X\big)$
we have
\begin{equation}\label{eqmultcorcorhjjhjholimparthhhnew}
\int\limits_a^b\bigg\{\big<\delta(t), \zeta(t)\big>_{X\times
X^*}+\Big<\frac{d\delta}{dt}(t), v(t)\Big>_{X\times
X^*}\bigg\}dt=0\,.
\end{equation}
Therefore, by the definition, $v(t)\in W^{1,q^*}(a,b;X^*)$ and
\begin{equation}\label{eqmultcorcorhjjhjholimpartnew}
\frac{d v}{dt}(t)=\zeta(t)\quad\quad\text{for a.e.}\;\;t\in(a,b)\,.
\end{equation}
Then, as before by Corollary \ref{vbnhjjmcor}, $w(t)$ is $H$-weakly
continuous on $[a,b]$ and for every $a\leq \alpha<\beta\leq b$ and
for every $\delta(t)\in C^1\big([a,b];X\big)$ we have
\begin{equation}\label{eqmultcorcorhjjhjholimlimnew}
\int\limits_\alpha^\beta\bigg\{\big<\delta(t),
\zeta(t)\big>_{X\times X^*}+\Big<\frac{d\delta}{dt}(t),
v(t)\Big>_{X\times X^*}\bigg\}dt=
\big<T\cdot\delta(\beta),w(\beta)\big>_{H\times
H}-\big<T\cdot\delta(\alpha),w(\alpha)\big>_{H\times H}\,.
\end{equation}
Plugging \er{eqmultcorcorhjjhjholimlimnew} into
\er{eqmultcorcorhjjhjholimnew}, for every $a\leq \alpha<\beta\leq b$
and for every $\delta(t)\in C^1\big([a,b];X\big)$ we obtain
\begin{equation}\label{eqmultcorcorhjjhjholimlimsravnew}
\lim\limits_{n\to +\infty}\bigg\{
\big<T\cdot\delta(\beta),w_n(\beta)\big>_{H\times
H}-\big<T\cdot\delta(\alpha),w_n(\alpha)\big>_{H\times H}\bigg\}=
\big<T\cdot\delta(\beta),w(\beta)\big>_{H\times
H}-\big<T\cdot\delta(\alpha),w(\alpha)\big>_{H\times H}\,.
\end{equation}
In particular for every $h\in X$
\begin{equation}\label{eqmultcorcorhjjhjholimlimsravdgdfgnew}
\lim\limits_{n\to +\infty} \big<T\cdot h,w_n(t)\big>_{H\times H}=
\big<T\cdot h,w(t)\big>_{H\times H}\quad\quad\forall t\in[a,b]\,.
\end{equation}
Therefore, since by \er{boundnes77889999new}, $\{w_n\}$ is bounded
in $L^\infty(a,b;H)$ and since the image of $T$ is dense in $H$,
using \er{eqmultcorcorhjjhjholimlimsravdgdfgnew} we deduce
\begin{equation}\label{eqmultcorcorhjjhjholimlimsravdgdfgfinnew}
w_n(t)\rightharpoonup w(t)\quad\text{weakly in}\;\;H\quad\forall
t\in[a,b]\,.
\end{equation}
In particular, since $w_n(a)=w_0$ we obtain that $w(a)=w_0$ and so
$u(t)$ belongs to
$\mathcal{A}_{w_0}=\big\{\psi\in\mathcal{R}_{q}(a,b):\,T\cdot\psi(a)=w_0\big\}$.
On the other hand by \er{boundnes77889999shodimostnew},
\er{boundnes77889999new} and \er{roststlambdgnew} we deduce that
$A_t\big(S\cdot u_n(t)\big)\to A_t\big(S\cdot u(t)\big)$ strongly in
$L^{q^*}(a,b;X^*)$.
Moreover, by \er{boundnes77889999shodimostnew} and given properties
of $\Theta_t$, we deduce that
$\Theta_t\big(u_n(t)\big)\rightharpoonup\Theta_t\big(u(t)\big)$
weakly in $L^{q^*}(a,b;X^*)$ Therefore, by
\er{boundnes77889999shodimostnew},
\er{eqmultcorcorhjjhjholimpartnew} and the facts, that we
established above, we obtain
\begin{equation}\label{boundnes77889999shodimostfjklgknew}
\begin{cases}
u_n(t)\rightharpoonup u(t)\quad\text{weakly in}\;\;L^q(a,b;X)\,,\\
\frac{d v_n}{dt}(t)\rightharpoonup\frac{d v}{dt}(t)
\quad\text{weakly in}\;\;L^{q^*}(a,b;X^*)\,,\\
\Theta_t\big(u_n(t)\big)\rightharpoonup \Theta_t\big(u(t)\big)
\quad\text{weakly in}\;\;L^{q^*}(a,b;X^*)\,,\\
A_t\big(S\cdot u_n(t)\big)\to A_t\big(S\cdot u(t)\big)
\quad\text{strongly in}\;\;L^{q^*}(a,b;X^*)\,,\\
w_n(b)\rightharpoonup w(b)\quad\text{weakly in}\;\;H\,,\\
w_n(a)=w(a)=w_0\,.
\end{cases}
\end{equation}
On the other hand, by the definition of $J_0$ in \er{hgffckgnew} and
by \er{naimensheeminsenew} we obtain
\begin{multline}\label{hgffckglimggknew}
K=\lim_{n\to\infty}J_{0}(u_n)=\lim_{n\to\infty}\Bigg(\frac{1}{2}\Big(\|w_n(b)\|_H^2-\|w_0\|_H^2\Big)+\\
\int\limits_a^b\Bigg\{\Psi_t\big(u_n(t)\big)+\Psi_t^*\bigg(-\frac{dv_n}{dt}(t)-\Theta_t\big(u_n(t)\big)-A_t\big(S\cdot u_n(t)\big)\bigg)+\\
\Big<u_n(t),\Theta_t\big(u_n(t)\big)\Big>_{x\times
X^*}+\Big<u_n(t),A_t\big(S\cdot u_n(t)\big)\Big>_{X\times
X^*}\Bigg\}dt\Bigg)\,.
\end{multline}
However, the functions $\Psi_t$ and $\Psi^*_t$ are convex and
G\^{a}teaux differentiable for every $t$. Moreover, they satisfy the
growth conditions \er{rostst}, \er{rostst*} and the corresponding
conditions on their derivatives as stated in \er{rostgrad} and
\er{rostgrad*} in Lemma \ref{Legendre}. Thus
$P(x):=\int_a^b\Psi_t(x(t))dt$ and $Q(h):=\int_a^b\Psi^*_t(h(t))dt$
are convex and G\^{a}teaux differentiable functions on $L^q(a,b;X)$
and $L^{q^*}(a,b;X^*)$ respectively and therefore, $P(x)$ and $Q(h)$
are weakly lower semicontinuous functions on $L^q(a,b;X)$ and
$L^{q^*}(a,b;X^*)$ respectively. Moreover,
by \er{boundnes77889999shodimostnew} and the given properties of
$\Theta_t$, we infer
\begin{equation}\label{vhgjkhkgugkjjk}
\liminf_{n\to+\infty}\int_a^b\Big<u_n(t),\Theta_t\big(u_n(t)\big)\Big>_{X\times
X^*}dt\geq\int_a^b\Big<u(t),\Theta_t\big(u(t)\big)\Big>_{X\times
X^*}dt\,.
\end{equation}
Therefore, using \er{boundnes77889999shodimostfjklgknew},
\er{hgffckglimggknew} and \er{vhgjkhkgugkjjk} we finally obtain
\begin{multline}\label{hgffckglimggklimnew}
\inf\limits_{\theta\in
\mathcal{A}_{w_0}}J_{0}(\theta)=K\geq\frac{1}{2}\Big(\|w(b)\|_H^2-\|w_0\|_H^2\Big)+
\int\limits_a^b\Bigg\{\Psi_t\big(u(t)\big)+\Psi_t^*\bigg(-\frac{d
v}{dt}(t)-\Theta_t\big(u(t)\big)-A_t\big(S\cdot u(t)\big)\bigg)\\+
\Big<u(t),\Theta_t\big(u(t)\big)\Big>_{X\times
X^*}+\Big<u(t),A_t\big(S\cdot u(t)\big)\Big>_{X\times
X^*}\Bigg\}dt=J_0(u)\,.
\end{multline}
Thus $u$ is a minimizer to \er{hgffckaaq1gnew}. Moreover, all the
conditions of Theorem \ref{EulerLagrange} are satisfied and thus $u$
must satisfy $J_0(u)=0$ and \er{uravnllllgsnachnew}. Finally, by
Proposition \ref{gftufuybhjkojutydrdkk}, the solution to
\er{uravnllllgsnachnew} is unique.
\end{proof}


As an important particular case of Proposition \ref{premainnew} we
have following statement:
\begin{theorem}\label{THSld}
Let $\{X,H,X^*\}$ be an evolution triple with the corresponding
inclusion operator $T\in \mathcal{L}(X;H)$, as it was defined in
Definition \ref{7bdf}, together with the corresponding operator
$\widetilde{T}\in \mathcal{L}(H;X^*)$, defined as in \er{tildet},
and let $a,b,q\in\R$ be s.t. $a<b$ and $q\geq 2$. Furthermore, for
every $t\in[a,b]$ let $\Psi_t(x):X\to[0,+\infty)$ be a strictly
convex function which is G\^{a}teaux differentiable at every $x\in
X$, satisfies $\Psi_t(0)=0$ and satisfies the growth condition
\begin{equation}\label{roststSLD}
(1/C_0)\,\|x\|_X^q-C_0\leq \Psi_t(x)\leq
C_0\,\|x\|_X^q+C_0\quad\forall x\in X,\;\forall t\in[a,b]\,,
\end{equation}
and the following uniform convexity condition
\begin{equation}\label{roststghhh77889lkagagSLD}
\Big<h,D\Psi_t(x+h)-D\Psi_t(x)\Big>_{X\times X^*}\geq
\frac{1}{C_0}\Big(\big\|x\big\|^{q-2}_X+1
\Big)\cdot\|h\|_X^2\quad\forall x,h\in X\;\,\forall t\in[a,b],
\end{equation}
for some $C_0>0$.
We also assume that $\Psi_t(x)$ is Borel on the pair
of variables $(x,t)$ (see Definition
\ref{fdfjlkjjkkkkkllllkkkjjjhhhkkk}).
Next let $Z$ be a Banach space, $S:X\to Z$ be a compact operator and
for every $t\in[a,b]$ let $F_t(z):Z\to X^*$ be a function, such that
$F_t$ is strongly Borel on the pair of variables $(z,t)$ and
G\^{a}teaux differentiable at every $z\in Z$, $F_t(0)\in
L^{q^*}\big(a,b;X^*\big)$ and the derivatives of $F_t$ satisfies the
growth conditions
\begin{equation}\label{roststlambdgnewSLD}
\big\|D F_t(S\cdot x)\big\|_{\mathcal{L}(Z;X^*)}\leq g\big(\|T\cdot
x\|\big)\,\Big(\|x\|_X^{q-2}+1
\Big)\quad\forall
x\in X,\;\forall t\in[a,b]\,,
\end{equation}
for some non-decreasing function $g(s):[0+\infty)\to (0,+\infty)$
Moreover, assume that
$\Psi_t$ and $F_t$ satisfy the following positivity condition:
\begin{multline}\label{MonotonegnewhghghghghSLD}
\Psi_t(x)+\Big<x,F_t(S\cdot x)\Big>_{X\times X^*}\geq \frac{1}{\bar
C}\,\|x\|^q_X -\bar C
\|S\cdot x\|^{2}_Z-
\bar\mu(t)
\Big(\|T\cdot x\|^{2}_H+1\Big)
\quad\forall x\in X,\;\forall t\in[a,b],
\end{multline}
where $\bar C>0$ is some constants and $\bar\mu(t)\in L^1(a,b;\R)$
is a nonnegative function.
%
%
%
%
%
%
%
Furthermore, let $w_0\in H$ be such that $w_0=T\cdot u_0$ for some
$u_0\in X$, or more generally, $w_0\in H$ be such that
$\mathcal{A}_{w_0}:=\big\{u\in\mathcal{R}_{q}(a,b):\,w(a)=w_0\big\}\neq\emptyset$.
Then there exists a unique $u(t)\in\mathcal{R}_{q}(a,b)$, which
satisfies
\begin{equation}\label{uravnllllgsnachnewSLD}
\begin{cases}\frac{d v}{dt}(t)+F_t\big(S\cdot u(t)\big)+D\Psi_t\big(u(t)\big)=0\quad\text{for
a.e.}\; t\in(a,b)\,,\\ w(a)=w_0\,,\end{cases}
\end{equation}
where $w(t):=T\cdot\big(u(t)\big)$, $v(t):=I\cdot
\big(u(t)\big)=\widetilde T\cdot \big(w(t)\big)$ with $I:=\widetilde
T\circ T:\,X\to X^*$ and we assume that $w(t)$ is $H$-weakly
continuous on $[a,b]$, as it was stated in Corollary
\ref{vbnhjjmcor}.
\end{theorem}
\begin{proof}
Using \er{roststghhh77889lkagagSLD} and \er{roststlambdgnewSLD}, we
obtain:
\begin{multline}\label{MonotonegnewagagSLD}
\bigg<h,\Big\{D\Psi_t(x+h)-D\Psi_t(x)\Big\}+ D F_t(S\cdot
x)\cdot (S\cdot h)\bigg>_{X\times X^*}\geq\\
\frac{1}{C_0}\Big(\big\|x\big\|^{q-2}_X+1
\Big)\cdot\|h\|_X^2 - g\big(\|T\cdot
x\|_H\big)\cdot\Big(\big\|x\big\|^{q-2}_X+1
\Big)\cdot\big\|
h\big\|_X\cdot\big\|S\cdot h\big\|_Z\quad\forall x,h\in X\;\forall
t\in[a,b]\,,
\end{multline}
On the other hand since $S$ is a compact operator, by Lemma
\ref{Aplem1} from the Appendix, there exists a nondecreasing
function $\hat g(s):[0,+\infty)\to(0,+\infty)$ such that
$$\big\|S\cdot h\big\|_Z\leq \frac{1}{2C_0g\big(\|T\cdot x\|_H\big)}\|h\|_X+\hat g\big(\|T\cdot x\|_H\big)\big\|T\cdot h\big\|_H
\quad\forall x,h\in X.$$ Plugging it into \er{MonotonegnewagagSLD}
we deduce
\begin{multline}\label{MonotonegnewagaghjgjSLD}
\bigg<h,\Big\{D\Psi_t(x+h)-D\Psi_t(x)\Big\}+ D F_t(S\cdot x)\cdot
(S\cdot h)\bigg>_{X\times
X^*}\geq\frac{1}{2C_0}\Big(\big\|x\big\|^{q-2}_X+1
\Big)\cdot\|h\|_X^2\\
- g\big(\|T\cdot x\|_H\big)\hat g\big(\|T\cdot
x\|_H\big)\cdot\Big(\big\|x\big\|^{q-2}_X+1
\Big)\cdot\|h\|_X\cdot\big\|T\cdot h\big\|_H\\ \geq -C_0
g^2\big(\|T\cdot x\|_H\big)\hat g^2\big(\|T\cdot
x\|_H\big)\cdot\Big(\big\|x\big\|^{q-2}_X+1
\Big)\cdot\big\|T\cdot
h\big\|^2_H \quad\forall x,h\in X\;\forall t\in[a,b].
\end{multline}
%
Similarly, by Lemma \ref{Aplem1}, there exists a constant $K>0$ such
that
$$\big\|S\cdot x\big\|^2_Z\leq \frac{1}{2\bar C^2}\,\|x\|^2_X+K\big\|T\cdot x\big\|^2_H
\quad\forall x\in X.$$ Plugging it into
\er{MonotonegnewhghghghghSLD} we obtain
\begin{multline}\label{MonotonegnewhghghghghSLDhhh}
\Psi_t(x)+\Big<x,\Lambda_t(S\cdot x)\Big>_{X\times X^*}\geq
\frac{1}{2\bar C}\Big(2\|x\|^q_X -\|x\|^2_X\Big)
-
\big(\bar\mu(t)+\bar C K
\big)
\Big(\|T\cdot x\|^{2}_H+1\Big)\\
\geq\frac{1}{2\bar C}\,\|x\|^q_X
-
\big(\bar\mu(t)+\tilde K\big)
\Big(\|T\cdot x\|^{2}_H+1\Big)
\quad\forall x\in X,\;\forall t\in[a,b],
\end{multline}
where $\tilde K>0$ is a constant. Thus the result follows by
applying Proposition \ref{premainnew} with
\er{MonotonegnewagaghjgjSLD} substituting \er{Monotone} and
\er{MonotonegnewhghghghghSLDhhh} substituting
\er{Monotonegnewhghghghgh}.
\end{proof}

\section{An example of the application to PDE}\label{dkgfkghfhkljl}
\subsection{Notations in the present section}
For a $p\times q$ matrix $A$ with $ij$-th entry $a_{ij}$ we denote
by $|A|=\bigl(\Sigma_{i=1}^{p}\Sigma_{j=1}^{q}a_{ij}^2\bigr)^{1/2}$
the
Frobenius norm of $A$.\\
For two matrices $A,B\in\R^{p\times q}$  with $ij$-th entries
$a_{ij}$ and $b_{ij}$ respectively, we write
$A:B\,:=\,\sum\limits_{i=1}^{p}\sum\limits_{j=1}^{q}a_{ij}b_{ij}$.\\
Given a vector valued function
$f(x)=\big(f_1(x),\ldots,f_k(x)\big):\O\to\R^k$ ($\O\subset\R^N$) we
denote by $\nabla_x f$ the $k\times N$ matrix with
$ij$-th entry $\frac{\partial f_i}{\partial x_j}$.\\
For a matrix valued function $F(x):=\{F_{ij}(x)\}:\R^N\to\R^{k\times
N}$ we denote by $div\,F$ the $\R^k$-valued vector field defined by
$div\,F:=(l_1,\ldots,l_k)$ where
$l_i=\sum\limits_{j=1}^{N}\frac{\partial F_{ij}}{\partial x_j}$.
\subsection{A parabolic system in a divergent form}
Let $\Phi(A,x,t):\R_A^{k\times N}\times\R_x^N\times\R_t\to\R$ be a
nonnegative measurable function. Moreover assume that $\Phi(A,x,t)$
is $C^1$ as a function of the first argument $A$ when $(x,t)$ are
fixed, which satisfies $\Phi(0,x,t)=0$ and it is uniformly convex by
the first argument $A$
i.e. there exists a constant $\tilde C>0$ such that,
$$\bigg(D_A\Phi\big(A_1,x,t\big)-D_A\Phi\big(A_2,x,t\big)\bigg):\big(A_1-A_2\big)\geq \tilde C\big|A_1-A_2\big|^2$$
for every $A_1,A_2\in\R^{k\times N}$, $x\in\R^N$ and $t\in\R$, where
$$D_A\Phi(A,x,t):=\bigg\{\frac{\partial \Phi}{\partial A_{ij}}(A,x,t)\bigg\}_{1\leq i\leq k,1\leq j\leq N}\in\R^{k\times N}\,.$$
Moreover, we assume that $\Phi$ satisfies the following growth
condition
\begin{multline}\label{roststglkjjjaplneabst} (1/C)|A|^q-|g_0(x)|\leq
\Phi(A,x,t)\leq C|A|^q+|g_0(x)| \quad\quad\forall A\in\R^{k\times
N},\;\forall x\in\R^N,\;\forall t\in\R\,,
\end{multline}
where $C>0$ is some constant, $g_0(x)\in L^1(\R^N,\R)$
and $q\in [2,+\infty)$.
%
%
%
%
%
%
%
%
%
Finally, let $\Xi(B,x,t):\R_B^k\times\R_x^N\times\R_t\to\R^{k\times
N}$ and $\Theta(B,x,t):\R_B^k\times\R_x^N\times\R_t\to\R^{k}$ be two
measurable functions. Moreover, assume that $\Xi(B,x,t)$ and
$\Theta(B,x,t)$ are $C^1$ as a functions of the first argument $B$
when $(x,t)$ are fixed. We also assume that $\Xi(B,x,t)$ and
$\Theta(B,x,t)$ are globally Lipschitz by the first argument $B$ and
satisfy
\begin{equation}\label{jeheefkeplplgryrtu}
\Xi(0,x,t)\in L^{q^*}\big(\R;L^{2}(\R^N,\R^{k\times
N})\big),\;\;\Theta(0,x,t)\in
L^{q^*}\big(\R;L^{2}(\R^N,\R^k)\big)\,.
\end{equation}
\begin{proposition}\label{divform}
Let $\Phi,
\Xi,\Theta$ be as above and let $\Omega\subset\R^N$
be a bounded open set, $2\leq q<+\infty$ and $T_0>0$.
Then for every $w_0(x)\in
W^{1,q}_0(\O,\R^k)$ there exists unique $u(x,t)\in
L^q\big(0,T_0;W_0^{1,q}(\O,\R^k)\big)$, such that $u(x,t)\in
L^\infty\big(0,T_0;L^2(\O,\R^k)\big)\cap W^{1,q^*}\big(0,T_0;
W^{-1,q^*}(\O,\R^k)\big)$, where $q^*:=q/(q-1)$, $u(x,t)$ is
$L^2(\O,\R^k)$-weakly continuous on $[0,T_0]$, $u(x,0)=w_0(x)$ and
$u(x,t)$ is a solution to
\begin{multline}\label{uravnllllgsnachlklimjjjaplneabst}
\frac{\partial u}{\partial t}(x,t)=\Theta\big(u(x,t),x,t\big)+div_x
\Big(\Xi\big(u(x,t),x,t\big)\Big)+
div_x \Big(D_A\Phi\big(\nabla_x
u(x,t),x,t\big)\Big)\quad\text{in}\;\;\O\times(0,T_0)\,,
\end{multline}
\end{proposition}
\begin{proof}
Let $X:=W_0^{1,q}(\O,\R^k)$
(a separable reflexive Banach space), $H:=L^2(\O,\R^k)$ (a Hilbert
space) and $T\in \mathcal{L}(X;H)$ be a usual embedding operator
from $W_0^{1,q}(\O,\R^k)$ into $L^2(\O,\R^k)$. Then $T$ is an
injective inclusion with dense image. Furthermore,
$X^*=W^{-1,q^*}(\O,\R^k)$ where $q^*=q/(q-1)$ and the corresponding
operator $\widetilde{T}\in \mathcal{L}(H;X^*)$, defined as in
\er{tildet}, is a usual inclusion of $L^2(\O,\R^k)$ into
$W^{-1,q^*}(\O,\R^k)$.
Then $\{X,H,X^*\}$ is an evolution triple with the corresponding
inclusion operators $T\in \mathcal{L}(X;H)$ and $\widetilde{T}\in
\mathcal{L}(H;X^*)$, as it was defined in Definition \ref{7bdf}.
Moreover, by the Theorem about the compact embedding in Sobolev
Spaces it is well known that
$T$ is a compact operator.

Next, for every $t\in[0,T_0]$ let $\Psi_t(x):X\to[0,+\infty)$ be
defined by
\begin{equation*}
\Psi_t(u):=\int_\O\Phi\big(\nabla u(x),x,t\big)dx
\quad\forall u\in
W^{1,q}_0(\O,\R^k)\equiv X\,.
\end{equation*}
Then
$\Psi_t(x)$ is G\^{a}teaux differentiable at every $x\in X$, satisfy
$\Psi_t(0)=0$ and by \er{roststglkjjjaplneabst} it satisfies the
growth condition
\begin{equation*}
(1/C)\,\|x\|_X^q-C\leq \Psi_t(x)\leq C\,\|x\|_X^q+C\quad\forall x\in
X,\;\forall t\in[0,T]\,,
\end{equation*}
%
%
%
%
%
%
%
%
Finally, for every $t\in[0,T_0]$ let $\Lambda_t(w):H\to X^*$ be
defined by
\begin{multline}\label{jhfjggkjkjhhkhkloopll}
\Big<\delta,\Lambda_t(w)\Big>_{X\times X^*}:=\int_\O
\bigg\{\Xi\big(w(x),x,t\big):\nabla\delta(x)-
\Theta\big(w(x),x,t\big)
\cdot\delta(x)\bigg\}dx\\
\forall w\in L^2(\O,\R^k)\equiv H,\; \forall\delta\in
W^{1,q}(\O,\R^k)\equiv X\,.
\end{multline}
Then
$\Lambda_t(w)$ is G\^{a}teaux differentiable at every $w\in H$, and,
since $\Xi$ and $\Theta$ are Lipshitz functions, the derivative of
$\Lambda_t(w)$ satisfy the Lipschitz condition
\begin{equation}
\label{roststlambdglkFFjjjapljjjkhkl}
\|D\Lambda_t(w)\|_{\mathcal{L}(H;X^*)}\leq C\quad\forall w\in
H,\;\forall t\in[0,T_0]\,,
\end{equation}
for some $C>0$. Thus all the conditions of
Corollary \ref{CorCorCor} are satisfied and therefore, by its
statement, there exists unique $u(t)\in L^q\big(0,T_0;X\big)$, such
that $(\widetilde T\circ T)\cdot u(t)\in
W^{1,q^*}\big(0,T_0;X^*\big)$, $T\cdot u(0)=w(0)$ and
$\frac{d}{dt}(\widetilde T\circ T)\cdot u(t)+\Lambda_t\big(T\cdot
u(t)\big)+D\Psi_t\big(u(t)\big)=0$. This completes the proof.
\end{proof}

\appendix
\section{Appendix}
\begin{lemma}\label{Aplem1}
Let $X$, $Y$ and $Z$ be three Banach spaces, such that $X$
is a reflexive space. Furthermore, let $T\in \mathcal{L}(X;Y)$ and
$S\in \mathcal{L}(X;Z)$ be bounded linear operators. Moreover assume
that $S$ is an injective inclusion (i.e. it satisfies $\ker
S=\{0\}$) and $T$
is a compact operator. Then for each $\e>0$ there exists some
constant $c_\e>0$ depending on $\e$ (and on the spaces $X$, $Y$, $Z$
and on the operators $T$, $S$) such that
\begin{equation}\label{Eqetaenerap}
\big\|T\cdot h\big\|_Y\leq\e\big\|h\big\|_X+c_\e\big\|S\cdot
h\big\|_Z\quad\forall h\in X\,.
\end{equation}
\end{lemma}
\begin{proof}
Assume by contradiction that for some $\e>0$ such a constant $c_e$
doesn't exist. Then for every natural number $n\in\mathbb{N}$ there
exists $h_n\in X$ such that
\begin{equation}\label{Eqetaenercontrap}
\big\|T\cdot h_n\big\|_Y>\e\big\|h_n\big\|_X+n\big\|S\cdot
h_n\big\|_Z\,.
\end{equation}
We consider the sequence $\{\xi_n\}\subset X$ defined by the
normalization
\begin{equation}\label{Eqnormalkkkap}
\xi_n:=\frac{h_n}{\|h_n\|_X}\,,
\end{equation}
which satisfy $\|\xi_n\|_X=1$ and by \er{Eqetaenercontrap},
\begin{equation}\label{Eqetaenercontrnormap}
\big\|T\cdot
\xi_n\big\|_Y>\e+n\big\|S\cdot\xi_n\big\|_Z\quad\quad\forall
n\in\mathbb{N}\,.
\end{equation}
However, since $\|\xi_n\|_X=1$, we have $\|T\cdot
\xi_n\|_Y\leq\|T\|_{\mathcal{L}(X;Y)}$. So by
\er{Eqetaenercontrnormap} we deduce
$$\big\|S\cdot\xi_n\big\|_Z<\frac{1}{n}\|T\|_{\mathcal{L}(X;Y)}\quad\quad\forall n\in\mathbb{N}\,.$$
In particular
\begin{equation}\label{Eqetaenercontrnormsledap}
S\cdot\xi_n\to 0\quad\text{as}\;\;n\to+\infty\quad\text{strongly
in}\;\;Z\,.
\end{equation}
On the other hand since $\|\xi_n\|_X=1$ and since $X$ is a reflexive
space, up to a subsequence we must have $\xi_n\rightharpoonup \xi$
weakly in $X$. Thus $S\cdot\xi_n\rightharpoonup S\cdot \xi$ weakly
in $Z$ and then by \er{Eqetaenercontrnormsledap} we have $S\cdot
\xi=0$. So since $S$ is an injective operator we deduce that $\xi=0$
and thus $\xi_n\rightharpoonup 0$ weakly in $X$. Therefore, since
$T$ is a compact operator we have
\begin{equation}\label{Eqetaenercontrnormsledkkllnullap}
T\cdot \xi_n\to 0\quad\text{strongly in}\;\;Y\,.
\end{equation}
However, returning to \er{Eqetaenercontrnormap}, in particular we
deduce
\begin{equation}\label{Eqetaenercontrnormprtyyuap}
\big\|T\cdot \xi_n\big\|_Y>\e\quad\quad\forall n\in\mathbb{N}\,,
\end{equation}
which contradicts with \er{Eqetaenercontrnormsledkkllnullap}. So we
proved \er{Eqetaenerap}.
\end{proof}
\begin{lemma}\label{hilbcomban}
Let $X$ be a separable Banach space. Then there exists a separable
Hilbert space $Y$ and a bounded linear inclusion operator $S\in
\mathcal{L}(Y;X)$ such that $S$ is injective (i.e. $\ker S=\{0\}$),
the image of $S$ is dense in $X$ and moreover, $S$ is a compact
operator.
\end{lemma}
\begin{proof}
If $X$ is finite dimensional then $X$ is isomorphic to $\R^k$ for
some $k$ and we are done. Otherwise since $X$ is a separable Banach
space there exists a countable sequence
$\{x_n\}_{n=1}^{+\infty}\subset X$ such that $\|x_n\|_X=1$ for every
$n$, every finite subsystem of the system $\{x_n\}_{n=1}^{+\infty}$
is linearly independent and the span of $\{x_n\}_{n=1}^{+\infty}$ is
dense in $X$. Set $\bar Y:=l^2$ where $l^2$ is a standard separable
Hilbert space defined by
\begin{equation}\label{jgggjgjgjgjgggkggggkuoujkl}
l^2:=\Big\{\bar
y=\alpha_n:\mathbb{N}\to\R:\;\sum_{n=1}^{+\infty}\alpha_n^2<+\infty\Big\}
\end{equation}
with the scalar product
\begin{equation}\label{jgggjgjgjgjgggkggggkuoujklllkk}
\big<\bar y_1,\bar y_2\big>_{\bar Y\times \bar
Y}=\sum_{n=1}^{+\infty}\alpha_n\beta_n\quad\quad\text{for}\;\; \bar
y_1=\{\alpha_n\},\;\bar y_2=\{\beta_n\}\,.
\end{equation}
Next we prove that for every $\bar y=\{\alpha_n\}$ there exists a
limit in $X$,
\begin{equation}\label{shhs8888888ghgsg}
x=\lim_{N\to +\infty}\sum_{n=1}^{N}\frac{\alpha_n}{n}\;x_n\,.
\end{equation}
Indeed since $\|x_n\|_X=1$ for every $N\in\mathbb{N}$ and
$m\in\mathbb{N}$ we have
$$\bigg\|\sum_{n=N}^{N+m}\frac{\alpha_n}{n}\;x_n\bigg\|^2_X\leq\bigg(\sum_{n=N}^{N+m}\alpha_n^2\bigg)\cdot
\bigg(\sum_{n=N}^{N+m}\frac{1}{n^2}\bigg)\leq\bigg(\sum_{n=N}^{+\infty}\alpha_n^2\bigg)\cdot
\bigg(\sum_{n=N}^{+\infty}\frac{1}{n^2}\bigg)\to
0\quad\text{as}\;\;N\to+\infty\,.$$ Thus since $X$ is a Banach space
the limit in \er{shhs8888888ghgsg} exists. Then define the linear
operator $\bar S:\bar Y\to X$ for every $\bar y=\{\alpha_n\}\in\bar
Y$ by
\begin{equation}\label{shhs8888888ghgsggh}
\bar S\cdot \bar y=\lim_{N\to
+\infty}\sum_{n=1}^{N}\frac{\alpha_n}{n}\;x_n\,.
\end{equation}
As before,
\begin{equation}\label{shhs8888888ghgsgghhuhkukl}
\big\|\bar S\cdot \bar y\big\|_X^2=\bigg\|\lim_{N\to
+\infty}\sum_{n=1}^{N}\frac{\alpha_n}{n}\;x_n\bigg\|_X^2\leq\bigg(\sum_{n=1}^{+\infty}\alpha_n^2\bigg)\cdot
\bigg(\sum_{n=1}^{+\infty}\frac{1}{n^2}\bigg)=
\bigg(\sum_{n=1}^{+\infty}\frac{1}{n^2}\bigg)\cdot\|\bar y\|_Y^2\,.
\end{equation}
Thus $\bar S$ is a bounded operator i.e. $\bar S\in
\mathcal{L}(Y;X)$. Next clearly for every finite linear combination
$z=\sum_{n=1}^{N}c_n x_n$ (where $c_n\in\R$) there exists $\bar
y\in\bar Y$ such that $\bar S\cdot\bar y=z$. So the image of $\bar
S$ is dense in $X$. We will prove now that $\bar S$ is a compact
operator. Indeed let $\bar
y_n:=\{\alpha_j^{(n)}\}_{j=1}^{+\infty}\in \bar Y$ be such that
$\bar y_n\rightharpoonup 0$
weakly in $\bar Y$. This means $\lim_{n\to+\infty}\alpha_j^{(n)}=0$
for every $j$ and
$\sum_{j=1}^{+\infty}\big(\alpha_j^{(n)}\big)^2\leq C$ for some
constant $C>0$. Fix some $\e>0$. Then since for every $n$
$$\bigg\|\lim_{m\to +\infty}\sum_{j=N}^{N+m}\frac{\alpha^{(n)}_j}{j}\;x_j\bigg\|^2_X\leq\bigg(\sum_{j=1}^{+\infty}\big(\alpha^{(n)}_j\big)^2\bigg)\cdot
\bigg(\sum_{j=N}^{+\infty}\frac{1}{j^2}\bigg)\leq
C\bigg(\sum_{j=N}^{+\infty}\frac{1}{j^2}\bigg)\to
0\quad\text{as}\;\;N\to+\infty\,,$$ there exists $N_0$ such that
\begin{equation}\label{epsepseps}
\bigg\|\lim_{m\to
+\infty}\sum_{j=N_0}^{N_0+m}\frac{\alpha^{(n)}_j}{j}\;x_j\bigg\|_X<\frac{\e}{2}
\quad\quad\forall
n\in\mathbb{N}\,.
\end{equation}
On the other hand, since $\lim_{n\to+\infty}\alpha_j^{(n)}=0$,
there exist $n_0$ such that $|\alpha_j^{(n)}
|<\e/(2N_0)$
for every $n>n_0$ and $1\leq j\leq N_0$ and thus
\begin{equation}\label{epsepsepsy7789ihk88899999999}
\bigg\|\sum_{j=1}^{N_0-1}\frac{\alpha^{(n)}_j}{j}\;x_j
\bigg\|_X<\frac{\e}{2}\quad\quad\forall n>n_0\,.
\end{equation}
Plugging \er{epsepsepsy7789ihk88899999999} into \er{epsepseps} we
deduce that
$$\big\|\bar S\cdot\bar y_n
\big\|_{X}<\e\quad\quad\forall n>n_0\,.$$ Therefore $S\cdot\bar
y_n\to
0$ strongly in $X$ and so $\bar S$ is a compact operator. Finally,
set $Z:=\{\bar y\in\bar Y:\;\bar S\cdot\bar y=0\}$. Then $Z$ is a
close subspace of $\bar Y$. Next define $Y$ to be the orthogonally
dual to $Z$ space
$$Y:=\Big\{\bar y\in\bar Y:\;\big<\bar y,z\big>_{\bar Y\times\bar
Y}=0\quad\forall z\in Z\Big\}\,.$$ Then $Y$ is a close subspace of
$\bar Y$. Therefore $Y$ is a separable Hilbert space by itself.
Define $S\in\mathcal{L}(Y;X)$ by $S:=\bar S\mid_Y$. Then clearly $S$
is injective i.e. $\ker S=\{0\}$. Moreover, if $x=\bar S\cdot \bar
y$ where $\bar y\in\bar Y$ then we can represent $\bar y=z+y$ where
$z\in Z$ and $y\in Y$, and since $\bar S\cdot z=0$ we have $x=S\cdot
y$. Therefore since the image of $\bar S$ is dense in $X$ we deduce
that the image of $S$ is also dense in $X$. Finally, $S$ is a
compact operator. This completes the proof.
\end{proof}
\begin{lemma}\label{hilbcombanback}
Let $X$ be a separable Banach space. Then there exists a separable
Hilbert space $Y$ and a bounded linear inclusion operator $S\in
\mathcal{L}(X;Y)$ such that $S$ is injective (i.e. $\ker S=\{0\}$),
the image of $S$ is dense in $Y$ and moreover, $S$ is a compact
operator.
\end{lemma}
\begin{proof}
By the Lindenstrauss's Theorem (see \cite{Diestel}) every separable
Banach space is continuously embedded in $c_0$ where $c_0$ is a
Banach space of real sequences which tend to $0$, i.e. it is defined
by
\begin{equation}\label{c0def9999llll}
c_0:=\big\{\alpha_n:\mathbb{N}\to\R\;:\;\lim_{n\to
+\infty}a_n=0\big\}\,,\quad\quad
\big\|\alpha_n\big\|_{c_0}:=\sup\limits_{n\in\mathbb{N}}\big|a_n\big|\,.
\end{equation}
So there exists an embedding operator $P\in \mathcal{L}(X;c_0)$
which is an injective inclusion (i.e. $\ker P=\{0\}$). Next define
$Q\in \mathcal{L}(c_0, l^2)$, where $l^2$ is the separable Hilbert
space defined in \er{jgggjgjgjgjgggkggggkuoujkl}, by the formula
\begin{equation}\label{c0def9999lllldjojod}
Q\cdot
\Big(\{\alpha_n\}_{n=1}^{+\infty}\Big)=\Big\{\frac{\alpha_n}{n}\Big\}_{n=1}^{+\infty}\in
l^2\quad\quad\forall \{\alpha_n\}_{n=1}^{+\infty}\in c_0\,.
\end{equation}
Then clearly $Q\in \mathcal{L}(c_0, l^2)$ is an injective inclusion.
Moreover we will prove now that $Q$ is a compact operator. Indeed
for every $j\in\mathbb{N}$ let
$h_j:=\{\alpha_{n}^{(j)}\}_{n=1}^{+\infty}\subset c_0$ be such that
$h_j\rightharpoonup 0$
weakly in $c_0$ as $j\to +\infty$. Thus in particular we have
\begin{equation}\label{c0def9999lllldjojodkoil}
\begin{cases}\lim\limits_{j\to
+\infty}\alpha_{n}^{(j)}=0
\quad\quad\forall
n\in\mathbb{N}\\
|\alpha_{n}^{(j)}|\leq C\quad\quad\forall n,j\in\mathbb{N}\,,
\end{cases}
\end{equation}
for some constant $C>0$. Then for every $j,m\in\mathbb{N}$ we have
\begin{multline}\label{c0def9999lllldjojodojjkjpkklklkkk}
\big\|Q\cdot h_j
\big\|_{l^2}=\sum\limits_{n=1}^{+\infty}\bigg(\frac{\alpha_{n}^{(j)}
}{n}\bigg)^2
=\sum\limits_{n=1}^{m}\bigg(\frac{\alpha_{n}^{(j)}
}{n}\bigg)^2+\sum\limits_{n=m}^{+\infty}\bigg(\frac{\alpha_{n}^{(j)}
}{n}\bigg)^2
\leq\sum\limits_{n=1}^{m}\bigg(\frac{\alpha_{n}^{(j)}
}{n}\bigg)^2+4C^2\sum\limits_{n=m}^{+\infty}\frac{1}{n^2}\,.
\end{multline}
Thus, since $\sum_{n=1}^{+\infty}\frac{1}{n^2}<+\infty$, for every
$\e>0$ there exists $m=m_\e\in\mathbb{N}$ such that
$4C^2\sum_{n=m}^{+\infty}\frac{1}{n^2}<\e$. Therefore, by
\er{c0def9999lllldjojodojjkjpkklklkkk} we obtain
\begin{equation}\label{c0def9999lllldjojodojjkjpkklkl}
\big\|Q\cdot h_j
\big\|_{l^2}\leq
\sum\limits_{n=1}^{m}\bigg(\frac{\alpha_{n}^{(j)}
}{n}\bigg)^2+\e\,.
\end{equation}
Then letting $j\to +\infty$ in \er{c0def9999lllldjojodojjkjpkklkl}
and using \er{c0def9999lllldjojodkoil} we deduce
$$\limsup\limits_{j\to +\infty}\big\|Q\cdot h_j
\big\|_{l^2}\leq\e\,,$$ and since $\e>0$ was arbitrary we finally
infer that $Q\cdot h_j\to 0$
strongly in $l^2$. So we proved that $Q$ is a compact operator. Next
define $S\in \mathcal{L}(X;l^2)$ by $S:=Q\circ P$, where $P\in
\mathcal{L}(X;c_0)$ is an injective embedding. Thus since $P$ and
$Q$ are injective, we obtain that $S$ is also an injective embedding
of $X$ to $l^2$. Moreover since $Q$ is a compact operator we obtain
that $S\in \mathcal{L}(X;l^2)$ is also a compact operator. Finally,
let $Y$ be the closure of the image of $S$ in $l^2$. Then $Y$ is a
subspace in $l^2$ and so the separable Hilbert space by itself.
Moreover $S\in(X;Y)$ is an injective compact inclusion of $X$ to $Y$
with dense in $Y$ image.
\end{proof}

\end{document}